\documentclass[11pt]{amsart}
\usepackage{graphicx,amsfonts,amsmath,amsthm,amssymb,amscd,appendix,tikz,pgfplots,xcolor}
\usepackage[margin=1in]{geometry}
\usetikzlibrary{matrix,arrows,external}
\makeatletter
\def\@tocline#1#2#3#4#5#6#7{\relax
  \ifnum #1>\c@tocdepth 
  \else
    \par \addpenalty\@secpenalty\addvspace{#2}%
    \begingroup \hyphenpenalty\@M
    \@ifempty{#4}{%
      \@tempdima\csname r@tocindent\number#1\endcsname\relax
    }{%
      \@tempdima#4\relax
    }%
    \parindent\z@ \leftskip#3\relax \advance\leftskip\@tempdima\relax
    \rightskip\@pnumwidth plus4em \parfillskip-\@pnumwidth
    #5\leavevmode\hskip-\@tempdima
      \ifcase #1
       \or\or \hskip 1em \or \hskip 2em \else \hskip 3em \fi%
      #6\nobreak\relax
    \dotfill\hbox to\@pnumwidth{\@tocpagenum{#7}}\par
    \nobreak
    \endgroup
  \fi}
\makeatother

\makeatletter
\newtheoremstyle{indented}
  {10pt}
  {10pt}
  {\setlength{\leftskip}{2.5em}\setlength{\rightskip}{2.5em}\addtolength{\@totalleftmargin}{2.5em}}
  {-2.5em}
  {\scshape}
  {.}
  {\newline}
  {}
\makeatother
\theoremstyle{indented}
\newtheorem{theorem}{Theorem}[subsection]
\newtheorem{lemma}[theorem]{Lemma}
\newtheorem{proposition}[theorem]{Proposition}
\newtheorem{corollary}[theorem]{Corollary}

\newtheorem{definition}[theorem]{Definition}

\newenvironment{example}[1][Example]{\begin{trivlist}
\item[\hskip \labelsep {\bfseries #1}]}{\end{trivlist}}
\newenvironment{remark}[1][Remark]{\begin{trivlist}
\item[\hskip \labelsep {\bfseries #1}]}{\end{trivlist}}
\newenvironment{observation}[1][Observation]{\begin{trivlist}
\item[\hskip \labelsep {\bfseries #1}]}{\end{trivlist}}
\newenvironment{note}[1][Note]{\begin{trivlist}
\item[\hskip \labelsep {\bfseries #1}]}{\end{trivlist}}

\newenvironment{warning}[1][Warning]{\begin{trivlist}
\item[\hskip \labelsep {\bfseries #1}]}{\end{trivlist}}
\newenvironment{twocents}[1][Two Cents' Worth]{\begin{trivlist}
\item[\hskip \labelsep {\bfseries #1}]}{\end{trivlist}}

\newenvironment{notation}[1][Notation]{\begin{trivlist}
\item[\hskip \labelsep {\bfseries #1}]}{\end{trivlist}}
\newenvironment{idea}[1][Idea]{\begin{trivlist}
\item[\hskip \labelsep {\bfseries #1}]}{\end{trivlist}}

\begin{document} 
\begin{titlepage}
\title {The K-theory of filtered deformations of graded polynomial algebras.}
\author {David Wayne - david.wayne@colorado.edu}

\begin{abstract}
Recent discoveries make it possible to compute the K-theory of certain rings from their cyclic homology and certain versions of their cdh-cohomology.  We extend the work of G. Corti\~nas et al. who calculated the K-theory of, in addition to many other varieties, cones over smooth varieties, or equivalently the K-theory of homogeneous polynomial rings.  We focus on specific examples of polynomial rings, which happen to be filtered deformations of homogeneous polynomial rings.  Along the way, as a secondary result, we will develop a method for computing the periodic cyclic homology of a singular variety as well as the negative cyclic homology when the cyclic homology of that variety is known.  Finally, we will apply these methods to extend the results of Michler who computed the cyclic homology of hypersurfaces with isolated singularities.  
\end{abstract}

\ \vspace{2 in}

\maketitle

\end{titlepage}

\newpage
\tableofcontents 
\setcounter{section}{-1}

\newpage
\section{Introduction}

Historically speaking, computing algebraic K-theory groups of a ring has ranged from very difficult to ``hopefully it's not impossible.''  Recent methods developed by various rock stars in the field (G. Corti\~nas, C. Haesemeyer, M. Schlichting, C. A. Weibel among others) have allowed us to compute some K-groups by using the cdh-topology and cdh-cohomology.  Our goal is to continue this line of research by computing, or partially computing the K-theory of certain rings in terms of their cyclic homology.  
   
Unfortunately for K-theory, unlike cyclic homology and it's variants, in the current state of the art we are still unable to compute K-theory for smooth affine schemes (regular Noetherian rings).  With this in mind, following previous authors, this paper will attempt to compute the fiber of the higher K-groups, $K_n(R)\rightarrow KH_n(R)$, called $\tilde{K}_n(R)$.  In some ways the Homotopy K-groups, $KH_n(R)$, are all the nicest parts of the original K-groups, in that they have all of the properties we would want (like excision, homotopy invariance, etc.). Furthermore, $K_n(R)=KH_n(R)$ whenever $R$ is a regular Noetherian ring (smooth affine schemes).  By removing $KH_n(R)$ from $K_n(R)$ we are left with groups that in a particular way measure singularities in a scheme. 

On smooth schemes, using these methods from Corti\~nas et al. will only yield trivial results.  So, at least when these techniques are involved we obtain the most information of the K-theory by looking at singular schemes.  The motivation for this is, in a sense, for singular schemes the worst parts of the K-theory, $\tilde{K}_n(R)$, are the same as the worst parts of the cyclic homology.  With this idea more formally presented later, we will arrive at the following exact sequence.  
\[...\rightarrow\tilde{K}_n(X)\rightarrow HC_n(X/\mathbb{Q})\rightarrow\mathbb{H}^{-n}_{cdh}(X,\textbf{HC}(-/\mathbb{Q}))\rightarrow\tilde{K}_{n-1}(X)\rightarrow...\] 

So, computing the groups $\tilde{K}_n(X)$ is as simple as computing the cyclic homology groups and the cdh-cohomology groups. It turns out that computing these other groups is not that simple. In general in order to compute the cyclic homology, $HC_n(X/\mathbb{Q})$, and the cdh-fibrant version of cyclic homology, $\mathbb{H}^{-n}_{cdh}(X,\textbf{HC}(-/\mathbb{Q}))$, we have to rely on complicated techniques and/or powerful tools.  

For cyclic homology, we will have to rely on other authors who have computed $HC_n(X/\mathbb{Q})$ for specific singular schemes $X$. The problem is, the cyclic homology is not known for every scheme.  It is known for smooth affine schemes, and certain smooth projective schemes, but as we mentioned before, this cannot really be used in our framework.  In fact, in the case where we are not dealing with nice smooth affine or projective schemes, even if the scheme is smooth, the general theory can still be very delicate and the cyclic homology can be difficult to compute.

For cdh-fibrant cyclic homology, we will need to compute the resolution of singularities of $X$ and then we can write $\mathbb{H}^{-n}_{cdh}(X,\textbf{HC}(-/\mathbb{Q}))$ in terms of the cyclic homology of some smooth schemes.  The problem is, while a strong resolution theorem for singular varieties (or excellent/quasi-excellent schemes) over characteristic zero has been proved by H. Hironaka (and refined by others), this is not a constructive theorem and does not yield explicit resolutions.  By this we mean, given a singular variety $X$, an explicit resolution of $X$ would yield a (sequence of) center(s) $Z$ it's resolution $Y$, and the exceptional fiber $E$.  There are computer algorithms that given such a scheme would yield an explicit resolution, but we are looking for classes of rings, not specific examples.   Making the problem worse is that these computer resolution programs, while they always find some resolution, have a hard time finding a ``nice" resolution  for anything but the most basic singular varieties.  

Compounding the difficulties, when you eventually find a resolution, the groups \\ $\mathbb{H}^{-n}_{cdh}(X,\textbf{HC}(-/\mathbb{Q}))$ are still written in terms of the cyclic homology of some smooth schemes, which as we stated can still be difficult to compute (unless they are affine or some specific type of a projective scheme).   

So, in order to get a complete calculation using these methods we need the previous groups to be computable.  Based on our discussion above, the following Venn diagram summarizes our needs.  
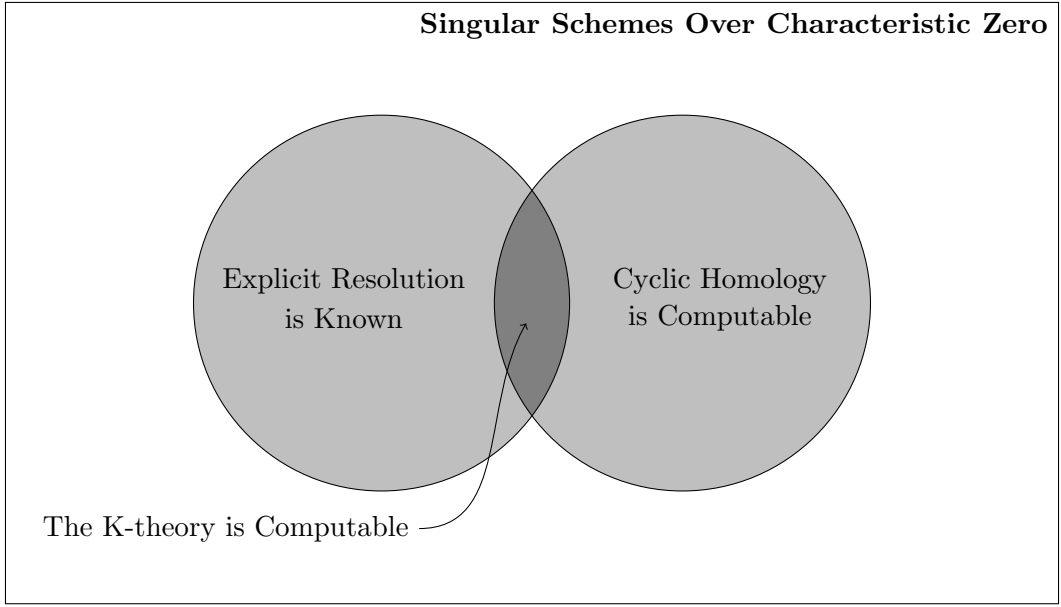
\begin{center}
\begin{tikzpicture}
\draw (-7,-4) rectangle (7,4) node[below left]{\textbf{Singular Schemes Over Characteristic Zero}};
\fill[lightgray] (-2,0) circle (2.5cm);
\fill[lightgray] (2,0) circle (2.5cm);
\begin{scope} 
\clip (-2,0) circle (2.5cm);
\fill[gray] (2,0) circle (2.5cm);
\end{scope}
\draw (-2,0) circle (2.5 cm);
\draw (-2.5,.3) node {Explicit Resolution};
\draw (-2.5,-.2) node {is Known};
\draw (2,0) circle (2.5 cm);
\draw (2.5,.3) node {Cyclic Homology};
\draw (2.5,-.2) node {is Computable};
\node[anchor=east] at (-1.5,-3) (description) {The K-theory is Computable};
\node[anchor=north] at (0,0) (text){};
\draw (description) edge[out=0,in=240,->] (text){};
\end{tikzpicture}
\end{center}

There have already been several papers that have found classes of varieties in the intersection of these two requirements.  Some of them include, curves, toric varieties, and cones.  Each of these papers is able to give a complete or partially complete calculation because each of the required computations is known.  

In our case, we will present some new results and give partial calculations.  While some of our results are deficient in one or the other of the two requirements, they are still results in that they can be applied by others in the future.  In particular, our main result lacks the computability of cyclic homology in positive degrees.  Despite the fact that we are missing this piece, in some sense, it still feels like a complete result.  For the other case, if the cyclic homology is known, but an explicit resolution is not known, then the result is ill formed.  The reason for this is that, while we know that a resolution will exist in some cases, the resolution is amorphous and the complete calculation sort of exists in the aether.  Knowing the resolution allows us to write our results in terms of more concrete well formed (but maybe not easily computable) objects.  

Specifically, we are looking at filtered deformations of certain graded rings.  G. Corti\~nas et al. computed the K-theory of cones over smooth projective varieties or equivalently the K-theory of homogeneous polynomial rings.  By this we mean, for a graded ring $R$ where Proj$(R)$ is smooth, $\tilde{K}_n(R)\cong \tilde{K}_n($Spec$(R))$ is calculated by G. Corti\~nas et al. in \cite{cortinas2009k}.  As we will see later, an explicit resolution is easy to formulate in the case of filtered deformations of these rings that are also local rings. We will be able to extend this in certain cases to rings that are not local, but will still be filtered deformations of these homogeneous polynomial rings.  In particular, the objects that we are looking at are stated in the assumptions of the main theorem which is given by the following.   

\begin{theorem}[Main theorem]
Let $A=k[x_1,...,x_n]/J$ be a graded ring where $J$ is a proper homogeneous ideal of $k[x_1,...,x_n]$ such that $E=$ Proj$(A)$ is a smooth projective scheme over some field $k$ of characteristic zero. Let $R=k[x_1,...,x_n]/I$ such that $I_{min}\cong J$ and Spec$(R)$ has only one isolated singularity at the origin.  Then $R$ is a filtered deformation of $A$ along the ideal filtration given by the ideal $\mathfrak{m}=\langle x_1,...,x_n\rangle$, and for each $n<0$ and $i\geq0$, there is a long exact sequence of Hodge components of K-theory and cyclic homology given by
\[...\rightarrow HC_{n+1}^{(i)}(E/\mathbb{Q})\rightarrow\tilde{K}^{(i+1)}_{n}(R)\rightarrow HC_{n}^{(i)}(\text{Proj}(R[\mathfrak{m}t])/\mathbb{Q}) \rightarrow HC_{n}^{(i)}(E/\mathbb{Q})\rightarrow\tilde{K}^{(i+1)}_{n-1}(R)\rightarrow... .\] 
\end{theorem}
(See \ref{fmin} and \ref{imin} for the definition of $I_{min}$)

We can get some intuition on why such objects from the main theorem are computable using this framework that was, as stated above, used to compute things like the K-theory of cones over smooth varieties.  Geometrically speaking, we already have some nice results relating to the resolution of cones and cone like singularities.  Furthermore, the resolution and the fiber of the resolution are related to each other in a nice enough way that doing computations, like K-theory computations, while never completely straight forward, is at least aesthetically pleasing.  

In our case, what we are doing is taking these nice cones and deforming them in such a way so as to keep certain properties and lose some others.  Essentially, we are not deforming these spaces too much to the point that we introduce new singularities, or to the point where we change the one isolated singularity's type (which is cone like).  We are, however, deforming these spaces just enough so that computations become much more difficult, and the end result of a partial calculation, is much less aesthetically pleasing.  

After we prove the main theorem, we will state some corollaries that will hopefully make it easier to compute the various components of the cyclic homology.  In particular, using the resulting long exact sequence, we will be able to compute the negative K-theory groups of these filtered deformations, or rather show that they vanish, for low enough degrees.   

As we present these ideas, we will give some secondary results using cdh techniques.  The first is a way to compute the periodic cyclic homology of singular varieties in terms of the de Rham cohomology of the resolution of that variety.  While this is most likely very well known, we will extend this to computations of negative cyclic homology.  This can only be achieved in full when the cyclic homology is known. We will apply this result in some known cases such as the example of hypersurfaces with isolated singularities.  In a paper by R. Michler, she computed the cyclic homology of these spaces.  We will extend these ideas to both the periodic cyclic homology using cdh techniques, and negative cyclic homology using all of our data and the SBI sequence. 

As stated before, not knowing the resolution yields a nebulous calculation in terms of the shapeless resolution.  However, despite these deficiencies, after everything is proved, we are still more equipped to do concrete computations with these secondary results than we are with our main results.

The structure of this paper is as follows.  Section 1 serves as a brief introduction to our objects of study, namely filtered rings and filtered deformations of graded rings.  Section 2 immediately ups the level of abstraction by introducing Grothendieck topologies.  In particular, by using the framework of resolution of singularities, we present our main tools: the cdh-topology and cdh-cohomology.  Section 3 presents K-theory and it's variants as well as cyclic homology and it's variants.  With the cdh-topology in mind we present various descent properties of these constructions.  Also in this section we present the general framework for making K-theory computations by using resolution of singularities and cyclic homology computations.  In section 4 we present our first new results which are the secondary results of this paper extending the cyclic homology computations to the other theories.  Section 5 presents all of our main results, and in particular, we will prove the main theorem.  
In the appendices, we will restate some of our secondary results along with their associated big diagrams that are hybrids of two long exact sequences.  

\newpage 

\begin{notation}
So as to reduce the number of times that we have to restate certain assumptions, we will adopt the following notation.  

The category Sch/$k$ refers to the category of schemes essentially of finite type over a field $k$.  For most of the results, we will need for the field $k$ to be of characteristic zero, but in most cases, we will try to restate this assumption.   The category Sm/$k$ refers to the subcategory of smooth schemes essentially of finite type over a field $k$.  In general, other than for the important main theorems, we will not restate the assumption of essentially of finite type, and it should be assumed as an assumption for each result.  Varieties, smooth or otherwise, are essentially of finite type over a field $k$ and therefore are contained in Sch/$k$.  

For a variety, or scheme, $X\in$ Sch/$k$ where $k$ is a field characteristic zero, we use the following conventions for the resolution of $X$.  For any result stated in terms of the letters $Y$, $Z$, and $E$, unless otherwise stated these will refer to the objects in an abstract blowup square.  By this we mean, the abstract blowup square,

\[\begin{CD}
E @>>> Y\\
@VVV @VVV\\
Z @>>> X 
\end{CD}\] 

\vspace{5mm}

\noindent describes the resolution of $X$, where $Y$ is the resolution, $Z$ is the center of the resolution, and $E$ is the exceptional fiber of the resolution.  The only other likely scenario where we use all of these letters together is to describe a blowup which may not resolve all of the singularities of $X$.  

Other notation conventions will be introduced as needed.   
\end{notation}

\newpage
\section{Filtered rings and filtered deformations}
In this section we will discuss our basic objects of study.  While this will only be a very brief introduction, we will still state and prove some results that we will need later for our main results.  
\subsection{Filtered rings}
Filtered rings are a generalization of graded rings in the obvious sense, instead of a grading, we have a filtration.  Before we can prove any results, we need to adopt certain definitions and conventions about filtered rings.  
\begin{definition}
A (descending) filtration on a ring $R$ is a sequence of subgroups of $R$ under addition, $F=\{F_i\}_{i\in\mathbb{N}}$, such that $R=F_0\supset F_1\supset F_2\supset ...\supset \cap^{\infty}F_n=\{0\}$.  
\end{definition}
With this definition, we can now define what it means to be a filtered ring.  
\begin{definition}
A filtered ring is a ring $R$ with a filtration $F=\{F_i\}_{i\in\mathbb{N}}$ that is compatible with the multiplication on $R$ in the following sense, for all $i,j\in\mathbb{N}$, $F_i\cdot F_j\subset F_{i+j}$.  
\end{definition}
The most natural type of filtration, and the only type that will be beneficial for us, is a filtration by an ideal.  
\begin{definition}
A ring $R$ is said to have a filtration by an ideal $I$, or is filtered by $I$, if the filtration is given by powers of $I$, $F=\{F_n\}_{n\in\mathbb{N}}=\{I^n\}_{n\in\mathbb{N}}$.
\end{definition}  
There are other filtrations that you can impose on ring, but they are usually quite close to filtrations by ideals, they mostly just shift index or insert zeros so that it is not exactly like an ideal filtration.  

\begin{observation}
Let $R$ be a graded ring, $R=\bigoplus_{i=0}^\infty R_i$, then there is a natural filtration on $R$ given by $F_n=\bigoplus_{i=n}^\infty R_i$.  This, in general, is not the only filtration we can impose on a graded ring.  
\end{observation}

While filtered rings are more general objects than graded rings, there is a canonical way to recover a graded ring from a filtered ring.  This is by constructing the associated graded ring to a filtered ring $R$ with respect to the filtration $F$, written $gr_F(R)$.  

\begin{definition}
Let $R$ be a filtered ring with filtration $F=\{F_n\}_{n\in\mathbb{N}}$.  We define the associated graded ring of $R$ with respect to $F$ as \[gr_F(R)=R/F_1\oplus F_1/F_2\oplus F_2/F_3\oplus...=\bigoplus_{i=0}^{\infty} F_i/F_{i+1},\] where addition is defined componentwise and multiplication is defined as follows.  If $a=[a]\in F_m/F_{m+1}$ and $b=[b]\in F_n/F_{n+1}$ (where $[x]$ denotes equivalence class) then $a\cdot b=[ab]\in F_{m+n}/F_{m+n+1}$.  
\end{definition}

This clearly defines a graded ring, moreover, this is the right way to define such a graded ring from a filtered ring.  The reason is, if you start out with a graded ring and take the natural filtration given by the grading, and then compute $gr_F$ of that filtered ring, then you get your original ring back.   

\begin{proposition}
Let $R=\bigoplus_{i=0}^\infty R_i$ be a graded ring and impose the natural filtration $F=\{F_n\}_{n\in\mathbb{N}}$ given by $F_n=\bigoplus_{i=n}^\infty R_i$ on $R$.  Then $gr_F(R)=R$.  
\end{proposition}
\begin{proof}
For each $n$, \[F_n/F_{n+1}=\bigoplus_{i=n}^\infty R_i\Bigg /\bigoplus_{i=n+1}^\infty R_i=R_n.\] Therefore $gr_F(R)=\bigoplus_{i=0}^{\infty} F_i/F_{i+1}=\bigoplus_{i=0}^\infty R_i=R.$ 
\end{proof}

Turning our focus to ideal filtrations, we have a slight change in terminology.  For a ring $R$ filtered by an ideal $I$, we denote the associated graded ring of $R$ with respect to the ideal $I$ by $gr_I(R)$.   

\begin{definition}
Let $R$ be a ring and $I$ an ideal of $R$.  We define the associated graded ring of $R$ with respect to $I$ as \[gr_I(R)=R/I\oplus I/I^2\oplus I^2/I^3\oplus...=\bigoplus_{i=0}^{\infty} I^i/I^{i+1},\] where the operations are the same as before with associated graded rings with respect to filtrations.  
\end{definition}
\begin{notation}
For the most part, we will use the notation, $gr_I(R)=\bigoplus_{i=0}^{\infty} I^i/I^{i+1}$ for the associated graded ring.  Later, however, when we speak about the blowup algebras, including the Rees algebra associated to an ideal, we may use the alternative notation $R[It]/IR[It]$ where $t$ is an indeterminate keeping track of the grading.  A simple argument shows that this is identical to the construction $\bigoplus_{i=0}^{\infty} I^i/I^{i+1}$, but in some cases it is more useful in that it is a quotient of the Rees algebra.  
\end{notation}

After we present the following definitions, we will present some examples of filtered rings and their associated graded rings. 

\begin{definition}\label{fmin}For a polynomial in the standard grading, we define $f(x_1,...,x_n)_{min}$ to be the minimal degree pieces of $f$ (e.g. if $f(x_1,x_2)=x_1^2-x_2^2+x_2^3$ then $f(x_1,x_2)_{min}=x_1^2-x_2^2$), where we sometimes write $f(x_1,...,x_n)_{min}$ as $(f)_{min}$.    
\end{definition}

\begin{definition}\label{imin}
Let $I$ be an ideal in a polynomial ring with the standard grading $k[x_1,...,x_n]$.  We define $I_{min}$ to be the ideal composed of the functions $(f)_{min}$ for all $f\in I$. 
\end{definition}

Following the discussion in chapter 3 of \cite{beltrametti2009lectures}, we make the following remarks.  

\begin{warning}
Suppose $I$ is an ideal in $k[x_1,...,x_n]$ with the standard grading generated by $m$ many polynomials, $I=\langle f_1,...,f_m\rangle$.  It is not necessarily the case that $I_{min}$ is generated by the $(f_i)_{min}$'s.  It is clear that $I_{min}\supset \langle(f_1)_{min},...,(f_m)_{min}\rangle$, however there may be more polynomials in $I_{min}$.    
\end{warning}
To illustrate this consider the following example.  
\begin{example}
Let $I=\langle x+y^2,x+z^3\rangle$ be an ideal in $k[x,y,z]$.  It is clear that for the two generating functions of $I$ that $(x+y^2)_{min}=x=(x+z^3)_{min}$, but it is not the case that $I_{min}=\langle x\rangle$.  This is because the difference between these two polynomials, $(x+y^2)-(x+z^3)=y^2-z^3$ is an element of $I$.  Hence $y^2\in I_{min}$, and in this case we have $I_{min}=\langle x,y^2\rangle$
\end{example}
For principle ideals, it is however the case that if $I=\langle f\rangle$, then $I_{min}=\langle (f)_{min}\rangle$.  

Now that we have all of the necessary definitions out of the way, we can proceed by presenting the following results.  
\begin{example}
Let $R=k[[x_1,...,x_n]]/I$ and $\mathfrak{m}$ be it's unique maximal ideal.  Then the associated graded ring to the ideal $\mathfrak{m}$, $gr_{\mathfrak{m}}(R)=\bigoplus_{i=0}^{\infty}\mathfrak{m}^i/\mathfrak{m}^{i+1}$ is isomorphic to $k[x_1,...,x_n]/I_{min}$.  
\end{example}
For this example we will provide a proof.  
\begin{proof}
In order to show that by applying $gr_{\mathfrak{m}}(R)$ that we obtain a normal polynomial ring, we apply $gr_{\mathfrak{m}}$ in the case where $R=k[[x_1,...,x_n]]$.  The 0th graded piece of $gr_{\mathfrak{m}}(R)$ is defined to be $R/\mathfrak{m}=k$.  The 1st graded piece of $gr_{\mathfrak{m}}(R)$ is $\mathfrak{m}/\mathfrak{m}^2$.  Since $\mathfrak{m}$ is generated by monomials of the form $x_i$ and $\mathfrak{m}^2$ is generated by monomials of the form $x_ix_j$ we have that the quotient $\mathfrak{m}/\mathfrak{m}^2$ contains only linear functions, or the standard homogeneous degree 1 polynomials.  In the general case, we have that the $d$th graded piece of $gr_{\mathfrak{m}}(R)$ is $\mathfrak{m}^d/\mathfrak{m}^{d+1}$.  Since $\mathfrak{m}^d$ is generated by monomials which are products of $d$ many of the variables $x_i$ and $\mathfrak{m}^{d+1}$ is generated by monomials which are products of $d+1$ many of the variables we have that the quotient $\mathfrak{m}^d/\mathfrak{m}^{d+1}$ contains only functions that are linear sums of monomials which are products $d$ many of the variables, or alternatively, the standard homogeneous degree $d$ polynomials.  We conclude that $gr_{\mathfrak{m}}(R)$ is the polynomial ring in $n$ variables in the standard grading.  

We now need to determine what happens in the scenario where $R=k[[x_1,...,x_n]]/I$.  Let $g(x_1,...,x_n)$ be an arbitrary power series in $I$.  Take $g(x_1,...,x_n)_{min}$ to be minimal degree parts of $g$.  By construction $(g)_{min}$ is homogeneous of degree $d=$min$\{$degrees of monomials in monomial decomposition of $g\}$.   We have that $(g)_{min}\in\mathfrak{m}^d$, and $-g+(g)_{min}\in\mathfrak{m}^{d+1}$, which means that $(g)_{min}=0$ in the quotient $\mathfrak{m}^d/\mathfrak{m}^{d+1}$ (since $(g)_{min}=-g+(g)_{min}$ in $R$).  We conclude that $(g)_{min}=0$ in $gr_{\mathfrak{m}}(R)$.  Finally, since $g$ was taken as arbitrary, we have this same correspondence for each $g\in I$, i.e. that $(g)_{min}=0$ in $gr_{\mathfrak{m}}(R)$.  This shows one way containment, i.e that $gr_{\mathfrak{m}}(R)=k[x_1,...,x_n]/J$, for some $J$ such that $J\supset I_{min}$.  

Now, suppose $h$ is a homogeneous polynomial such that $h=0\in gr_{\mathfrak{m}}(R)$. This means that for some $c$, some incarnation of $h$ is in both $\mathfrak{m}^c$ and $\mathfrak{m}^{c+1}$.  This can only happen if we can replace $h$ with all higher degree monomials.  This can only happen if $h=(g)_{min}$ for some $g\in I$.  Hence every polynomial $h$ such that $h=0\in gr_{\mathfrak{m}}(R)$ is actually some $(g)_{min}$ for some power series $g\in I$.  Therefore, we have that $J=I_{min}$, and $gr_{\mathfrak{m}}(R)=k[x_1,...,x_n]/I_{min}$.
\end{proof}

The following result could be stated as an example, but we will use this result later so, we call it a lemma instead.  
\begin{lemma}\label{associatedgraded}
Let $R=k[x_1,...,x_n]/I$ for $I=\langle f_1,...,f_r\rangle$, and let $\mathfrak{m}$ be the maximal ideal $\langle x_1,...,x_n\rangle$.  If each $f_i$ has a zero constant term, then the associated graded ring to the ideal $\mathfrak{m}$, $gr_{\mathfrak{m}}(R)=\bigoplus_{i=0}^{\infty}\mathfrak{m}^i/\mathfrak{m}^{i+1}$ is isomorphic to $k[x_1,...,x_n]/I_{min}$ with the standard grading.  
\end{lemma}

\begin{proof}
The proof of this lemma is identical to the argument given in the proof of the previous example.  
\end{proof}
For more algebraic details on the associated graded ring, please see \cite{atiyah1969introduction}.

This idea of the associated graded ring has a geometric interpretation.  In particular, it describes the behavior of the tangents to a scheme at a particular point.  There is a well known space, called the tangent space, that describes the tangent behavior at various points.  The problem is, the tangent space can in some cases contain too much information, and it therefore does not accurately describe the tangent behavior at certain points, namely singular points.  For these scenarios, we adopt the following construction.  

\begin{definition}
For a point $P$ on an affine scheme Spec$(R)$ (where $P$ is a prime ideal in $R$), the tangent cone at $P$ is defined to be Spec$(gr_{P_P}(R_P))$.     
\end{definition}

\begin{note}
This definition can of course be extended to non-affine schemes in the following way.  For a scheme $X$, and a point $x\in X$, we have that the stalk of the structure sheaf $\mathcal{O}_{X,x}$ is a local ring with one maximal ideal $\mathfrak{m}$.  The tangent cone at $x$ is given by \[gr_{\mathfrak{m}}(\mathcal{O}_{X,x})=\bigoplus_{i=0}^{\infty}\mathfrak{m}^i/\mathfrak{m}^{i+1}.\]
\end{note}

The best way to gain an intuition about the geometric behavior of this tangent cone construction is to look at an example.  The standard example to look at is the nodal cubic.  Let $R=k[x,y]/\langle y^2-x^2-x^3\rangle$ be the coordinate ring defining the nodal cubic.  We have that $X=$ Spec$(R)$ is given by the following curve (to visualize $k$ is something like $\mathbb{R}$).  
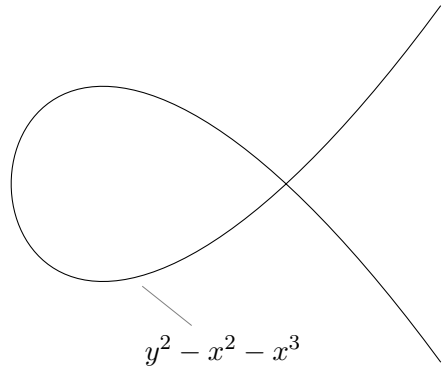
\begin{center}
\begin{tikzpicture}
    \begin{axis}[axis lines=none,]
            \addplot [domain=-1.25:1.25,samples=100]({x^2-1},{x^3-x}) 
             node[pos=.65,pin=273:{$y^2-x^2-x^3$}] {};
    \end{axis}
\end{tikzpicture}
\end{center}
In this scenario, we will first discuss the inaccuracies of the tangent space.  Since $X$ is locally defined by equations (namely the equation $f=y^2-x^2-x^3$), we can compute the tangent space at a point by using the partial derivatives.  By this we mean for each function $f_j$ defining $X$ (locally), and for a point $p=(p_1,...,p_n)$ we take the collection $g_j=\sum_{i=1}^{n}\frac{\partial f_j}{\partial x_i}(p)(x_i-p_i)$. The zero set of each of these $g_j$ is the tangent space at $p$, $T_p(X)$.  

For this example of the nodal cubic, let's compute the tangent space at two points, (-1,0), and  the origin (0,0).  For the point $p=(-1,0)$, we have that $T_{(-1,0)}(X)=V(x+1)$ (where $V(x+1)$ just means the collection of all points in $\mathbb{A}^2$ where $x+1=0$).   Visualizing this on our picture of the nodal cubic, we have that the tangent space behaves as we want.  
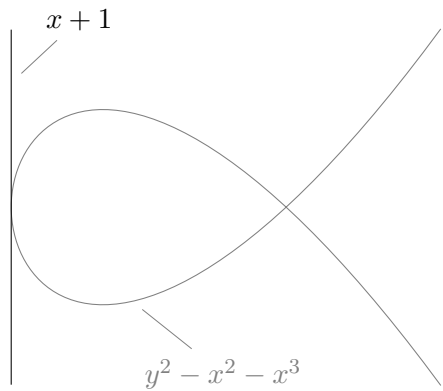
\begin{center}
\begin{tikzpicture}
    \begin{axis}[axis lines=none,]
            \addplot [color=gray,domain=-1.25:1.25,samples=100]({x^2-1},{x^3-x}) 
             node[pos=.65,pin=273:{$y^2-x^2-x^3$}] {};
            \addplot [domain=-.7:.7,samples=100]({-1},{x})
            node[pos=.85,pin=60:{$x+1$}] {};
    \end{axis}
\end{tikzpicture}
\end{center}
Based on our intuition and the picture, the tangent space does a pretty good job describing the tangent behavior at $(-1,0)$.  

Now, let's look at the tangent space at the origin. Going through a similar computation using partial derivatives, we have that the tangent space at the origin is \[T_{(0,0)}(X)=V(0)=\mathbb{A}^2.\]  This is the entire affine plane, and really does not accurately describe the tangent behavior of anything other than a surface, which is definitely not equal to our cubic curve.   

The solution to this problem is to look at the tangent cone as defined above.  Call $\langle x,y\rangle=\mathfrak{m}$, the ideal defining the point at the origin (0,0).  Using our definition from above, the tangent cone of $X$ at the point (0,0), $TC_{(0,0)}(X)$, is equal to Spec$(gr_{\mathfrak{m}_{\mathfrak{m}}}(R_{\mathfrak{m}}))$.  Using lemma \ref{grlocalization}, and computations carried out in the proof of theorem \ref{app1}, we know that in this case, \[gr_{\mathfrak{m}_{\mathfrak{m}}}(R_{\mathfrak{m}})=gr_{\mathfrak{m}}(R)=k[x,y]/\langle y^2-x^2-x^3\rangle_{min}=k[x,y]/\langle y^2-x^2\rangle=k[x,y]/\langle (y-x)(y+x)\rangle.\]  So, our computation shows that the tangent cone $TC_{(0,0)}(R)=$ Spec$(k[x,y]/\langle (y-x)(y+x)\rangle)$.  Visualizing this on our graph of the cubic shows that the tangent behavior is indeed modeled more accurately with the tangent cone.  

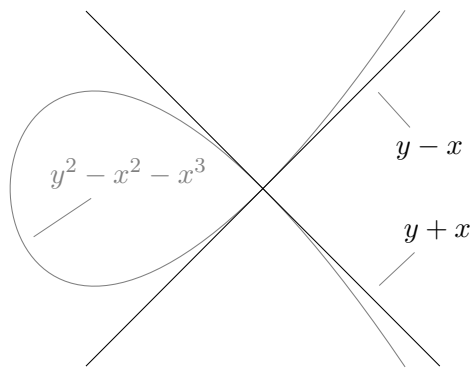
\begin{center}
\begin{tikzpicture}
    \begin{axis}[axis lines=none,]
            \addplot [color=gray,domain=-1.25:1.25,samples=100]({x^2-1},{x^3-x}) 
             node[pos=.55,pin=70:{$y^2-x^2-x^3$}] {};
            \addplot [domain=-.7:.7,samples=100]({x},{x})
            node[pos=.8,pin=290:{$y-x$}] {};
            \addplot [domain=-.7:.7,samples=100]({x},{-x})
            node[pos=.8,pin=60:{$y+x$}] {};
    \end{axis}
\end{tikzpicture}
\end{center}
As we can see from the picture, the tangent cone gives one tangent line for each of two pieces of the nodal cubic crossing at the origin.  For more (or less) complicated examples, the tangent cone will behave in the same way.  Essentially looking at the tangent line leading up to the singular point and then at the singular point it will give one tangent line for each of the sequences leading up to that point.  

It is more work, but one can show that the tangent cone at the point (-1,0) agrees with our computation of the tangent space.  Although doing this directly is possible (and the way we computed the tangent cone), the easiest way to do this would probably be to do a coordinate change and you would end up with a computation similar to that from above. 

For more description and geometric properties of the tangent cone, see section 5.4 of \cite{eisenbud1995commutative}, and section III.2.4 of \cite{eisenbud2000geometry}.

\subsection{Deformations}
For anyone interested in deformation theory, this paper is not a deformation theory paper.  When studying the paper on cones over smooth projective varieties, \cite{cortinas2009k}, it was thought that we could generalize those results to different singular varieties.  When attempting to do so, we discovered a class of rings (or equivalently affine varieties) in which the K-theory could be computed or partially computed.  It was only after the fact that we determined that what we were looking at were filtered deformations of the rings defining these cones over smooth varieties. 

With that being said, we will still try to give a brief description, in very basic terms, of what we mean by a deformation, and why we call our objects of study deformations of certain algebras with nice properties.  

We begin by looking at a general description of an algebra deformation, for which we will provide specific examples.  This idea of deforming algebras was first studied by M. Gerstenhaber in his series of papers beginning with \cite{gerstenhaber1964deformation}.  While this is the best reference to look at for these types of deformations, for a novice in the field, a more readable introduction (and the reference that we use for this presentation) is \cite{fox1993introduction}.  

\begin{idea}
A deformation of an object $X$ essentially refers to a family of $X_t$, where the $X_t$ are obtained by deforming certain aspects of some structure on $X$, and the structure on $X_t$ varies in a smooth way according to the parameter $t$.  For an algebra $A$, the structure that we usually consider deforming is the multiplicative structure on $A$ resulting in some family $A_t$.  This can be achieved in various ways, but in the case of the coordinate ring defining a variety, the most intuitive way is to alter in a specific way the equations defining that variety.  In particular, we will consider adding higher degree pieces with varying coefficients.  
\end{idea}

\begin{notation}
According to the previous description of deformation, it would seem that the entire family of objects $X_t$ refers to the deformation of $X$.  While this may be true in some cases, we will mostly refer to a specific object $X_{t_0}$ in the family $X_t$ as a deformation of $X$, not the whole family.  
\end{notation}

We will attempt a more formal description and definition, and we will follow these with some motivating examples.  
\begin{definition}
A one parameter formal deformation of a $k$-algebra $A$ is a formal power series $F=\sum_{n=0}^{\infty} f_n t^n$ where each of the $f_n$ is a morphism $f_n:A\otimes A\rightarrow A$, and such that $f_0:A\otimes A\rightarrow A$ is the standard multiplication in $A$.  
\end{definition}
We will try to parse through this language a little bit.  For the algebra $A$, we have a multiplicative structure on $A$ given by the map $f_0:A\otimes A\rightarrow A$.  We can change that multiplicative structure by using a collection of different morphisms $f_n: A\otimes A\rightarrow A$.  In order to vary this multiplicative structure in a smooth way according to the parameter $t$, we introduce this formal sum.  This gives the smooth variation on the multiplicative structure, so as $t$ varies from 0 to say a, the standard  multiplicative structure transforms smoothly into the structure $\sum_{n=0}^{\infty} f_n a^n: A\otimes A\rightarrow A$.  

This still seems a little mysterious, so let's take a look at some examples.   
\begin{example}
Let $A=k[x]/\langle x^2\rangle$ be the truncated polynomial ring in one variable.  This is a very basic $k$-algebra with the multiplication rule given by $f_0(x,x)=0$, (and then coefficients in $k$ behave as they should in conjunction with this rule).  Now, let $f_1:A\otimes A\rightarrow A$ be another multiplication rule given by $f_1(x,x)=1$.  From these two rules, we can define a power series $F=\sum_{n=0}^{\infty} f_n t^n=f_+f_1t$, and hence a one parameter formal deformation defining the multiplication rule on $A_t$.  Under this multiplication rule, we have that $F(x,x)=f_0(x,x)+f_1(x,x)t=0+1t=t$.  So in $A_t$ we have that $x^2=t$. It is not hard to see then that the objects in the family of deformations are of the form $A_t=k[x,t]/\langle x^2-t\rangle$
\end{example}
While this is a very basic example of a deformation, it shows that we can obtain deformations through the alteration of the equations defining a variety, or rather the coordinate ring of a variety.   
\begin{example}
Let $A=k[x,y]/\langle y^2-x^2\rangle$.  Let's go the other way now, take $y^2-x^2-tx^3$ to be the original equation with the addition of higher order terms.  This defines a deformation $A_t=k[x,y,t]/\langle y^2-x^2-tx^3\rangle$ in the following way.  The multiplication $f_0$ is inherited from $A$ so that things like $f_0(y,y)=x^2$.  Then to this we add only one other multiplication rule, $f_1$ (times the variable $t$), where for $f_1$, we have that $f_1(y,y)=x^3$ (and other multiples are defined accordingly).  It is not hard to see then for the complete multiplication rule on $A_t$, we have $F(y,y)=f_0(y,y)+f_1(y,y)t=x^2+tx^3$, as expected.  For more details see example 1.1 and the discussion after 1.2 in \cite{fox1993introduction}.  
\end{example}

The previous two examples were chosen because they are both specific examples of deformations where we alter the equation defining a variety, and one just happens to be the sum of our defining equation with higher order terms.  In the second example, we showed that we could deform the slanted cross, defined by $y^2-x^2$, into the nodal cubic, defined by $y^2-x^2-x^3$.  We can get an idea of what is going on with this type of deformation by looking at the following picture.  Taking the equation $y^2-x^2-tx^3$, and varying the parameter $t$, we observe the following.

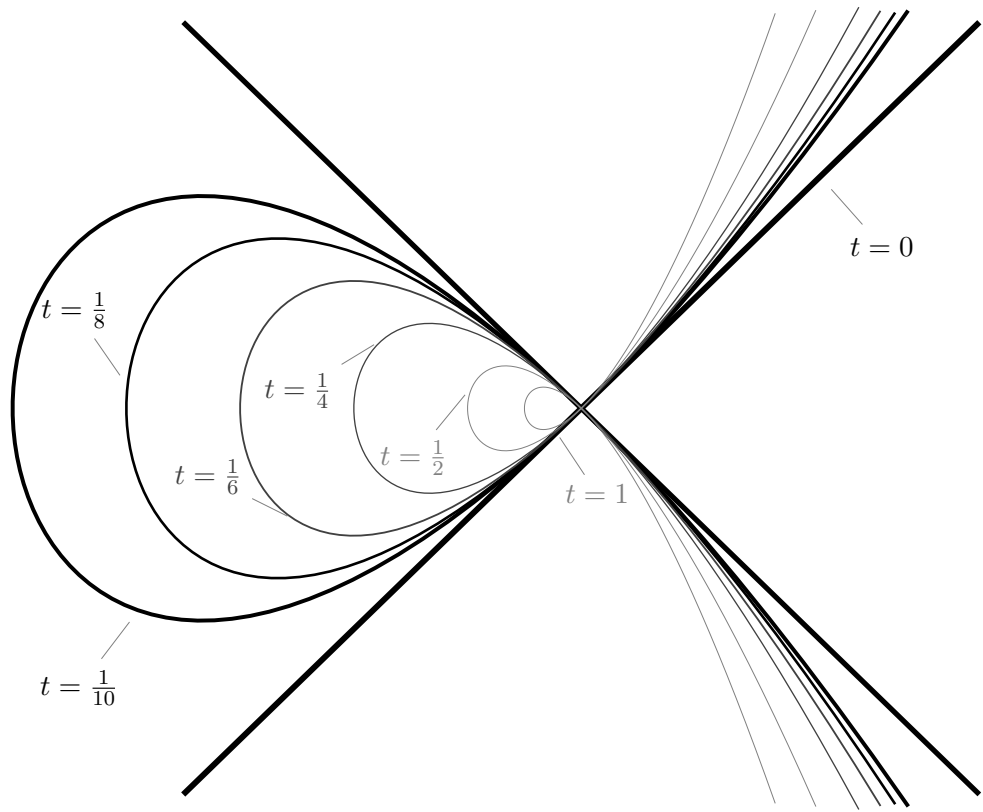
\begin{center}
\begin{tikzpicture}
    \begin{axis}[axis lines=none, scale=2.25]
            \addplot [line width=2.25pt,domain=-7:7,samples=100]({x},{x})
            node[pos=.8,pin=290:{$t=0$}] {};
            \addplot [line width=2pt,domain=-7:7,samples=100]({x},{-x});
            \addplot [line width=1.5pt,domain=-1.255:1.255,samples=100]({10*x^2-10},{10*x^3-10*x}) 
             node[pos=.6,pin=260:{$t=\frac{1}{10}$}] {};
            \addplot [line width=1.0pt,domain=-1.3:1.3,samples=100]({8*x^2-8},{8*x^3-8*x}) 
             node[pos=.49,pin=100:{$t=\frac{1}{8}$}] {};
            \addplot [color=darkgray,line width=.7pt,domain=-1.37:1.37,samples=100]({6*x^2-6},{6*x^3-6*x}) 
             node[pos=.57,pin=160:{$t=\frac{1}{6}$}] {};
            \addplot [color=darkgray,line width=.5pt,domain=-1.49:1.49,samples=100]({4*x^2-4},{4*x^3-4*x}) 
             node[pos=.45,pin=210:{$t=\frac{1}{4}$}] {};
            \addplot [color=gray,line width=.3pt,domain=-1.75:1.75,samples=100]({2*x^2-2},{2*x^3-2*x})
             node[pos=.48,pin=250:{$t=\frac{1}{2}$}] {};
            \addplot [color=gray,line width=.1pt,domain=-2.1:2.1,samples=100]({x^2-1},{x^3-x}) 
             node[pos=.54,pin=273:{$t=1$}] {};
    \end{axis}
\end{tikzpicture}
\end{center}
Essentially, at $t=0$ we begin with a slanted cross, and then as the parameter $t$ increases, that cross folds in on itself (on the left hand side) into the standard nodal cubic. While this is a very basic example of a deformation, it is examples like this that will give us intuition on the geometric nature of these deformations.  

As we have already seen, the coordinate ring of the nodal cubic,  $k[x,y]/\langle y^2-x^2-x^3\rangle$, is a filtered ring with filtration given by the ideal $\mathfrak{m}=\langle x,y\rangle$.  Moreover, the associated graded ring to this filtration is the slanted cross again, $k[x,y]/\langle y^2-x^2\rangle$.  This seems to suggest that deformations are related to these associated graded ring constructions in a particular way.  
 
This idea leads to the following discussion of filtered deformations.  While in the previous more specific example we already have that $k[x,y]/\langle y^2-x^2-x^3\rangle$ is a deformation of $k[x,y]/\langle y^2-x^2\rangle$, but we are now more concerned with the general case of filtered rings and their associated graded rings.  
\begin{definition}
Let $A$ be a graded ring.  We say that a filtered ring $R$ with filtration $F$ is a filtered deformation of $A$ if $gr_F(R)=A$.  
\end{definition}
This definition yields the following corollary of an earlier result.    
\begin{corollary}\label{associatedgradedcorollary}
Let $R=k[x_1,...,x_n]/I$, where for $I=\langle f_1,...,f_r\rangle$ each $f_i$ has a zero constant term.  Then, $R$ is a filtered deformation of $A=k[x_1,...,x_n]/I_{min}$.  
\end{corollary}
\begin{proof}
By lemma \ref{associatedgraded}, we know that $gr_{\mathfrak{m}}(k[x_1,...,x_n]/I)=A=k[x_1,...,x_n]/I_{min}$, proving that $R$ is a filtered deformation of $A$ along the ideal filtration given by $\mathfrak{m}$
\end{proof}
This is a nice result, and it is really the only definition we need.  For someone interested in the more general case, the following results show that filtered deformations can be viewed as more traditional deformations of graded algebras.  
\begin{theorem}\label{filtdefaredef}
A filtered ring is a deformation of it's associated graded ring, and hence a filtered deformation is a deformation.  
\end{theorem}
\begin{proof}
This proof is very technical, so we will not give the argument here.  This result was originally proved by M. Gerstenhaber as the main result in the second of his series of papers \cite{gerstenhaber1966deformation2}.  His proof is actually a proof of a slightly different statement for completed rings, however other authors have used and extended his work to the non-completed case as well.  
\end{proof}

While we do not really need any of these general deformation theory results, we now have some idea on the basics.  With the previous result in mind, for future study, it may be possible to extend our main results to more general deformations other than the filtered type.

\newpage
\section{A brief introduction to cdh-cohomology}
For the time being, we are going to depart from our discussion about filtered rings and filtered deformations in order to introduce a fairly powerful computational tool.  The cdh-topology on the category of schemes and cdh-cohomology for sheaves on this site have been investigated by previous authors, and this will be our main tool of computation. The idea is to construct a topology with enough covers so as to reduce computations on singular schemes to that of smooth schemes.  While it does not seem to be such an easy feat, other authors, namely Heisuke Hironaka, have given us the way to do this.  As we will soon see, this topology was constructed explicitly so that it works well with the methods of resolution of singularities.  

\subsection{Resolution of singularities}
As stated above, the whole purpose of cdh-cohomology is so that it behaves well with the methods of resolution of singularities.  Before we can dive into the cdh-topology, we should state precisely what is meant by the methods of resolution of singularities.  The main resource for this is a series of classes given by the Clay Math Institute in Obergurgl Austria during the summer of 2012.  Lecture notes for this can be found at \cite{hauserblowupsresolution}.  In addition to this, we look at other standard references: \cite{cutkosky2004resolution}, and \cite{kollar2009lectures}.  The general idea of resolution is as follows.  

Suppose $X$ is a variety over a field $k$.  
\begin{definition}
A resolution of singularities of $X$ is a proper birational morphism $\varphi:Y\rightarrow X$ such that $Y$ is a nonsingular variety.  
\end{definition}
\begin{note}This definition can of course be reformulated in terms of schemes.  We will avoid discussing resolution of singular schemes in full generality, because any such discussion will involve the very technical definitions of excellent or quasi-excellent schemes.  
\end{note}
The question of whether or not every variety admits a resolution was, and is still, investigated by many mathematicians.  Eventually, Heisuke Hironaka obtained the following result in \cite{hironaka1964resolution} (a more accessible reference, that also includes some technical details, is \cite{hauser2003hironaka}), which solved the problem in characteristic zero.
\begin{theorem}[Hironaka]
Let $X$ be a variety over a field $k$ of characteristic zero.  Then there is a resolution $\varphi:Y\rightarrow X$ such that \\
$\bullet$ $\varphi$ is an isomorphism over the regular points of $X$\\
$\bullet$ $\varphi^{-1}($Sing$(X))$, called the exceptional fiber, is a union of smooth hypersurfaces of $Y$ having only normal crossings.  \\
Furthermore, this resolution is a composition $Y\rightarrow X_n\rightarrow...\rightarrow X_1\rightarrow X_0=X$ of finitely many blowups along smooth centers $Z_i(\subset X_i)$.  
\end{theorem}
The proof of this theorem is extremely technical, so we will avoid saying anything non-trivial about it.  The construction required for this resolution theorem, the blowup, is a little more tractable and will be an important part of our construction of the cdh-topology.  There are many ways to define a blowup (\cite{hauserblowupsresolution} lists 7 definitions), but for our purposes we will use a more scheme theoretic definition.  

Let $X=$ Spec$(R)$ be an affine scheme, let $I$ be an ideal of $R$, and call $Z=$ Spec$(R/I)$ the closed subset of $X$ defined by $I$.    
\begin{definition}
The blowup algebras of $R$ along $I$ are given by $R[It]=\bigoplus_{n\geq 0}I^n$ and $gr_I(R)=\bigoplus_{n\geq 0}I^n/I^{n+1}(=R[It]/IR[It])$.
\end{definition}
Using this we can define the blowup of the affine scheme $X$.  
\begin{definition}
The blowup of $X$ along $Z$ is given by $Y=$ Proj$(R[It])$ along with the map $\pi:Y\rightarrow X$.  The exceptional fiber of this map is given by $E=\pi^{-1}(Z)=$ Proj$(gr_I(R))$.  
\end{definition}
\begin{twocents}
At this point our discussion of filtered deformations should seem very relevant to resolution of singularities.  Namely, one of the blowup algebras (the one corresponding to the exceptional fiber) contains the $gr_I$ construction.  This seems to suggest that $R$ is a filtered deformation of this blowup algebra.  
\end{twocents}

The data of the blowup can be summed up in the following, rather important, diagram.  
\[\begin{CD}
\text{Proj}(gr_I(R))  @>>> \text{Proj}(R[It])\\
@VVV @VV\pi V\\
\text{Spec}(R/I)@>\iota >> \text{Spec}(R)
\end{CD}\]
The importance of this diagram will become more clear shortly.

\subsection{Topologies on categories}
We will begin with a brief review of Grothendieck topologies.  
\begin{idea}The general idea with Grothendieck topologies is to generalize the notion of a sheaf.  The most basic version of a sheaf is a contravariant functor with the domain as the category of open sets of a topological space, where the objects are open sets and arrows correspond to inclusion (along with some other conditions).  Instead, if we wanted a sheaf on a more general category we would still have a contravariant functor, but our notion of open sets and inclusion would be replaced with something else, namely sieves and covering sieves
\end{idea}

We begin with the traditional construction of a sheaf on a topological space.  For simplicity we will construct a sheaf into the category of Sets, where this construction can be easily generalized to other categories such as Groups, Rings, etc.  
\begin{definition}
Let $X$ be a topological space. We define $O(X)$ to be the category whose objects consist of open subsets of $X$, and whose morphisms consists of inclusion maps, i.e. if $U\subset V$ are two open subsets in $X$ then there is a morphism $f\in$ Morph$(O(X))$ corresponding to the inclusion $f:U\rightarrow V$.  
\end{definition}
\begin{definition}
Let $X$ be a topological space.  A presheaf of sets on $X$ is a contravariant functor $F$ from $O(X)$ to the category of Sets.  
\end{definition}
\begin{definition}
A presheaf $F$ of sets on a topological space $X$ is a sheaf if it satisfies the following conditions:\\
$\bullet$ (Locality) For a collection of open subsets $(U_i)$ that is an open covering of an open set $U$, if $s,t\in F(U)$ such that on the restriction $s|_{U_i}=t|_{U_i}$ for all $U_i$ in the covering $(U_i)$, then $s=t$.  \\
$\bullet$ (Gluing) For an open cover $(U_i)$ of $U$, if for each $i$ there is a section $s_i\in F(U_i)$ such that $s_i|_{U_i\cap U_j}=s_j|_{U_i\cap U_j}$ for each pair $U_i$ and $U_j$, then there is a section $s\in F(U)$ such that  $s|_{U_i}=s_i$ for each $i$. 
\end{definition}
As we will see, the whole goal is to generalize sheaves so that our domain category is arbitrary.  It turns out it is fairly easy to generalize the notion of a presheaf, what is a little more difficult is developing a concrete way to satisfy the locality and gluing axioms for a sheaf.  

We note that the actual structure of the topological space $X$ from above is largely ignored.  The information that we use is related to subsets of open sets, and from there we reformulate everything in terms of the inclusion morphism.  This is how we will generalize the notion of covering, in a sense, a Grothendieck topology determines when a morphism is an ``inclusion."  

Let $C$ be a category.  
\begin{definition}
Let $c\in$ Obj$(C)$.  A sieve on $c$ is a family of morphisms in Morph$(C)$ all with codomain $c$ such that $f\in S$ implies that $f\circ g\in S$ whenever this composition makes sense.   
\end{definition}
\begin{definition}
If $S$ is a sieve on $c$ and $h:d\rightarrow c$ is any arrow to $c$ in Morph$(C)$, then \[h^*(S)=\{g|\text{cod}(g)=d, h\circ g\in S\}\] is a sieve on $d$ called the pullback sieve of $S$ along $h$ and denoted $h^*S$.  
\end{definition}
We are now ready for the definition of a Grothendieck topology on a category.  
\begin{definition}A Grothendieck topology on a category $C$ is an assignment which associates to each object $c$ of $C$ a collection of ``covering sieves"  denoted by $J(c)$ such that the following axioms hold.  \\
$\bullet$ The maximal sieve $t_c=\{f|\text{cod}(f)=c\}$ is in $J(c)$.  \\
$\bullet$ (Stability) If $S\in J(c)$ then the pullback $h^*S\in J(d)$ for any $h:d\rightarrow c$ in Morph$(C)$.\\  
$\bullet$ (Transitivity) If $S\in J(c)$ and $R$ is any sieve on $c$ such that the pullbacks $h^*R\in J(d)$ for all $h\in S$, then $R\in J(c)$.  
\end{definition}
In the standard way we can define smaller constructions, such as a ``basis",  that will generate a Grothendieck topology.  For our purposes a very precise definition of these construction won't be too necessary.  We will simply say that the Grothendieck topology generated by certain families of morphisms is the smallest Grothendieck topology such that these families are part of these distinguished covering sieves. 

We will call a category $C$ equipped with a Grothendieck topology $J$ a site, denoted by $(C,J)$  
The whole point of this was to generalize sheaves, so we will do this construction now.  
\begin{definition}
Let $C$ be a category.  A presheaf of sets on $C$ is a contravariant functor $F$ from $C$ to the category of Sets.  
\end{definition}
\begin{definition}
For a presheaf $F$ and a covering sieve $S$ of an object $c$, a matching family for $S$ of elements of $F$ is a function which assigns to each element $f:d\rightarrow c$ of $S$ an element $x_f\in F(d)$ such that \[x_{fg}=F(g)(x_f)\ \ \ \text{for all}\ g:e\rightarrow d\ \text{in}\ C\]
\end{definition}
Note that because $S$ is a sieve $fg$ is again an element of $S$.  
\begin{definition}
An amalgamation of such a matching family is a single element $x\in F(c)$ with \[x_{f}=F(f)(x)\ \ \ \text{for all}\ f\in S\]
\end{definition}
\begin{definition}
A presheaf $F$ of sets on a category $C$, is a sheaf with respect to the topology $J$ if every matching family has a unique amalgamation.  (Equivalently we could say that $F$ is a sheaf on the site $(C,J)$)
\end{definition}
It is fairly easy to see that the existence of an amalgamation is analogous to the gluing condition in the original definition of sheaf, and uniqueness of this amalgamation corresponds to locality.  Therefore this, at least intuitively, is the correct way to generalize a sheaf.  

For a more in depth presentation of Grothendieck topologies on categories and sheaves on a site, see \cite{mclane1994sheaves}.

\subsection{The cdh-topology}
Next we will give the construction of the completely decomposed h-topology, or the cdh-topology, on the category of schemes essentially of finite type over a field $k$, written Sch/$k$.  We will also define a topology on the subcategory consisting of smooth schemes, Sm/$k$, called the scdh-topology.  For these constructions, we mostly follow the presentation given in \cite{cortinas2005cyclic}, however we also use \cite{haesemeyer2004descent}, \cite{voevodsky2000homotopy}, and \cite{voevodsky2008unstable} as additional references as well as lectures given by Dr. Corti\~nas at the University of Colorado in the spring of 2013.  

Our goal will be to build a topology on the category of schemes with nice properties (which we will discuss later). To achieve this, we will build the topology with a class of covers given by distinguished squares. By this we mean that given a specific square,
\[\begin{CD}
A  @>>> B\\
@VVV @VV\psi V\\
C @>\varphi >> D
\end{CD}\]
if certain conditions on the morphisms are met, we include the cover $\{\psi,\varphi\}$ as part of the generating covers for the topology.  Such a collection of squares will be given the following name.
\begin{definition}
For a small category $C$, a $cd$-structure on $C$ is a class of commutative squares in $C$ that is closed under isomorphism.  
\end{definition}

The process of resolving singularities, namely the blowup, yields a framework to produce these distinguished squares in the category of schemes.  Namely, 
\[\begin{CD}
E=\text{Exceptional Fiber} @>>> Y=\text{Blowup of }X\text{ along Z}\\
@VVV @VVV\\
Z=\text{Center of the blowup} @>>> X
\end{CD}
\]
At this point, if we have any hope of relating singular and nonsingular schemes, our topology must include covers from these blowup squares.  However, as we have mentioned before, one blowup is not always sufficient to resolve singularities.  While we could stop here and leave computations to iterated blowup squares, it will be easier in these computations to include more covers related to the resolution.   

The following type of distinguished square is meant to contain all of the relevant data of resolution of singularities, and it is called an abstract blowup square. 
\begin{definition}
An abstract blowup square of schemes in Sch/$k$ is a Cartesian square of schemes over $k$
\[\begin{CD}
E @>>> Y\\
@VVV @VVV\\
Z @>>> X
\end{CD}
\]
with the following properties: $Y\rightarrow X$ is a proper morphism, and $Z\rightarrow X$ is a closed embedding such that the induced morphism $(Y\setminus E)^{red}\rightarrow (X\setminus Z)^{red}$ is an isomorphism. 
\end{definition}
Geometrically speaking, this square contains all of the information of a resolution of singularities. We interpret this in the following way, let $X$ be some scheme with closed subset $Z$, we have  
\[\begin{CD}
E=\text{Exceptional Fiber} @>>> Y=\text{Resolution of }X\\
@VVV @VVV\\
Z=\text{Singular Locus} @>>> X
\end{CD}
\]
where the conditions on the morphisms ensure that this is indeed a resolution.  

While abstract blowup squares contain all of the data of resolution of singularities, they include many other types of squares as well.  An obvious example of this is the following.
\begin{example}
Let $X$ be the disjoint union of two identical closed subsets $Z$, $X=Z\sqcup Z$.  We have that 
\[\begin{CD}
Z @>>> X\\
@VVV @VVV\\
Z @>>> X
\end{CD}
\]
is an abstract blowup square.  
\end{example}

In addition to the covers given by abstract blowup squares, the cdh-topology also contains covers from a coarser topology, called the Nisnevich topology. 
\begin{definition}
The Nisnevich topology Sch/$k$ is the topology generated by the following covers, $\{U\rightarrow X, V\rightarrow X\}$, where $U\rightarrow X$ is an open embedding and $V\rightarrow X$ is an Etale morphism that is an isomorphism over $X\setminus U$
\end{definition}
In order to visualize these covers, we usually write out what is called a Nisnevich square.  In the same scenario as before, we have that 
\[\begin{CD}
U\times_X V @>>> V\\
@VVV @VVV\\
U @>>> X
\end{CD}\]
is a (elementary distinguished) Nisnevich square.  

\begin{remark}The squares given above are examples of cd-structures in a category.  We call the combined cd-structure on Sch/$k$ the collection of all Nisnevich squares and the collection of all abstract blowup squares.  For the combined cd-structure on the smooth subcategory of Sch/$k$, called Sm/$k$, we use the Nisnevich squares and the squares of smooth schemes that correspond to smooth blowups along smooth centers. \end{remark}

These cd-structures have some nice properties including the following which we will present without definition and without proof. 
\begin{lemma}Each of these combined cd-structures, the one on Sch/$k$ and the one on Sm/$k$, are complete, bounded, and regular.  
\end{lemma}
\begin{proof}See \cite{voevodsky2008unstable} section 2.\end{proof}

So, we arrive at our definition of the cdh-topology.  
\begin{definition}
The cdh-topology on Sch/$k$ is the topology generated by the combined cd-structure on Sch/$k$, i.e. the Nisnevich topology and covers of the form $\{Z\rightarrow X, Y\rightarrow X\}$ for abstract blow-ups $Y\rightarrow X$ with center $Z$.
\end{definition}

The cdh-topology is a topology on Sch/$k$, we will also use a similar topology on the category of smooth schemes over a field $k$ called the scdh-topology.  We can go through a similar construction, and we end up with the following definition.  
\begin{definition}
The scdh-topology on Sm/$k$ is the topology generated by the combined cd-structure on Sm/$k$.
\end{definition}
As luck would have it, since Sm/$k$ is a proper subcategory of Sch/$k$, and the definitions happen to agree on this subcategory, we can simply define this new topology on Sm/$k$ to simply be the restriction of the topology on Sch/$k$.  
\begin{proposition}
The scdh-topology on the category Sm/$k$ is the restriction of the cdh-topology on Sch/$k$ to the subcategory Sm/$k$.  
\end{proposition}
For more information on these two topologies, their properties, and how they are related to one another, please see \cite{voevodsky2008unstable}.  

\subsection{Sheaves on a site}
Now that we have defined all of the necessary topologies on our categories of concern, we can take a closer look at sheaves on Sch/$k$.  The following constructions will work for an arbitrary topology $t$ on the category Sch/$k$.  When we require the use of the cdh-topology, we will note it's requirement.  

We will mostly be concerned with sheaves into the category of spectra.  We will briefly recall the notion of a topological spectrum (see \cite{aguilar2002algebraic}).     
\begin{definition}
A family of pointed spaces $\{P_n\}$ together with pointed maps $\sigma_n:\Sigma P_n\to P_{n+1}$ (where $\Sigma X$ denotes the suspension of $X$) is called a prespectrum.  
\end{definition}
As these definitions tend to proceed, now that we have the notion of a prespectrum, we can continue and define a spectrum.  
\begin{definition}
A spectrum is a prespectrum, $\{P_n\}_{n\in\mathbb{Z}}$ together with $\sigma_n:\Sigma P_n\to P_{n+1}$, such that the adjoint maps $\hat{\sigma}_n:P_n\rightarrow \Omega P_{n+1}$ are homeomorphisms (all we really need however is that these maps are weak equivalences).  
\end{definition}
The purpose of spectra is to define various generalized cohomology theories from a purely homotopical viewpoint.  By this we mean that we can construct spectra so that their homotopy groups agree with whatever generalized cohomology theory that we are working with.  
\begin{definition}
Given a (pre)spectrum $P=\{P_n\}$ we define its homotopy groups by \[\pi_n(P)=\lim_{\overset{\longrightarrow}{k}}\pi_{n+k}(P_k).\]  
\end{definition}
\begin{remark}
Note, that homotopy groups of spectra are defined to be direct limits of homotopy groups of topological spaces.  Since the initial groups in a direct limit computation are (mostly) irrelevant, this definition allows for negative homotopy groups of spectra.  This is because in the limit stage all of the computations are for positive homotopy groups of topological spaces.   
\end{remark}

Returning to our discussion of sheaves we introduce the following notation.  
\begin{definition}
Let Shv$_t$(Sch/$k$,Spectra) denote the category of sheaves on Sch/$k$ with topology $t$ into the category of spectra.  
\end{definition}
Notice that we are beginning by looking at sheaves into the category of spectra.  From this, we can infer that our eventual goal will be to define homotopy groups of the spectra, or equivalently, to use homotopy groups on Shv$_t$(Sch/$k$,Spectra) to define a functor from Sch/$k$ into groups.  

In order to do this, we need to impose a model structure on the category Shv$_t$(Sch/$k$,Spectra).  Following \cite{cortinas2005cyclic} we note that there are several relevant model structures that we can impose on this category.  Namely, they are the global projective model structure, the local injective model structure, and the local projective (or Brown-Gersten) model structure.  While we will not explicitly use the definition of any of these in proof, as in \cite{cortinas2005cyclic}, we will primarily look at the local injective model structure, and it is defined as follows.  

\noindent \textbf{Local Injective Model Structure}
\begin{itemize}
\item Weak equivalences are defined at the level of stable homotopy groups.  By this we mean, a morphism $f: \mathcal{E}\rightarrow \mathcal{E}'$ is a (local) weak equivalence if $f$ induces an isomorphism on sheaves of stable homotopy groups. 
\item Cofibrations are defined objectwise, meaning a morphism $f: \mathcal{E}\rightarrow \mathcal{E}'$ is a cofibration provided $f(U):\mathcal{E}(U)\rightarrow \mathcal{E}'(U)$ is a cofibration of spectra for each $U\in$ Sch/$k$.
\item Fibrations are not defined objectwise, but rather are defined using the right lifting property.  
\end{itemize}
In addition to these classes of morphisms we will define a stronger version of weak equivalence.
\begin{definition}  
A morphism $f: \mathcal{E}\rightarrow \mathcal{E}'$ is a global weak equivalence provided $f(U):\mathcal{E}(U)\rightarrow \mathcal{E}'(U)$ is a weak equivalence (of spectra) for each $U\in$ Sch/$k$.
\end{definition}

Now that we have imposed our model structure on this category, we can focus on objects with certain nice properties.  
\begin{definition}
In any category with a suitable model category structure, a fibrant object is an object whose unique map to the terminal object is a fibration.
\end{definition}
\begin{definition}
A fibrant replacement of $\mathcal{E}$ in a suitable model category is a trivial cofibration $\mathcal{E}\rightarrow \mathcal{E}'$ where  $\mathcal{E}'$ is a fibrant object.
\end{definition}
One way in which the fibrant objects of a closed model category can be described is by having a right lifting property with respect to any trivial cofibration in the category. The right lifting property makes fibrant objects the best objects on which to define homotopy groups. Choosing fibrant objects (or their fibrant replacements) and imposing the correct topology (cdh-topology in our case) is the way that we ensure that our homotopy groups inherit nice properties

The process of obtaining a fibrant replacement can be done functorially using the small object argument.  Since we can find a fibrant replacement functorially, we will do so (for the local injective model structure).  
\begin{definition}
We define $\mathbb{H}_t(-,\mathcal{E})$ to be the fibrant replacement of the sheaf $\mathcal{E}$ (in the local injective model structure).
\end{definition}
As stated above, this can be done using a functor, call it $\mathbb{H}_t(-,-)$, where  \[\mathbb{H}_t(-,-):\text{Shv}_t\text{(Sch/}k\text{,Spectra)}\rightarrow\text{Shv}_t\text{(Sch/}k\text{,Spectra)}\] such that \[\mathcal{E}(-)\mapsto\mathbb{H}_t(-,\mathcal{E}),\] where $t$ is a topology on Sch/$k$.  (When the topology is clear from the context $t$ may be dropped from the notation.)

\begin{remark}The fibrant objects in the category of presheaves on a site are just the sheaves themselves.  You may think of the process of obtaining a fibrant replacement of a presheaf as just performing the t-sheafification of that presheaf.    
A familiar example would be the presheaves on the Zariski site.  The Zariski-fibrant replacement of some presheaf $\mathcal{E}$ is then just the traditional sheafification of $\mathcal{E}$. 
\end{remark}

While replacing these objects with certain fibrant objects will yield very nice properties when we eventually pass to the homotopy groups, in general the fibrant replacement does not yield the same spectrum as the original sheaf.  There are certain objects in this category that are replaced without any such difficulties.  
\begin{definition}
A presheaf of spectra $\mathcal{E}$ on a site $(C,t)$ is called quasifibrant if the map to it's fibrant replacement is a global weak equivalence, i.e., the map $\mathcal{E}(U)\rightarrow\mathbb{H}_t(U,\mathcal{E})$ is a weak equivalence for all $U$ in $C$.  
\end{definition}

For the topologies that we will work with, quasifibrant presheaves are given the following terminology.  
\begin{definition}
Let $\mathcal{E}$ be a quasifibrant presheaf of spectra on a site $(C,t)$ such that the topology $t$ is generated by a complete regular bounded $cd$-structure.  We say then that $\mathcal{E}$ satisfies $t$-descent, or descent for the $t$-topology.  
\end{definition}
We will spend a fair amount of time in the later sections talking about sheaves that satisfy various types of descent.  Before we do so, we will continue with more preliminaries.  

We should note, that there is an analog of this construction for presheaves taking values in a different category, namely cochain complexes.  In the same way as before, we can define a category of sheaves now into cochain complexes.  
\begin{definition}
Let Shv$_t$(Sch/$k$,Cochain Complexes) denote the category of sheaves on Sch/$k$ with topology $t$ into the category of cochain complexes.  
\end{definition}
What is slightly different here is what is used for a fibrant replacement of such a sheaf $A^{\bullet}$ of cochain complexes.  
\begin{definition}
A presheaf of cochain complexes $A^{\bullet}$ on a site $(C,t)$ is called quasifibrant if the natural map $A^{\bullet}(U)\rightarrow R\Gamma(U,A^{\bullet})$ is a quasi-isomorphism for each object $U$ of $C$.  
\end{definition}
Again, we will adopt the descent notation for quasifibrant sheaves of cochain complexes.   
\begin{definition}
Let $A^{\bullet}$ be a quasifibrant presheaf of cochain complexes on a site $(C,t)$ such that the topology $t$ is generated by a complete regular bounded $cd$-structure.  We say then that $A^{\bullet}$ satisfies $t$-descent, or descent for the $t$-topology.  
\end{definition}

This construction for cochain complexes agrees with our original construction for spectra in the following way.  Let $A^{\bullet}\rightarrow|A^{\bullet}|$ be the Eilenberg-Maclane functor that sends cochain complexes to spectra and presheaves of cochain complexes to presheaves of spectra.  This functor sends quasi-isomorphisms to weak equivalences and, as we will see below, their homotopy/cohomology groups coincide.  It is because of this that we will treat these notions of descent as interchangeable, and only try to distinguish them during constructions of well known (co)homology theories.  

\subsection{Hypercohomology and sheaf cohomology}  
The standard way to define a sheaf type cohomology theory for a given sheaf $F$ is to find an injective resolution of $F$, \[0\rightarrow F\rightarrow I^1\rightarrow I^2\rightarrow...\] and then apply the global sections functor, $\Gamma(X,-)$ to this sequence, \[ \Gamma(X,F)\rightarrow \Gamma(X,I^1)\rightarrow \Gamma(X,I^2)\rightarrow\Gamma(X,I^3)\rightarrow...\] to get a chain complex.  Taking the homology of this complex yields sheaf cohomology, $H^*(X,F)$.  In our scenario, we have two (practically indistinguishable) sheaves.  Namely, sheaves of spectra and sheaves of complexes.  The problem we have is how to find an injective resolution of a spectrum or a complex.  The short answer is, there is not a good way to find a injective resolution.  

The way we deal with this problem is to use hypercohomology.  Hypercohomology takes a whole complex or spectrum of sheaves and, where it would not be possible or at least it would be very restrictive which complexes you could use, defines something related to an injective resolution.  The solution is to find an injective complex that is quasi-isomorphic to our given complex and use this in our definition.  
\begin{definition}\label{hypercohomology}
Let $A^{\bullet}$ be a presheaf of complexes, and let $I^{\bullet}$ be an injective complex such that $I^{\bullet}$ is quasi-isomorphic to $A^{\bullet}$.  Let $X\in$ Sch/$k$.  We define the hypercohomology of $X$ with coefficients in $A^{\bullet}$ as follows
\[\mathbb{H}^n(X,A^{\bullet})=H^n(\Gamma(X,I^{\bullet})).\]
\end{definition}
This type of construction for hypercohomology is more intuitive than some of the others.  And, much later, for the most general constructions this is where we will gain our intuition.  

For our scenario involving fibrant replacements in a particular topology, however, we will adopt a different definition of hypercohomology.  Applying homotopy functors to these specific presheaves of spectra will yield our first definition of some type of fibrant version of a cohomology theory. 
\begin{definition}
Let $\mathcal{E}$ be a (pre)sheaf of spectra on Sch/$k$, and let $\mathbb{H}_t(-,\mathcal{E})$ be it's fibrant replacement.  Let $X\in$ Sch/$k$.  The hypercohomology of $X$ with coefficients in $\mathcal{E}$ is defined as follows
\[\mathbb{H}^n_t(X,\mathcal{E})=\pi_{-n}(\mathbb{H}_t(X,\mathcal{E})).\]
\end{definition}
Applying this to the cdh-topology, we obtain our first definition of cdh-cohomology.  

According to Thomason (\cite{thomason1985algebraic}, p. 532), this previous definition yields a type of hypercohomology.  So, it should be analogous to our previous definition \ref{hypercohomology}. While some authors use the homotopy groups of a specific spectrum as the definition for hypercohomology, others use our complex of injectives and still others use a derived category construction.  These definitions all agree with one another, and the reason for this is a spectral sequence argument.  In order to avoid such arguments, we will simply use the previous definitions as our definition of hypercohomology and $t$-cohomology separately in theory building, and interchangeably in practice.  
\begin{notation}
Let $\mathcal{E}$ be a (pre)sheaf of spectra on (Sch/$k,t)$ defining some cohomology theory $H^*$, and let $\mathbb{H}_{t}(-,\mathcal{E})$ be it's fibrant replacement.  For some $X\in$ Sch/$k$ the $t$-fibrant version of the cohomology theory $H^*$, written as cdh-fibrant $H^*$ cohomology is given by $\mathbb{H}^n_t(X,\mathcal{E})$, and can sometimes be written as $\mathbb{H}^n_t(X,H^*)$.
\end{notation}
So, cdh-fibrant versions of our various generalized cohomology theories are simply defined by the hypercohomology groups, which are the homotopy groups of their fibrant replacements.  

Each of our hypercohomology constructions is really a generalization of the sheaf cohomology groups construction. In some ways, this is the right way to extend our construction, because these two constructions agree where they overlap.  Namely when we have a complex where all of the data is concentrated at the zero-ith level, the sheaf cohomology agrees with the hypercohomology. This can be seen in several different ways, but in particular by looking at the $t$-cohomology definition we have the following idea that is in part explained by the scholium of great enlightenment (\cite{thomason1985algebraic}, 5.32).  
\begin{idea}
Let $\mathcal{E}$ be the presheaf of Eilenberg-MacLane spectrum associated to the presheaf of (simplicial) groups $A^{\bullet}$ (this is basically looking at a presheaf of complexes where all of the data is concentrated at the zero-ith level).  Then $\mathbb{H}_t(-,\mathcal{E})$, the fibrant replacement for $\mathcal{E}$, is also an Eilenberg-Maclane spectrum.  In this case, $\mathbb{H}_t(-,\mathcal{E})$ is the Eilenberg-MacLane spectrum associated to the right derived functors of the global sections of $A^{\bullet}$, or more concisely $R\Gamma(-,A^{\bullet})$.  So, we have that  \[\pi_{-n}(\mathbb{H}_t(X,\mathcal{E}))=R^n\Gamma(X,A^{\bullet})=H_t^n(X,A^{\bullet})\] which is just the sheaf cohomology  (where switching the index from the subscripts to superscripts adds a negative sign).
Combining this with our original definition of $t$-cohomology we have that,  \[\mathbb{H}_t^n(X,\mathcal{E})=\pi_{-n}(\mathbb{H}_t(X,\mathcal{E}))=R^n\Gamma(X,A^{\bullet})=H_t^n(X,A^{\bullet}).\]
This proves equality in the case where we have a well defined correspondence between a sheaf of spectra $\mathcal{E}$ and a sheaf of cochain complexes $A^{\bullet}$ .  
\end{idea}
\begin{notation}
This previous idea gives us agreement with the traditional sheaf cohomology groups, provided we specify the topology.  In notation, we will denote these groups as $H_{t}^n(X,A^{\bullet})$.  For the most part we will try to distinguish between the two scenarios, i.e. when we are dealing with hypercohomology we will use $\mathbb{H}$, and $H$ otherwise.  But since they agree in many scenarios, we may slip up from time to time.  
\end{notation}

We will continue our discussion with some important properties of these cohomologies, such as a Mayer-Vietoris sequence.  
\begin{definition}
A presheaf $\mathcal{E}$ of spectra on the category Sch/$k$ is said to satisfy the Mayer-Vietoris property for a Cartesian square of schemes  
\[\begin{CD}
A  @>>> B\\
@VVV @VVV\\
C @>>> D
\end{CD}\] if applying $\mathcal{E}$ to this square yields a homotopy Cartesian square in the category of spectra.  
\end{definition}
We will occasional call this the MV-property for brevity.  It will be important when the MV-property holds for some distinguished class of squares.  
\begin{definition}
The presheaf $\mathcal{E}$ satisfies the MV-property for a class of squares $P$ if it satisfies the MV-property for each square in the class.
\end{definition}
Tying this into our previous work, we have the following result.  
\begin{theorem}\label{mvweak}
Let $C$ be a category with a complete bounded $cd$-structure called $P$, and let $t$ be a topology induced by the $cd$-structure $P$.  Then a presheaf of spectra $\mathcal{E}$ on $C$ is quasifibrant with respect to $t$ if and only if $\mathcal{E}$ satisfies the MV-property for the class of squares given by $P$.  
\end{theorem}
\begin{proof}
See \cite{cortinas2005cyclic} theorem 3.4.
\end{proof}
\begin{corollary}
In the scenario from the previous theorem, if $\mathcal{E}$ satisfies the MV-property for the class of squares given by $P$, then $\mathcal{E}$ satisfies $t$-descent.  
\end{corollary}
\begin{remark}
This previous construction holds if we replace all spectra and sheaves of spectra, denoted by $\mathcal{E}$, with cochain complexes and sheaves of cochain complexes, denoted by $A^{\bullet}$.  By this we mean the following.
\end{remark}
\begin{theorem}\label{mvchain}
Let $C$ be a category with a complete bounded $cd$-structure called $P$, and let $t$ be a topology induced by the $cd$-structure $P$.  Then a presheaf of complexes $A^{\bullet}$ on $C$ is quasifibrant with respect to $t$ if and only if $A^{\bullet}$ satisfies the MV-property for the class of squares given by $P$.  Again, in this scenario, we say that $A^{\bullet}$ satisfies $t$-descent.  
\end{theorem}

The MV-property as defined above turns out to be the property that will allow us to perform computations.  The reason for this is, as the notation suggests, the MV-property yields a Mayer-Vietoris sequence of cohomology groups.  
\begin{theorem}
Let the following be an arbitrary square in a complete bounded $cd$-structure $P$ such that $P$ generates a topology $t$.
\[\begin{CD}
A  @>>> B\\
@VVV @VVV\\
C @>>> D
\end{CD}\] Then for any presheaf of spectra $\mathcal{E}$ (resp. simplicial sets, cochain complexes), there is a Mayer-Vietoris sequence of cohomology groups given by 
\[ ...\rightarrow\mathbb{H}_t^{n-1}(A,\mathcal{E})\rightarrow\mathbb{H}_t^n(D,\mathcal{E})\rightarrow\mathbb{H}_t^n(B,\mathcal{E})\oplus\mathbb{H}_t^n(C,\mathcal{E})\rightarrow\mathbb{H}_t^n(A,\mathcal{E})\rightarrow\mathbb{H}_t^{n+1}(D,\mathcal{E})\rightarrow...\]
\end{theorem}
\begin{proof}
These cohomology groups are defined by the functor $\mathbb{H}_t^n(-,\mathcal{E})=\pi_{-n}(\mathbb{H}_t(-,\mathcal{E}))$.  Trivially, we know that the sheaf $\mathbb{H}_t(-,\mathcal{E})$ satisfies $t$-descent (because the map $\mathbb{H}_t(-,\mathcal{E})\rightarrow \mathbb{H}_t(-,\mathcal{E})$ is a global weak equivalence).  Therefore, by \ref{mvweak}, we know that $\mathbb{H}_t(-,\mathcal{E})$ satisfies the MV-property for every square in $P$.   Applying $\mathbb{H}_t(-,\mathcal{E})$ to our given square, we have that 
\[\begin{CD}
\mathbb{H}_t(A,\mathcal{E})   @<<< \mathbb{H}_t(B,\mathcal{E}) \\
@AAA @AAA\\
\mathbb{H}_t(C,\mathcal{E})  @<<< \mathbb{H}_t(D,\mathcal{E}) 
\end{CD}\]is a homotopy Cartesian square.  The map $\mathbb{H}_t(D,\mathcal{E})\rightarrow \mathbb{H}_t(B,\mathcal{E})\times\mathbb{H}_t(C,\mathcal{E})$ is a fibration with fiber homotopy equivalent to the loop space spectrum of $\mathbb{H}_t(A,\mathcal{E})$, $\Omega\mathbb{H}_t(A,\mathcal{E})$.  Using the inherited long exact sequence of homotopy groups for fibration sequences, we have that 
\[ ...\rightarrow\pi_{-(n-1)}(\mathbb{H}_t(A,\mathcal{E}))\rightarrow\pi_{-n}(\mathbb{H}_t(D,\mathcal{E}))\rightarrow\pi_{-n}(\mathbb{H}_t(B,\mathcal{E}))\oplus\pi_{-n}(\mathbb{H}_t(C,\mathcal{E}))\rightarrow\hspace{1 in}\] \[\hspace{2.5 in}\rightarrow\pi_{-n}(\mathbb{H}_t(A,\mathcal{E}))\rightarrow\pi_{-(n+1)}(\mathbb{H}_t(D,\mathcal{E}))\rightarrow...\]
Thus proving the claim.
\end{proof}

Turning our focus back to cdh-site, we will first summarize everything up to this point.  In the case of the cdh-cohomology, our site is (Sch/$k$,cdh).  Now substituting the topology into our cohomology constructions, we define the cdh-cohomology groups in the following way.  
\begin{definition}Let $X$ be a scheme over a field $k$.  Then the cdh-cohomology groups of $X$ with coefficients in a sheaf of spectra $\mathcal{E}$ are given by 
\[\mathbb{H}^n_{cdh}(X,\mathcal{E})=\pi_{-n}(\mathbb{H}_{cdh}(X,\mathcal{E})),\]
where $\mathbb{H}_{cdh}(X,\mathcal{E})$ is the locally injective fibrant replacement sheaf in the cdh-site.  
\end{definition}
We have also shown that these hypercohomology groups may coincide with the definition of the sheaf cohomology groups in some cases.  Using the MV-property and it's associated sequence is where the utility of cdh-cohomology becomes clear.  As a corollary to the previous theorem, we have the following. 

Suppose $X$ is a singular variety over characteristic zero.  From Hironaka's desingularization theorem, we find a smooth resolution $Y$ from a finite sequence of blowups along smooth centers.  This is summarized in the following abstract blowup square.  We have 
\[\begin{CD}
E @>>> Y\\
@VVV @VVV\\
Z @>>> X
\end{CD} \]  
where $Y$ and $Z$ are smooth and $E$ is a union of smooth hypersurfaces of $Y$ having only normal crossings. 
\begin{corollary}\label{mvsequence}
For a presheaf of spectra $\mathcal{E}$ on the cdh-site, and in the scenario above of the resolution of the singularities of $X$, we have the following Mayer-Vietoris sequence of cohomology groups, 
\[ ...\rightarrow\mathbb{H}_{cdh}^{n-1}(E,\mathcal{E})\rightarrow\mathbb{H}_{cdh}^n(X,\mathcal{E})\rightarrow\mathbb{H}_{cdh}^n(Y,\mathcal{E})\oplus\mathbb{H}_{cdh}^n(Z,\mathcal{E})\rightarrow\mathbb{H}_{cdh}^n(E,\mathcal{E})\rightarrow\mathbb{H}_{cdh}^{n+1}(X,\mathcal{E})\rightarrow... .\]
\end{corollary}
As stated before, this is our most important computational tool.  The reason is that, because of Hironaka's theorem, we always have a resolution in characteristic 0.  So, computing the cdh-cohomology of a singular variety (which does not sound so easy) is reduced to computing cdh-cohomology of these smooth schemes (which is maybe a little easier).  So we can easily reference this fact later, we will write this in the following way.  
\begin{corollary}\label{notsmoothsmooth}
According to corollary \ref{mvsequence}, since the cdh-cohomology groups of any singular variety over characteristic zero fits into a long exact sequence with the cdh-cohomology groups of nonsingular varieties, the cdh-cohomology groups of any scheme in Sch/$k$ is determined by cdh-cohomology groups of schemes in Sm/$k$.  
\end{corollary}

Furthermore, for any sheaf that satisfies cdh-descent, we can do the same thing.  
\begin{example} As a preview of the next section, let $\textbf{HP}(-/k)$ be the sheaf of spectra defining periodic cyclic homology, $HP_n(-/k)$, that is computed over the field $k$.  By a result in the next section $\textbf{HP}(-/k)$ satisfies cdh-descent.  This means that the fibrant replacement, $\mathbb{H}_{cdh}(-,\textbf{HP}(-/k))$, of $\textbf{HP}(-/k)$ is a global weak equivalence. Finally, this means that for a scheme $X$, \[HP_{-n}(X/k)=\pi_{-n}(\textbf{HP}(X/k))=\pi_{-n}(\mathbb{H}_{cdh}(X,\textbf{HP}(-/k)))=\mathbb{H}_{cdh}^n(X,\textbf{HP}(-/k)).\]
So, since periodic cyclic homology coincides with it's cdh version, it shares some of it's nice properties, namely, the Mayer-Vietoris sequence from the resolution of $X$.  Therefore, as before, computing the periodic cyclic homology of a singular variety is reduced to computing periodic cyclic homology of smooth schemes.  
\end{example}

\begin{remark}
While the previous statement in corollary \ref{notsmoothsmooth} is technically correct, we should note that our justification on the way to that result is not.  We recall from Hironaka that, the center of the resolution, which we are calling $Z$, is more correctly written as a sequence of centers.  When taken together, these centers should look something like the singular locus, which in general is not a smooth subvariety.  It is, however, a closed subvariety and hence it is of smaller dimension.  We can therefore iterate this process until we resolve the singularities in $Z$.  

Also, while the exceptional fiber, $E$, may not be smooth, normal crossings are in some ways the easiest type of singularities to resolve.  Furthermore, since $E$ is composed of hyperplanes, we know that it also has lower dimension.  Hence again, by induction, we would again be able to iterate the resolution process until all pieces of the blowup squares are resolved into smooth varieties.  

Since we are able to construct iterated resolutions for each of the pieces associated to these blowup squares, we will eventually arrive at some hybrid diagram of Mayer-Vietoris sequences where the value of the cdh-cohomology on each singular scheme is still determined by it's values on smooth schemes.  Thus confirming our result in the previous corollary.  
\end{remark}

\newpage
\section{K-theory and cdh-fibrant cyclic homology}
In this section, we will begin by defining all of the relevant spectra, or complexes.  From there, we will look at the cdh-fibrant versions of each of these cohomology theories.  
\subsection{The K-theory and Homotopy K-theory spectra}
We begin with Quillen's definition for the K-groups of a ring.  Quillen's definition of K-theory involves the homotopy groups of a topological space.  There are several definitions we could give for this space, but we will try to be concise and use a more ``functorial" construction. 
\begin{definition}
An H-space $X$ is a topological space with a map $\varphi:X\times X\rightarrow X$ such that there exists an identity element, i.e. there is an $e$ such that $\varphi(e,x)=\varphi(x,e)=x$.  
\end{definition}
Other algebraic properties will be important up to homotopy.  We give the following example.  If $X$ is homotopy commutative then the map $\varphi:X\times X\rightarrow X$ is commutative up to homotopy, i.e. $\varphi(x,y)$ is homotopic to $\varphi(y,x)$.  
\begin{definition}
Let $X$ be a homotopy associative, homotopy commutative H-space. A group completion of $X$ is another H-space called $Y$ and an H-space map $X\rightarrow Y$ with the following properties.\\
$\bullet$ $\pi_0(Y)$ is the group completion of the commutative monoid $\pi_0(X)$ (i.e. the \\ Grothendieck group of $\pi_0(X)$).  \\
$\bullet$ The homology ring $H^*(Y,S)$ with coefficients in $S$ is isomorphic to the localization $\pi_0(X)^{-1}H^*(X,S)$ of $H^*(X,S)$ by the natural map, for all commutative rings $S$. 
\end{definition}
\begin{definition}
Let $BGL_n{R}$ be the classifying space of the group of invertible $n\times n$ matrices with entries in $R$. We define $BGL(R)^+$ to be the basepoint component of the group completion of the H-space $\coprod_{n=0}^{\infty}BGL_n{R}$.  
\end{definition}
This is the topological space that we will use in our construction of the K-groups.  All that remains now is to evaluate this space using the homotopy group functors.
\begin{definition}
Let $R$ be a ring.  We define the K-groups of R to be $K_n(R)=\pi_n(BGL(R)^+)$ for $n\in\mathbb{N}$.
\end{definition}
Notice that this definition only defines the non-negative K-groups, we will see that by using a spectrum, one such definition of negative K-groups will be an obvious consequence.  
\begin{note}There is a way to define the negative K-groups without using spectra.  This construction comes from Bass, and is as follows. For a ring $R$, we define $K_{-1}(R)$ to be the cokernel of the map \[K_0(R[t])\oplus K_0(R[t^{-1}])\rightarrow K_0(R[t,t^{-1}]).\] We could then continue inductively to define all of the negative K-groups.  Any definition using a spectrum has of course been constructed so that it agrees with this more natural definition.  
\end{note}
The way we will proceed from here is by using a method of Gersten and Wagoner which uses the following construction of the suspension of ring.  
\begin{definition}
Let $R$ be a ring, and let the cone ring $C(R)$ be the ring consisting of row-and-column finite matrices.  The finite matrices $M(R)\subset C(R)$ form a 2-sided ideal.  We define the suspension ring of $R$ to be the quotient, $S(R)=C(R)/M(R)$.   
\end{definition}
\begin{proposition}[Gersten and Wagoner]
For a ring $R$, the topological space $K_0(R)\times BGL(R)^{+}$ is homotopy equivalent to $\Omega BGL(S(R))^+$ where $S(R)$ is the suspension ring of $R$.  This further implies that $K_n(R)\cong K_{n+1}(S(R))$ for all $n\geq0$.  
\end{proposition}
It follows from the definition of spectra (actually the definition of $\Omega$-spectra, the version of spectra where we only have weak equivalences) that we can define our sequence of spaces inductively as follows.  
\begin{definition}
We define the sequence of spaces $K_0(S^i(R))\times BGL(S^i(R))^{+}$ that form a non-connective spectrum to be the K-theory $\textbf{K}(R)$ (where $S^i(R)$ denotes the $i$th iterated suspension of $R$).  
\end{definition}
By connective and non-connective spectrum, we simply mean that, for a connective spectrum the negative homotopy groups are all zero, and for a non-connective spectrum the negative homotopy groups are not necessarily zero.  
\begin{proposition}
The algebraic K-groups of a ring $R$ are given by the following,
\[\pi_n(\textbf{K}(R))\cong K_n(R)\] for all $n\in\mathbb{Z}$, where these groups are equivalent to any previous definition of the algebraic K-groups.  
\end{proposition}

We now move on to the construction of Homotopy K-theory.  In some sense, the previously defined K-groups are not homotopy invariant.  The goal with homotopy K-theory is to define some related K-theory that is homotopy invariant.  Before we can do anything, we first must define what it means to be homotopy invariant.  The primary resource for this construction is \cite{weibel2013k}.

\begin{definition}
A functor $F$ from rings to sets is homotopy invariant if $F(R)\cong F(R[t])$ for every $R$.  Similarly, a functor from rings to CW complexes is called homotopy invariant if for every ring $R$ the natural map $R\rightarrow R[t]$ induces a homotopy equivalence $F(R)\simeq F(R[t])$.  
\end{definition}

\begin{definition}
Let $R$ be a ring.  The coordinate rings of the standard simplices form a simplicial ring $R[\Delta^\cdot]
$. The simplicial structure can be described by the following diagram.  

\[R\leftleftarrows R[t_1]\overset{\leftleftarrows}{\leftarrow} R[t_1,t_2]\overset{\leftleftarrows}{\leftleftarrows}...R[t_1,...,t_n]...\]

\noindent with $R[\Delta^n]=R[t_0,t_1,...,t_n]/\left(\sum t_i=1\right)\cong R[t_1,...,t_n].$ The face maps $\partial_i$ are given by: $\partial_i(t_i)=0;$ $\partial_i(t_j)$ is $t_j$ for $j<i$ and $t_{j-1}$ for $j>i$.  Degeneracies $\sigma_i$ are given by; $\sigma_i(t_i)=t_i+t_{i+1}$; $\sigma_i(t_j)$ is $t_j$ for $j<i$ and $t_{j+1}$ for $j>i$.  
\end{definition}

The rest of the construction is straight forward.  We simply use $R[\Delta^\cdot]$ in place of $R$ and construct the homotopy groups of its K-theory spectrum.  

\begin{definition}
We denote $\textbf{KH}(R)$ as the geometric realization of the simplicial spectrum $\textbf{K}(R[\Delta^\cdot])$.  We define the homotopy K-groups to be \[KH_n(R)=\pi_n(\textbf{KH}(R))\] for all $n\in\mathbb{Z}$.
\end{definition}

The homotopy K-groups have several nice properties, in particular, $\textbf{KH}$ satisfies cdh-descent.  We will look at this in more depth later, but first, we will explore some important properties of K-theory, Homotopy K-theory, and their relationship.  

\begin{proposition}
Homotopy K-theory is homotopy invariant.  
\end{proposition}
\begin{proof}
This is a specific example of what is called the homotopization of a functor.  Essentially, with the simplicial ring $R[\Delta^\cdot]$ we can define a simplicial homotopy using the inclusion map $R[\Delta^\cdot]\subset R[x][\Delta^\cdot]$.  Simplicial homotopies, when taking the geometric realization, turn into topological homotopies.  Therefore, we have that $KH_n(R)\cong KH_n(R[x])$.  See \cite{weibel2013k} lemma IV.11.5.1. 
\end{proof}

One of the main purposes of Homotopy K-theory is the fact that is satisfies excision.  While other variants of K-theory have nice properties, they may not satisfy excision.  In particular, the original K-theory does not satisfy excision.  
\begin{theorem}
Let $I$ be an ideal in a unital ring $R$. Then $\textbf{KH}(I)\rightarrow\textbf{KH}(R)\rightarrow\textbf{KH}(R/I)$ is a homotopy fibration.  Therefore, at the level of homotopy groups, there is a long exact sequence of homotopy groups for fibrations given by the following, 
\[...\rightarrow KH_{n+1}(R/I)\rightarrow KH_n(I)\rightarrow KH_n(R)\rightarrow KH_n(R/I)\rightarrow ...\]
\end{theorem}
\begin{proof}
See \cite{weibel2013k} theorem IV.12.4.  
\end{proof}
From this excision theorem for Homotopy K-theory, we have the following corollaries.  
\begin{corollary}
Let $R$ be a ring, and $I\subset R$ a nilpotent ideal.  Then $\textbf{KH}(I)$ is contractible and $KH_n(I)=0$ for all $n$ (and therefore $KH_n(R/I)\cong KH_n(R)$).   
\end{corollary}

\begin{corollary}
Let $R\rightarrow S$ be a map of commutative rings, where an ideal $I\subset R$ is mapped to an ideal $J\subset S$ with $I\cong J$.   Then there is a Mayer-Vietoris sequence given by. 
\[...KH_{n+1}(S/I)\rightarrow KH_n(R)\rightarrow KH_n(R/I)\oplus KH_n(S)\rightarrow KH_n(S/I)\rightarrow...\] 
\end{corollary}

Later we will focus on singular schemes.  As has been hinted all along, what happens with singular schemes is in some cases determined (or at least influenced) by what happens on smooth schemes.  We have the following relationship between K-theory and Homotopy K-theory on regular rings. 
\begin{proposition}\label{kh=k}
If $R$ is regular Noetherian, then $\textbf{K}(R)\simeq\textbf{KH}(R)$, meaning that when we pass to homotopy groups, $K_n(R)=KH_n(R)$ for all $n$.    
\end{proposition}
\begin{proof}
This boils down to a spectral sequence argument which can be found in \cite{weibel2013k}, theorem IV.12.3 and corollary IV.12.3.1.  
\end{proof}
Recall that if $R$ is a regular ring, then the affine scheme $X=$ Spec$(R)$ is a smooth scheme.  Therefore the previous statement is really a statement about smooth schemes.  

\begin{remark}
There is a way to define both K-theory and homotopy K-theory for schemes.  The easiest way to this is to use Quillen's $Q$-construction for determining the K-theory of categories (for a scheme $X$ we would take the category of algebraic vector bundles over $X$, $VB(X)$). While we chose not to do this here, we just note that the construction is in many respects completely analogous and we simply rephrase many of the objects in the algebraic geometry setting.  For our purposes, we simply need to know that in the case of affine schemes, everything is as expected.  We have that, for $X=$ Spec$(R)$, the K-groups are given by $K_n(X)=K_n(R)$ and $KH_n(X)=KH_n(R)$
\end{remark}

We really will not work with the K-theory of more general schemes, so, we will not take the time give the construction of K-theory for these objects.  Instead we will continue with other constructions, for which we will need the more general scheme theoretic constructions.  These constructions for general schemes are completely analogous to the K-theory constructions for general schemes.  So for any further study that requires the more general definitions, we will build off of the constructions in the following sections.

\subsection{Variations of cyclic homology}The main reference for these constructions is \cite{loday1997cyclic} while their properties are more thoroughly presented in \cite{weibel1995introduction}.  The proper way to define cyclic homology and it's relevant variations, namely negative cyclic homology and periodic cyclic homology, is via a chain complex.  We will do this by first defining a bicomplex.  

Let $C=(C_n)_{n\geq 0}$ be a cyclic module with face maps $d_i$, degeneracy maps $s_j$, and cyclic operator $t$. We define the extra degeneracy map to be $s=(-1)^nts_n$.  From these we put \[b=\sum_{i=0}^n(-1)^id_i,\ \ \ N=1+t+t^2+...+t^n,\ \text{ and }\ B=(-1)^{n+1}(1-t)sN.\] 

\begin{definition}
The periodic bicomplex, $CCP$, is given by

\vspace{5 mm}

$CCP:=$
\[\begin{CD}
@. @VbVV @VbVV @VbVV @VbVV @.\\
... @<B<< C_3 @<B<< C_2 @<B<< C_1 @<B<< C_0 @.\\
@. @VbVV @VbVV @VbVV @. @.\\
... @<B<< C_2 @<B<< C_1 @<B<< C_0 @. @.\\
@. @VbVV @VbVV @. @. @.\\
... @<B<< C_1 @<B<< C_0 @. @. @.\\
@. @VbVV @. @. @. @.\\
... @<B<< C_0 @. @. @. @.\\
{\rm Column}\#@. -1 @. 0 @. 1 @. @. @. \\
\end{CD}
\]

\end{definition}
The vertical boundary map, $b$, is used in all definitions of cyclic homology.  The horizontal map, called Connes' boundary map $B$, is what is new and required for this construction of cyclic homology.  The reason that this is a bicomplex, and everything behaves well when we will eventually go to (something like) the homology groups is the relations on the two boundary maps given by,
\[b^2=B^2=bB+Bb=0.\]

\begin{note}
Some authors call this a double complex and reserve the term bicomplex for that other complex that (equivalently) defines cyclic homology.  
\end{note}

Each of the cyclic, negative cyclic, and periodic cyclic homology theories can be constructed by using all or part of this new bicomplex.  

\begin{definition}The normal cyclic bicomplex, $CC$,  is given by replacing all of the negative index columns of the periodic bicomplex by the 0 complex.  
\end{definition}

\begin{definition}The negative bicomplex $CCN$ is given by replacing all the columns whose index is $> 0$ by the 0 complex. 
\end{definition} 

We now have three bicomplexes, $CC$, $CCN$, and $CCP$, the relations between these complexes is summarized by the following exact sequence of complexes.  
\[0\rightarrow CCN\rightarrow CCP\rightarrow CC[0,2]\rightarrow 0\]
The reason for this shift is the nonzero overlap of $CC$ and $CCN$.  Shifting the horizontal indices by two ensures that the overlap is 0, and hence the sequence is exact.  

Recall that the way that we construct the homology of a bicomplex is to first form the total complex of the bicomplex, and then take the homology of that complex.  

\begin{definition}
We define the complex Tot $CC$ as the complex whose term of degree $n$ is $\bigoplus_{p+q=n}(CC_{pq})$

\noindent Similarly the complex Tot $CCN$ is such that its term of degree n is is $\prod_{p+q=n}CCN_{pq}$. 

\noindent Finally, we define complex Tot $CCP$ as the complex whose term of degree $n$ is \\ $\prod_{p+q=n}CCP_{pq}$. 
\end{definition}

\begin{remark}
Notice for $CC$ we have defined the total complex to be a direct sum.  In the case of cyclic homology, where the bicomplex is indexed by $\mathbb{N}\times\mathbb{N}$, there are no problems with this definition.  This is because for each $n\in \mathbb{Z}$ there are only a finite number of indices $(p,q)\in\mathbb{N}\times\mathbb{N}$ such that $p+q=n$. 

However, with the other bicomplexes, this is no longer the case, there will be infinitely many $(p,q)\in\mathbb{Z}_{\leq 0}\times\mathbb{N}$, and also infinitely many $(p,q)\in\mathbb{Z}\times\mathbb{N}$
 such that $p+q=n$. In order to avoid certain pathologies we have done a slight modification, we have used a direct product instead of a sum.  
\end{remark} 

We can now define the negative and periodic cyclic homologies.  
\begin{definition}
Let $C$ be a cyclic module. Then the cyclic homology, negative cyclic homology, and periodic cyclic homology of $C$ are respectively
\[HC_n(C) := H_n({\rm Tot} CC)\text{ and }HN_n(C) := H_n({\rm Tot} CCN)\text{ and }HP_n(C) := H_n({\rm Tot} CCP),\ n\in\mathbb{Z}. \]
\end{definition}

Most of the time, we are really only concerned with $C(A)=(A^{\otimes n})_{n\geq 1}$ where $A$ is a $k$-algebra, and the tensor product is over $k$.  In this case, we say that we are really computing the cyclic homology, and it's variations, of the algebra $A$.  When $C=C(A)$, where $A$ is a $k$-algebra, we write $HC_*(A/k)$, $HN_*(A/k)$, and  $HP_*(A/k)$ instead of $HC_*(C(A))$, $HN_*(C(A))$, and $HP_*(C(A))$.  The reason we have written $A/k$ inside each of the previous cyclic homology functors is that it will be important for us to keep track of the field $k$ from the tensor product. 

\begin{note}
This definition using the nonstandard bicomplex from above really only works when the cyclic module has degeneracies.  This is equivalent to saying that the algebra $A$, defining the cyclic module $C(A)$, is unital.  
\end{note}

We will continue now with some immediate computations and properties of our three variations of cyclic homology.  
\begin{proposition}
HP is actually periodic of period 2, i.e $HP_0(C)=HP_{2n}(C)$ and $HP_1(C)=HP_{2n+1}(C)$ for all $n$.   
\end{proposition}

\begin{proposition}
For $A=k$ we have that 
\[HC_n(k/k)=
\begin{cases}
 0, & \text{if }n<0 \\
 0, & \text{if }n=2i-1>0 \\
 k, & \text{if }n=2i\geq 0
\end{cases},\]
\[HN_n(k/k)=
\begin{cases}
 0, & \text{if }n>0 \\
 0, & \text{if }n=2i-1<0 \\
 k, & \text{if }n=2i\leq 0
\end{cases},\]
and \[HP_n(k/k)=
\begin{cases}
 0, & \text{if }n=2i-1 \\
 k, & \text{if }n=2i
\end{cases}.\ \ \ \ \ \]   
\end{proposition}

From the previous result we have evidence, and again from the definitions, we conclude that for $n\leq0$, $HP_n\cong HN_n$.  More generally, returning to our exact sequence of bicomplexes, \[0\rightarrow CCN\rightarrow CCP\rightarrow CC[0,2]\rightarrow 0,\]
we have the following induced long exact sequence.  
\begin{proposition}
For any cyclic module $C$, there is a long exact sequences of cyclic homology groups, called the SBI sequence, as follows. 
\[
...\overset{S}{\rightarrow} HC_{n-1}(C)\overset{B}{\rightarrow} HN_{n}(C) \overset{I}{\rightarrow} HP_n(C) \overset{S}{\rightarrow} HC_{n-2}(C) \overset{B}{\rightarrow} ...
\]
\end{proposition}

Again, as we stated earlier, in order to determine what happens with singular schemes, it will be important for us to know what happens on smooth schemes.  The following results attempt to rewrite cyclic homology and it's variants on smooth schemes in terms of other more well known theories such as de Rham cohomology. 

\begin{note}
In the following theorems we refer to a notion of smooth algebras.  We usually define smooth algebras as satisfying some lifting properties.  For our purposes, since we will eventually be looking at some geometric structures, we will say that an algebra is smooth if it's associated affine scheme is a smooth scheme.  These two notions of smooth are of course equivalent.  
\end{note} 

\begin{theorem}[Connes] Let $A$ be a smooth $k$-algebra over a field $k$ of characteristic zero.  The we have that the cyclic homology is given by the following.  
\[HC_n(A/k)\simeq \Omega_{A/k}^n A/d\Omega_{A/k}^{n-1} A\oplus \bigoplus_{i\geq 1}H^{n-2i}_{DR}(A/k).\]
\end{theorem}
\begin{proof}
In the analytic case, this boils down to replacing the cyclic modules in the definition of the periodic bicomplex with the module of forms. In this algebraic case, things are a little more delicate. The proof, when written concisely requires additional constructions, see \cite{weibel1995introduction} theorem 9.4.7 and theorem 9.8.13.  
\end{proof}
Notice that in the previous theorem, we have that the cyclic homology is very close to de Rham cohomology.  There is a correction term that prevents it from being identical to (a direct sum of) the de Rham cohomology groups.  An important property of periodic cyclic homology, and one that we will use later, is that periodic cyclic homology is in some ways just a generalization of de Rham cohomology.  The reason for this is, according to the following result, the two theories agree in the case of smooth algebras.  
\begin{proposition}\label{hpderham}
Let $A$ be a smooth $k$-algebra such that $\mathbb{Q}\subset k$.  Then \[HP_0(A/k)=H_{DR}^{ev}(A/k)=\prod_{i\geq0}H_{DR}^{2i}(A/k)\] and \[HP_1(A/k)=H_{DR}^{odd}(A/k)=\prod_{i\geq0}H_{DR}^{2i+1}(A/k).\]
So, periodic cyclic homology is a generalization of de Rham cohomology. 
\end{proposition}
\begin{proof}
See \cite{loday1997cyclic} 5.1.12.  
\end{proof}
For negative cyclic homology, we do not necessarily have a concise formula in terms of de Rham cohomology for smooth algebras.  We can, however, use the long exact SBI sequence to compute the negative cyclic homology from the other two.  The result of this computation is as follows.  
\begin{proposition}
For $A$ a smooth $k$-algebra such that $\mathbb{Q}\subset k$, we have that
\[HN_n(A/k)\cong Z^n(A)\times\prod_{i\geq 1} H_{DR}^{n+2i}(A/k),\]
where $Z^n(A)=$Ker$(d:\Omega_{A/k}^{n}\rightarrow\Omega_{A/k}^{n+1})$.  
\end{proposition}

\begin{notation}
In the previous theorems, $\Omega^1_{A/k}$ refers the module of K\"{a}hler differentials.  More generally, the module $\Omega^n_{A/k}$ is defined to be the the $n$th exterior product of $\Omega^1_{A/k}$, $\Lambda^n(\Omega^1_{A/k})=\Omega^n_{A/k}$.  For a more detailed construction, and some general properties, see \cite{weibel1995introduction} 8.8.1, 9.4.2.  
\end{notation}
\begin{remark}
The de Rham cohomology referred to in the previous theorems, as well as in the rest of this paper, is the so called algebraic de Rham cohomology.  Instead of in the differential geometry case where we look at the complex consisting of differential forms of varying degrees, in this case we look at the complex consisting of the various K\"{a}hler modules $\Omega^n_{A/k}$.  This construction of the algebraic de Rham cohomology groups from the complex $\Omega^*_{A/k}$ is completely analogous to to the construction of the more traditional de Rham cohomology groups from the de Rham complex (the one with differential forms).  
\end{remark}

Just like we said could be done for K-theory, this entire construction is done for rings (or algebras).  We can easily extend this construction to schemes, although in this case things are a little more delicate. We proceed in the following manner.  

Let $(X,\mathcal{O}_X)$ be a scheme $X$ with structure sheaf $\mathcal{O}_X$.   We can easily form a ``cyclic" sheaf of modules by using the same construction as before, i.e. $C(\mathcal{O}_X)=\mathcal{O}_X^{\otimes n}$, again where the tensoring takes place over a field $k$.  We define $HC(-/k)$ to be the presheaf of bicomplexes that takes the scheme $(X,\mathcal{O}_X)$ and maps it to the the total complex of the bicomplex $CC(C(\mathcal{O}_X))$ (meaning that we plug in $C(\mathcal{O}_X)$ to the bicomplex construction from above).  

The problem is that $HC(-/k)$ may not be a sheaf.  So, the logical thing to do would be sheafify $HC(-/k)$ which is precisely what we do.  So, we define the sheaf of  complexes $\textbf{HC}(-/k)$ to be the sheafification of $HC(-/k)=$Tot$CC(-)$.   This can be done for each of the variants of cyclic homology that we have discussed up to this point, thus defining sheaves $\textbf{HC}(-/k)$, $\textbf{HN}(-/k)$, and $\textbf{HP}(-/k)$.  

Once we have these sheaves, we need define some sort of cohomology.  Since these sheaves of complexes are not all concentrated in one degree (say degree zero), our general method for defining cohomology for these sheaves is to use hypercohomology.  So, we arrive at the following definition.  
\begin{definition}
Let $X$ be a scheme. Then the cyclic homology, negative cyclic homology, and periodic cyclic homology of $X$ are respectively
\[HC_n(X/k) := \mathbb{H}^n(X,\textbf{HC}(-/k))\]\[HN_n(X/k) := \mathbb{H}^n(X,\textbf{HN}(-/k))\]\[HP_n(X/k) := \mathbb{H}^n(X,\textbf{HP}(-/k))\]
for all $n\in\mathbb{Z}$.
\end{definition}
This definition guarantees that on affine schemes, as expected, we have the following nice property.  
\begin{proposition} 
Let $X=$ Spec$(R)$ be affine, Noetherian, and finite dimensional scheme over a field $k$. Then $HC_n(X/k)=HC_n(R/k)$, $HN_n(X/k)=HN_n(R/k)$, and $HP_n(X/k)=HP_n(R/k)$
\end{proposition}
\begin{proof}
This follows from a standard spectral sequence argument which we will not give.  For an alternative argument see \cite{weibel1996cyclic} theorem 2.5.  
\end{proof}
\begin{corollary}\label{cycliczero}
If the affine scheme $X=$ Spec$(A)$ is Noetherian and finite dimensional, then \\ $HC_n(X/k)=0$ for $n<0$.  
\end{corollary}
\begin{proof}
See corollary 4.6.1 \cite{weibel1991etale}.  
\end{proof}
In the more general case, the cyclic homology in negative degrees may be nonzero. In these cases, we are still able to say something about the cyclic homology in low enough degrees.  
\begin{proposition}\label{cyclicnegd}
If $X$ is a noetherian $k$-scheme of dimension $d$, then $HC_n(X)=0$ for $n<-d$ and $HC_{-d}(X)\cong H^d(X,\mathcal{O}_X)$.  
\end{proposition}
\begin{proof}
This result is actually stated before the previous corollary in \cite{weibel1991etale} (meaning the previous corollary is actually a result of this proposition).  See the proof of lemma 4.6 in \cite{weibel1991etale}.  
\end{proof}
For more information on the cyclic homology of schemes, see \cite{weibel1996cyclic}.

\begin{remark}
Using the sheaf cohomology definition only works when the target of your sheaf is an object in some abelian category (most of the time it is abelian groups).   Trying to define $HC_n(-/k)$ as a functor using the previous construction would result in the target object being a complex (or more specifically a complex of sheaves).  This is why in this scenario, we use the more general hypercohomology construction from \ref{hypercohomology}.  When we choose a specific scheme, our target category is abelian groups, so if things work out right (like $d^2=0$), we could technically use a regular cohomology definition. It would only be in the case of affine schemes and in other very specific circumstances that these definitions would agree.  So, for more general schemes it is more complete to use this construction and obtain a more general functorial definition of $HC_n(-/k)$.  
\end{remark}

\begin{observation}
The scheme versions of cyclic homology (and it's variants) behaves in some ways like a cohomology theory.  By this we mean that as a functor from schemes to abelian groups, it is contravariant.  According to the their agreement on affine schemes, this is precisely the correct behavior.  The reason is, the Spec functor from commutative rings to affine schemes is itself contravariant.  Hence if we want an agreement between these two definitions for affine schemes as above, we need the scheme version of to be contravariant.  We could use cohomological indexing on these groups, namely $HC^{-n}($Spec$(R)/k)=HC_n(R/k)$.  But, either out of laziness or for some other good reason, we choose to treat these groups as interchangeable.  Therefore we treat them both as homology theories with homological indexing, but in practice the scheme theoretic version is more cohomological.   
\end{observation}

\begin{note}
As stated before, we can in some respects treat complexes and spectra as interchangeable.  There are some instances where we really want to use a cyclic homology spectrum (and variations).  In order to construct a spectrum out of these complexes, we rely on the Eilenberg-Maclane spectrum construction.  Which, as before, we will skip.  We simply treat $\textbf{HC}(-/k)$, and the other variants, as sheaves of spectra when necessary.  
\end{note}

\begin{notation}
Following the convention in \cite{cortinas2005cyclic}, as well as in other sources, when we omit the field $k$ from $\textbf{HC}(-/k)$ we mean that we are computing cyclic homology over $\mathbb{Q}$.  So, $\textbf{HC}=\textbf{HC}(-/\mathbb{Q})$.  We also adopt this convention beyond the level of spectra.  So, if we write $HC_n(-)$ we mean $HC_n(-/\mathbb{Q})$.  

The reason for this simplification is that later we will use $\textbf{HC}(-/\mathbb{Q})$ much more often than the general case.  Moreover, when we discuss the fibers $\mathcal{F}_{\textbf{HC}}$, to be defined later, using the notation $\textbf{HC}(-/\mathbb{Q})$ is just not very aesthetically pleasing.  
\end{notation}

\subsection{The Chern character I}$\ $
The important relationships between K-theory and cyclic homology all come from the Chern character map.  This Chern character idea is actually common among many generalized cohomology theories.  Originally it was developed to get a map from topological K-theory (a generalized cohomology theory) to regular cohomology.  Since then it has been adapted by many different authors to get maps from other generalized cohomology theories to regular cohomology theories.  

In particular we will be looking at the Chern character map from algebraic K-theory to negative cyclic homology.  This construction is not the easiest to formally write down, but we will at least give an idea on how this map is developed.  
\begin{definition}
Let $A$ be a $k$-algebra.  We define the fusion map to be the $k$-algebra homomorphism \[f:k[GL_r(A)]\rightarrow M_r(A),\]
which simply replaces the formal sum in $k[GL_r(A)]$ with an actual sum in $M_r(A)$.  
\end{definition}
From this map, and the trace map on the $r\times r$ matrices $M_r(A)$, we have the following sequence of modules.  
\begin{proposition}
There is a sequence of cyclic $k$-modules
\[k[GL_r(A)^n]\rightarrow k[GL_r(A)^{n+1}]\cong k[GL_r(A)]^{\otimes n+1}\overset{f^{\otimes n+1}}{\rightarrow} M_r(A)^{\otimes n+1}\overset{tr^{\otimes n+1}}{\rightarrow} A^{\otimes n+1},\]
which induces a sequence of cyclic bicomplexes,
\[CCN(GL_r(A))\rightarrow CCN(BGL_r(A)])\rightarrow CCN(k[GL_r(A)])\rightarrow CCN(M_r(A))\rightarrow CCN(A)\]
where $BGL_r(A)$ is the Eilenberg-Maclane space for $GL_r(A)$ in degree 1.  
\end{proposition}
\begin{proof}
The first sequence of cyclic modules follows immediately from the definition of the fusion map, and the trace map from $M_r(A)\rightarrow A$.  The second sequence follows from the first sequence by the fact that everything behaves well with tensor products.  
\end{proof}
\begin{proposition}
For any abelian group $G$,
\[HN_n(G)\cong\prod_{i\geq 0}H_{n+2i}(G),\]
where $HN_n(G)=HN_n(C(BG))$ and $C(BG)$ is the Eilenberg-Maclane complex of $G$.    
\end{proposition}
\begin{proof}
This is developed in section 7.3.9 of \cite{loday1997cyclic}, and stated explicitly in equation (7.3.9.4).
\end{proof}
Combining the previous two results with one another yields the desired construction.   
\begin{theorem}
There is a map, called the Chern character map \[ch_n:K_n(A)\rightarrow HN(A/k),\] defined for any $n$.  
\end{theorem}
\begin{proof}
We will proceed by proving the case where $n\geq 1$.  When $n\geq 1$, this map is simply a composition of the Hurewicz homomorphism, \[K_n(A)=\pi_n(GL(A)^+)\rightarrow H_n(GL(A)^+)\cong H_n(GL(A)),\] and the map induced from the previous two results, \[H_{n+2i}(GL(A))\rightarrow HN_n(GL(A))\rightarrow HN_n(A/k).\]  
\end{proof}
For simplicity we may just write $ch$ instead of $ch_n$.  

This entire construction for $n\geq 1$ (except for the result from chapter 7) is carried out in section 8.4 and then concluded with the higher algebraic K-groups in section 11.4 of \cite{loday1997cyclic}.  To extend to 0 and negative degrees requires a bit more work. This is carried out, for $n=0$ in section 8.3 of  \cite{loday1997cyclic}, and for negative degrees in chapter V.II of \cite{weibel2013k}.  These constructions are similar, but they require more technical details, so (although we require them) we will skip them. We refer the reader to the references for any missing details or further ideas.  

\begin{remark}
Based on some of the descent properties below, we really want to use the Jones-Goodwillie Chern character construction.  The previous Chern character definition is actually a different construction from that of Jones and Goodwillie.  The Jones-Goodwillie Chern character construction is very technical and involves a lot of rational homotopy theory ideas.   Fortunately, the two Chern character maps have been shown in \cite{cortinas2009relative} to agree in specific cases (namely for nilpotent ideals).  This is the main application where we will use this Chern character, which is why we have chosen the less technical definition from above.  Goodwillie's Chern character (or at least the good properties that we obtain from it's use) does however require additional assumptions.  Namely, that for our $k$-algebra $k$ is of characteristic zero, meaning that $\mathbb{Q}\subset k$, and that the target of the Chern character is $HN_n(-/\mathbb{Q})=HN_n(-)$.  The reason for is that infinitesimal K-theory (which will be defined later) will fail to satisfy certain descent properties (also explored later) in the more general case.  
\end{remark}

We will continue our study of the Chern character by defining a K-theory that will, in a sense, measure what is left behind under the Chern character map.  Based on our notation from the previous section, we have that $\textbf{HN}$ (of course meaning $\textbf{HN}(-/\mathbb{Q})$) is the sheaf of complexes defining negative cyclic homology.  By using the Eilenberg-Maclane spectrum construction we can construct a sheaf of spectra from this, and by using an abuse of notation, we will also call this $\textbf{HN}$.  From the Chern character construction, we have that there is an induced Chern character map from $\textbf{K}(X)\rightarrow\textbf{HN}(X)$.  In order to measure the difference between these two spectra, following Corti\~nas \cite{cortinas2006obstruction}, we give the following construction.  
\begin{definition}
Let $X$ be a scheme over $\mathbb{Q}$.  We define the infinitesimal K-theory spectrum of $X$, called $\textbf{K}^{inf}(X)$, to be the homotopy fiber of the Chern character map $\textbf{K(X)}\rightarrow\textbf{HN}(X)$.  
\end{definition}
As with other spectra, this homotopy fiber defines it's own variant of K-theory.  This is given by the following.  
\begin{definition}
The infinitesimal K-groups of a scheme $X$ (or a ring) are given by the following,
\[\pi_n(\textbf{K}^{inf}(X))= K^{inf}_n(X)\] for all $n\in\mathbb{Z}$.  
\end{definition}
\subsection{Descent properties}
In this section we present the descent properties of the various constructions, or rather the sheaves associated to these constructions.  Returning to our discussion of descent, we have two theorems describing what it means for a sheaf to satisfy a certain type of descent.  They are the theorems \ref{mvweak} and \ref{mvchain}.  Our primary focus will be to look at descent for the scdh-topology and descent for the cdh-topology.  We will begin with the following corollary to the referenced theorems.  
\begin{corollary}
Let $A$ be a presheaf of spectra (or complexes) on the site (Sm/$k$,scdh).  Then $A$ satisfies scdh-descent if and only if $A$ satisfies Nisnevich descent, and the MV-property for all smooth blow-up squares.   
\end{corollary}

For our sheaves of concern, we will begin by determining which of these sheaves satisfy scdh-descent.  
\begin{theorem}
The presheaves of spectra (or complexes) $\textbf{HC}(-/k)$ and $\textbf{HN}(-/k)$, defining cyclic homology and negative cyclic homology, satisfy scdh-descent.
\end{theorem}
\begin{proof}
The proof that these satisfy Nisnevich descent is contained in \cite{weibel1991etale} 4.2.1, and 4.8.  These satisfy the MV-property for smooth blowups by \cite{cortinas2005cyclic} theorem 2.10.  
\end{proof}
The presheaf defining periodic cyclic homology also satisfies scdh-descent, but it also has stronger descent properties, so we will save this result for later. 

\begin{theorem}
The presheaf of spectra defining K-theory, $\textbf{K}$, satisfies scdh-descent.  
\end{theorem}
\begin{proof}
$\textbf{K}$ satisfies Nisnevich descent by \cite{thomason2007higher}, 10.8. And $\textbf{K}$ satisfies the MV-property for smooth blowups by \cite{cortinas2005cyclic}, 1.6.  
\end{proof}

While these sheaves that do not satisfy cdh-descent don't have as nice of properties as those that do, they still fit in the whole picture in some nice ways.  For a presheaf $A$ on Sch/$k$ we will define $rA$ to be it's restriction to the subcategory Sm/$k$.  We have the following result the relates maps of presheaves to maps of their restrictions.  
\begin{proposition}\label{scdhtocdh}
Let $\mathcal{E}$ be a presheaf of spectra on Sch/$k$.  Then, $\mathbb{H}_{scdh}(-,\mathcal{E})=r\mathbb{H}_{cdh}(-,\mathcal{E})$.  In particular if $\mathcal{E}$ satisfies scdh-descent, then $\mathcal{E}(X)\rightarrow\mathbb{H}_{cdh}(X,\mathcal{E})$ is a weak equivalence for any smooth $X(\in$ Sm$/k)$.  
\end{proposition}
\begin{proof}
See the argument before theorem 2.4 in \cite{cortinas2008k}, and the first part of the proof of theorem 3.12 in \cite{cortinas2005cyclic}.  
\end{proof}

Summarizing things up to this point, for smooth schemes, we have that cyclic homology, negative cyclic homology and K-theory each satisfy scdh-descent.  By the previous result, this means that when we take the cdh-cohomology versions of each of these theories, at least on smooth schemes we gain nothing new.  By this we mean the following.  
\begin{corollary}
For a smooth scheme $X$ over a field $k$ of characteristic zero,
\begin{align*}
HC_n(X/k)&=\mathbb{H}^{-n}_{cdh}(X,\textbf{HC}(-/k))\\
HN_n(X/k)&=\mathbb{H}^{-n}_{cdh}(X,\textbf{HN}(-/k))\\
K_n(X)&=\mathbb{H}^{-n}_{cdh}(X,\textbf{K})
\end{align*}
for all $n$.
\end{corollary}

Now we move on to cdh-descent.  We will primarily use the following result in order to determine which among our sheaves satisfy cdh-descent.  

\begin{theorem}\label{cdh}
Let $\mathcal{E}$ be a presheaf of spectra on the site (Sch/$k$,cdh).  Suppose $\mathcal{E}$ satisfies excision, is invariant under infinitesimal extension, satisfies Nisnevich descent and satisfies the MV-property for every blowup along a regular sequence.  Then $\mathcal{E}$ satisfies cdh-descent.  
\end{theorem}
\begin{proof}
See \cite{cortinas2005cyclic} theorem 3.12. 
\end{proof}

The first example of this, is one that we referenced back in section 2.  We have the following result for periodic cyclic homology.  
\begin{theorem}
Let $k$ be a field of characteristic zero.  For each subfield $l\subseteq k$ the presheaf $\textbf{HP}(-/l)$ satisfies cdh-descent on Sch/$k$.
\end{theorem}
\begin{proof}
The main idea of this proof is to use the fact that periodic cyclic homology coincides with crystalline cohomology \cite{feigin1985additive}, and this cohomology has the desired properties. For $k=\mathbb{Q}$ see \cite{cortinas2005cyclic} corollary 3.13.  For the general case, see \cite{cortinas2008k} theorem 1.2.
\end{proof}

\begin{theorem}
The presheaf of spectra $\textbf{KH}$ satisfies cdh-descent on Sch/$k$.  
\end{theorem}
\begin{proof}
This example is actually the basis for the proof of \ref{cdh}.  This was done by Haesemeyer before the general case was fully worked out.  This proof can be found in \cite{haesemeyer2004descent} theorem 6.4.
\end{proof}
This result yields some very nice properties.  In particular, one of the main properties that we will use later is the following.  
\begin{corollary}
For a scheme $X$ over a field $k$ of characteristic zero, Homotopy K-theory is cdh-fibrant K-theory, meaning that $\mathbb{H}_{cdh}(X,\textbf{K})\rightarrow \textbf{KH}(X)$ is a weak equivalence.
\end{corollary}
\begin{proof}
Since by proposition \ref{kh=k} Homotopy K-theory and K-theory agree on smooth schemes, combining this with proposition \ref{scdhtocdh} and corollary \ref{notsmoothsmooth}, we have the desired result.  
\end{proof}

\begin{theorem}
The presheaf of spectra $\textbf{K}^{inf}$ satisfies cdh-descent on Sch/$k$.  
\end{theorem}
\begin{proof}
Excision is one of the main theorems proved by Corti\~nas in \cite{cortinas2006obstruction}.  $\textbf{K}^{inf}$ satisfies Nisnevich descent because both $\textbf{K}$ (\cite{thomason2007higher},10.8) and $\textbf{HN}$ (\cite{cortinas2005cyclic},2.9) do.  $\textbf{K}^{inf}$ satisfies the MV-property for every blowup along a regular sequence because both $\textbf{K}$ (\cite{cortinas2005cyclic},1.6) and $\textbf{HN}$ (\cite{cortinas2005cyclic},2.10) do. 
 
Finally, if $A$ is a $\mathbb{Q}$-algebra and $I\subset A$ is a nilpotent ideal then $\textbf{K}^{inf}(A,I)$ is contractible.  This is because, by Goodwillie \cite{goodwillie1986relative}, the Chern character induces a weak equivalence of spectra $\textbf{K}(A,I)\simeq\textbf{HN}(A,I)$.  This implies invariance under infinitesimal extension.  

Therefore, since all of the assumptions of theorem \ref{cdh} are satisfied, we know $\textbf{K}^{inf}$ satisfies cdh-descent
\end{proof}

The other properties of these sheaves will be a consequence of the following result.  Recall $r$ restricts presheaves to the subcategory of smooth schemes Sm/$k$.  
\begin{proposition}\label{cdhscdh}
Let $f:A\rightarrow B$ be a map of presheaves of spectra on Sch/$k$.  If $f$ is a global weak equivalence in the cdh-topology, then the map of restrictions $rf:rA\rightarrow rB$ is a global weak equivalence in the scdh-topology.  
\end{proposition}
\begin{proof}
See \cite{cortinas2005cyclic} lemma 3.11.  
\end{proof}
\begin{corollary}
Any presheaf $\mathcal{E}$ that satisfies cdh-descent has the property that it's restriction $r\mathcal{E}$ satisfies scdh-descent.  In particular, the presheaves $\textbf{HP}(-/k)$, $\textbf{KH}$, and $\textbf{K}^{inf}$ (or rather their restrictions to Sm/$k$) satisfy scdh-descent.  
\end{corollary}
\begin{proof}
For a presheaf $\mathcal{E}$ that satisfies cdh-descent, since $\mathcal{E}\rightarrow\mathbb{H}_{cdh}(-,\mathcal{E})$ is a global weak equivalence, by \ref{cdhscdh} we know that the restriction $r\mathcal{E}\rightarrow r\mathbb{H}_{cdh}(-,\mathcal{E})=\mathbb{H}_{scdh}(-,\mathcal{E})$ is a global weak equivalence and hence $r\mathcal{E}$ satisfies scdh-descent.  Since each of the presheaves $\textbf{HP}(-/k)$, $\textbf{KH}$, and $\textbf{K}^{inf}$ each satisfy cdh-descent, their restrictions must satisfy cdh-descent.  
\end{proof}

Summarizing as before, we have the following.  
\begin{corollary}
For any scheme $X$ essentially of finite type over a field $k$ of characteristic zero,
\begin{align*}
HP_n(X/k)&=\mathbb{H}^{-n}_{cdh}(X,\textbf{HP}(-/k))\\
KH_n(X)&=\mathbb{H}^{-n}_{cdh}(X,\textbf{KH})\\
K^{inf}_n(X)&=\mathbb{H}^{-n}_{cdh}(X,\textbf{K}^{inf})
\end{align*}
for all $n$.
\end{corollary}

The following table is a summary of the descent properties for all of our sheaves of concern.  

\vspace{5 mm}

\begin{center}
\begin{tabular}{c|c|c|}
  & scdh-descent & cdh-descent \\  \hline                      
  $\textbf{HC}(-/k)$ & $\surd$ & \\ \hline 
  $\textbf{HN}(-/k)$ & $\surd$ & \\ \hline 
  $\textbf{HP}(-/k)$ & $\surd$ & $\surd$ \\ \hline 
  $\textbf{K}$ &  $\surd$ &  \\ \hline 
  $\textbf{KH}$ & $\surd$ & $\surd$ \\ \hline 
  $\textbf{K}^{inf}$ & $\surd$ & $\surd$\\ 
  \hline  
\end{tabular}
\end{center}

\vspace{5 mm}

Before we continue with how this fits together with the Chern character, we need a concise way to measure how our various sheaves of spectra differ from their cdh versions.  
\begin{definition}
Let $\mathcal{E}$ be a (pre)sheaf of spectra on (Sch/$k$,cdh).  For a scheme $X$, we define $\mathcal{F}_\mathcal{E}(X)$ to be the homotopy fiber of the map of spectra $\mathcal{E}(X)\rightarrow\mathbb{H}_{cdh}(X,\mathcal{E})$.  We write $\mathcal{F}_\mathcal{E}(-)$, or just  $\mathcal{F}_\mathcal{E}$, as the associated sheaf of spectra.  
\end{definition}
Using the summary of information from above we have the following immediate results for the fibers of the sheaves that we care about. 
\begin{proposition}
Let $X$ be a smooth scheme over a field of characteristic zero, and let $\mathcal{E}$ denote any of the sheaves from the table above (or any sheaf satisfying scdh-descent), then the spectrum $\mathcal{F}_\mathcal{E}(X)$ is contractible.  
\end{proposition}
\begin{proof}
Since each of the sheaves satisfy scdh-descent we know that, when evaluated at a scheme, the resulting spectrum is homotopy equivalent to it's cdh-fibrant replacement.  Therefore the fibers must be contractible.  
\end{proof} 
\begin{proposition}
For any scheme $X$ over a field of characteristic zero, the spectra $\mathcal{F}_{\textbf{HP}}(X/k)$, $\mathcal{F}_{\textbf{KH}}(X)$, and  $\mathcal{F}_{\textbf{K}^{inf}}(X)$ (or the fiber associated to any sheaf satisfying cdh-descent) are all contractible.  
\end{proposition}
\begin{proof}
Since each of the sheaves satisfy cdh-descent we know that, when evaluated at a scheme, the resulting spectrum is homotopy equivalent to it's cdh-fibrant replacement.  Therefore each of the fibers are contractible.  
\end{proof}

\begin{notation}
For a scheme $X$, we write the fiber of the $\mathcal{F}_{\textbf{K}}(X)$ to be the fiber of the map from $\textbf{K}(X)\rightarrow\mathbb{H}_{cdh}(X,\textbf{K})$.  We also write $\tilde{\textbf{K}}(X)$ to be the fiber of the map $\textbf{K}(X)\rightarrow\textbf{KH}(X)$.  Since $\mathbb{H}_{cdh}(X,\textbf{K})=\textbf{KH}(X)$ (over characteristic zero), we know that $\mathcal{F}_{\textbf{K}}(X)=\tilde{\textbf{K}}(X)$, or $\mathcal{F}_{\textbf{K}}=\tilde{\textbf{K}}$.  So, these two notations describe the same thing.  For the most part, unless it is more convenient to use the other notation, we will try to use $\tilde{\textbf{K}}(X)$ to describe this fiber, or more descriptively, the difference between K-theory and Homotopy K-theory.  
\end{notation}

\begin{remark}
This notation of $\tilde{\textbf{K}}(X)$, defining the groups $\tilde{K}_n(X)$, is usually reserved for graded rings.  For a graded ring $R=k\oplus R_1\oplus R_2\oplus...$, where $K_n(R)\rightarrow KH_n(R)=KH_n(k)$ is a split surjection, the groups $\tilde{K}_n(X)$ can be written as the quotient $K_n(R)/KH_n(k)$, and this is the traditional definition of $\tilde{K}_n(X)$.  While our results do not necessarily involve graded rings (except in the quasi-homogeneous case) we will still use this notation for it's simplicity over the alternative.  
\end{remark}

\subsection{The Chern character II} With these descent properties in mind, we will revisit the Chern character and present several new properties.  

Recall that the Chern character map can be written as a map at the level of spectra from K-theory to negative cyclic homology, $\textbf{K(X)}\rightarrow\textbf{HN}(X)(=\textbf{HN}(X/\mathbb{Q}))$.  By definition of infinitesimal K-theory, we have that \[\textbf{K}^{inf}(X)\rightarrow\textbf{K}(X)\rightarrow\textbf{HN}(X)\] is a fibration sequence of spectra.  By looking at the cdh versions of each of these and the associated, we have induced fibration sequences at each level.  This yields the following diagram.  
\[\begin{CD}
\mathcal{F}_{\textbf{K}^{inf}}(X)@>>> \mathcal{F}_{\textbf{K}}(X)@>>> \mathcal{F}_{\textbf{HN}}(X)\\
@VVV @VVV @VVV\\
\textbf{K}^{inf}(X) @>>> \textbf{K}(X)  @>>> \textbf{HN}(X)\\
@VVV @VVV @VVV\\
\textbf{K}^{inf}(X)\cong \mathbb{H}_{cdh}(X,\textbf{K}^{inf}) @>>> \mathbb{H}_{cdh}(X,\textbf{K})@>>> \mathbb{H}_{cdh}(X,\textbf{HN}) 
\end{CD}\] 
\begin{proposition}
Let $X$ be scheme over a field of characteristic zero.  We have that $\mathcal{F}_{\textbf{K}}(X)=\tilde{\textbf{K}}(X)$ is homotopy equivalent to $\mathcal{F}_{\textbf{HN}}(X)$.
\end{proposition}
\begin{proof}
Since $\textbf{K}^{inf}$ satisfies cdh-descent, we know that $\mathcal{F}_{\textbf{K}^{inf}}(X)$ is contractible.  Combining this with the previous diagram of fibration sequences, we know that $\mathcal{F}_{\textbf{HN}}(X)$ is homotopy equivalent to $\mathcal{F}_{\textbf{K}}(X)$, proving the result.  
\end{proof}

To complete this result, we now look at the fibers in the SBI sequence for cyclic homology and the two variations.  Let $X$ be a scheme.  We have that \[\textbf{HC}_{-1}(X/k)\rightarrow\textbf{HN}(X/k)\rightarrow\textbf{HP}(X/k)\] is a fibration sequence of spectra defining the various versions of cyclic homology and inducing the SBI sequence. The subscript of $-1$ denotes the shift in index.  When we pass to their cdh versions, we have the following diagram of induced fibration sequences.  
\[\begin{CD}
\mathcal{F}_{\textbf{HC}_{-1}(-/k)}(X)@>>> \mathcal{F}_{\textbf{HN}(-/k)}(X)@>>> \mathcal{F}_{\textbf{HP}(-/k)}(X)\\
@VVV @VVV @VVV\\
\textbf{HC}_{-1}(X/k) @>>> \textbf{HN}(X/k) @>>> \textbf{HP}(X/k)\\
@VVV @VVV @VVV\\
\mathbb{H}_{cdh}(X,\textbf{HC}_{-1}(-/k)) @>>> \mathbb{H}_{cdh}(X,\textbf{HN}(-/k))@>>> \mathbb{H}_{cdh}(X,\textbf{HP}(-/k))\cong\textbf{HP}(X/k)
\end{CD}\]
\begin{proposition}
Let $X$ be scheme over a field of characteristic zero.  We have that $\mathcal{F}_{\textbf{HC}_{-1}(-/k)}(X)$ is homotopy equivalent to $\mathcal{F}_{\textbf{HN}(-/k)}(X)$.
\end{proposition}
\begin{proof}
From the previous diagram of fibration sequences and the fact that $\mathcal{F}_{\textbf{HP}(-/k)}(X)$ is contractible, we obtain the desired result.  
\end{proof}
Now combining these two results, by setting $k=\mathbb{Q}$, we have the following equivalence induced by the Chern character map.  
\begin{corollary}
The Chern character map induces an equivalence of spectra, $\tilde{\textbf{K}}(X)\cong \mathcal{F}_{\textbf{HC}_{-1}}(X)$.  
\end{corollary}

Using this, and some previous results, we finally have a framework where we can explicitly relate (the fibers of) cyclic homology and K-theory.  This is given to us by the following diagram.  
\[\begin{CD}
\tilde{\textbf{K}}(X) @>>> \textbf{K}(X) @>>> \textbf{KH}(X)\\
@| @. @.\\
\mathcal{F}_{\textbf{HC}_{-1}}(X)@>>> \textbf{HC}_{-1}(X)@>>> \mathbb{H}_{cdh}(X,\textbf{HC}_{-1})
\end{CD}\]
\begin{remark}
This framework really only works for singular schemes.  Since both $\textbf{K}$ and $\textbf{HC}$ satisfy scdh-descent, we know that for smooth schemes the fibers $\tilde{\textbf{K}}$, and $\mathcal{F}_{\textbf{HC}_{-1}}$ are contractible.  This is consistent with this framework, in that $\tilde{\textbf{K}}\cong\mathcal{F}_{\textbf{HC}_{-1}}$, but it does not give us any new information.  In order to use information from cyclic homology and it's cdh version in the computation of the K-theory, using this framework, the fibers need to be nontrivial.  So, as stated above, we really need to use singular schemes to obtain any new information about the K-theory.  
\end{remark}

We need to determine what happens at the level of K-groups and cyclic homology groups.  In order to do this, we will simply pass to the homotopy groups.  Based on this framework, for a scheme $X$, 
\[\tilde{\textbf{K}}_{+1}(X)\rightarrow\textbf{HC}(X)\rightarrow\mathbb{H}_{cdh}(X,\textbf{HC})\] is a homotopy fibration.  When we evaluate this fibration sequence in the homotopy group functors, we obtain the standard long exact sequence of homotopy groups for fibrations.  
\begin{theorem}\label{mainlong}
Let $X$ be a scheme over a field $k$ of characteristic zero.  Then there exists the following long exact sequence. 
\[...\rightarrow\tilde{K}_{n+1}(X)\rightarrow HC_n(X)\rightarrow\mathbb{H}^{-n}_{cdh}(X,\textbf{HC})\rightarrow\tilde{K}_{n}(X)\rightarrow...\]
\end{theorem}
\begin{proof}
Recall, we use cohomological indexing for cdh-cohomology.  Hence the reason for the negative superscript, $\pi_n{\mathbb{H}_{cdh}(X,\textbf{HC})}=\mathbb{H}^{-n}_{cdh}(X,\textbf{HC})$.  The rest of this proof is basically given before the statement of the theorem, the long exact sequence is inherited by the long exact fibration sequence for homotopy groups.  
\end{proof}

Combining this with the Mayer-Vietoris sequence for cdh-cohomology is what will allow us to make new computations for K-theory (or at least the fiber $\tilde{K}_n$).  This is given by the following theorem.  

\begin{theorem}\label{mainhybrid}
Let $X$ be a singular variety over a field of characteristic zero, and let $Y$, $Z$, and $E$ be the resolution, center, and exceptional fiber of the resolution respectively (each of $Y$, $Z$, and $E$ are assumed to be smooth).  Then there is a diagram which is a combination of two long exact sequences.  A rough approximation of this diagram is as follows.  The solid arrows are the long exact sequence from \ref{mainlong}, and the dashed arrows is the Mayer-Vietoris sequence for cdh-cohomology.  
\end{theorem}

\noindent\begin{tikzpicture}[scale=.97, every node/.style={scale=.97} ]
\matrix (m) [matrix of math nodes, row sep=3em,
column sep=1em, text height=2ex, text depth=0.2ex]
{   & &   & & & HC_{n-1}(E) &\\
    & &   & & & HC_{n-1}(Y)\oplus HC_{n-1}(Z) &\\
... &HC_{n}(X) & \mathbb{H}^{-n}_{cdh}(X,\textbf{HC}) & \tilde{K}_n(X) & HC_{n-1}(X) & \mathbb{H}^{-n+1}_{cdh}(X,\textbf{HC})& ...\\
  & & HC_n(Y)\oplus HC_n(Z)  & & & &  \\ 
  & & HC_n(E)  & & & &... \\ };\path[->]
(m-3-1) edge (m-3-2);
\path[->]
(m-3-2) edge (m-3-3);
\path[->]
(m-3-3) edge (m-3-4);
\path[->]
(m-3-4) edge (m-3-5);
\path[->]
(m-3-5) edge (m-3-6);
\path[->]
(m-3-6) edge (m-3-7);
\path[<-][densely dashed]
(m-5-3) edge (m-4-3);
\path[<-][densely dashed]
(m-4-3) edge (m-3-3);
\path[<-][densely dashed]
(m-1-6) edge (m-2-6);
\path[<-][densely dashed]
(m-2-6) edge (m-3-6);
\path[<-][densely dashed]
(m-3-1) edge [bend right=20] (m-5-3);
\path[<-][densely dashed]
(m-3-3) edge [bend left=20] (m-1-6);
\path[<-][densely dashed]
(m-3-6) edge [bend right=20] (m-5-7);
\end{tikzpicture}

\subsection{Hodge decomposition}
In this section we will discuss the decomposition of all of our theories into smaller more manageable pieces.  Most of our main results will be stated in terms of this decomposition, so we will reference this section often.  We have divided this section into three smaller sections dealing with de Rham cohomology, cyclic homology, and K-theory separately.

\subsubsection{de Rham cohomology}
We have already been briefly introduced to de Rham cohomology.  This was earlier in the section when we stated results about the cyclic homology of smooth algebras.  As we have seen time and time again,  things that work for algebras, and affine schemes do not always extend so nicely to more general schemes.  We first need to rewrite the de Rham complex in terms of sheaves on a general scheme.  We introduce the following notations.  
\begin{notation}
For the category Sch/$k$, the sheaf theoretic algebraic de Rham complex is denoted by $\Omega_{-/k}^{*}$ is given by the complex, \[0\rightarrow\mathcal{O}_{-/k}\rightarrow\Omega_{-/k}\rightarrow\Omega_{-/k}^2\rightarrow\Omega_{-/k}^3\rightarrow...\]
where for a specific scheme $X$, $\Omega_{X/k}^{n}$ denotes the $n$th exterior power of the module of K\"{a}hler differentials $\Omega_{X/k}$.  
\end{notation}
For specific scenarios, we will use the following truncated complex. 
\begin{notation}
For the category Sch/$k$, the sheaf theoretic \textbf{truncated} algebraic de Rham complex denoted by $\Omega_{-/k}^{\leq i}$ is given by the complex, \[0\rightarrow\mathcal{O}_{-/k}\rightarrow \Omega_{-/k}\rightarrow \Omega_{-/k}^2\rightarrow...\rightarrow \Omega_{-/k}^i\rightarrow 0\rightarrow 0\rightarrow...\]
We also define the \textbf{cotruncated} algebraic de Rham complex denoted by $\Omega_{-/k}^{\geq i}$ is given by the complex, \[0\rightarrow0\rightarrow...\rightarrow0\rightarrow\Omega_{-/k}^i\rightarrow \Omega_{-/k}^{i+1}\rightarrow \Omega_{-/k}^{i+2}\rightarrow...\]
\end{notation}

We recall that for sheaves of complexes, we can not really form an injective resolution.  So, in order to define a sheaf theoretic cohomology for this complex, we will use hypercohomology.  Recalling the definition \ref{hypercohomology} we have that following general definition of de Rham cohomology groups.  
\begin{definition}
For a smooth scheme $X$ of finite type over $k$, we define the de Rham cohomology groups of $X$ to be the hypercohomology 
\[H_{DR}^n(X/k)=\mathbb{H}^n(X,\Omega^{*}_{-/k}).\]
\end{definition}
This definition, like those other sheaf theoretic constructions of the various homology groups, is designed so that it behaves well with affine schemes.  
\begin{theorem}
Let $X=$ Spec$(A)$ be a smooth affine scheme over a field $k$ of characteristic zero.  Then for each $n$,  
\[H_{DR}^n(X/k)=\mathbb{H}^n(X,\Omega^{*}_{-/k})=H^n(\Gamma(X,\Omega^{*}_{-/k})).\]
(Where is $H^n(\Gamma(X,\Omega^{*}_{-/k}))=H_{DR}^n(A/k)$ as before for algebras.)  
\end{theorem}
\begin{proof}
This result is stated informally in \cite{grothendieck1966rham} equation (3), as the result of the standard spectral sequence argument for hypercohomology.  
\end{proof}
\begin{remark}
This definition requires that the scheme be smooth.  There is a more general construction that gives the de Rham cohomology of possibly singular schemes.  This is more technical and the easy versions of the construction involve both an embedding into a smooth ambient space, and a completion argument.  This is done in full in \cite{hartshorne1975rham}, and since we don't need this more general idea for our results, we will skip these constructions.
\end{remark}

In addition to the remarks about the scheme theoretic versions of de Rham cohomology, the whole purpose of this section was to impose a Hodge decomposition on the de Rham cohomology groups.  This is one of those scenarios where it easier to state a result for projective schemes. 
\begin{proposition}
For a smooth projective scheme $X$, the Hodge decomposition of de Rham cohomology is given by the following product of sheaf cohomology groups. \[H_{DR}^n(X/k)=\prod_{p+q=n}H^q(X,\Omega_{-/k}^p)\]
\end{proposition}
\begin{proof}
This comes from a specific spectral sequence called the Hodge to de Rham spectral sequence.  In certain cases, including smooth projective schemes, the sequence degenerates at the $E^{p,q}_1=\prod_{p+q=n}H^q(X,\Omega_{-/k}^p)$ term, which would prove the claim.  See the remarks before section 3 in \cite{weibel1997Hodge}.  
\end{proof}
This is a little different than the Hodge decompositions that we will see in the following sections.  The reason for this is that it is bigraded, and this may not behave well later with the other Hodge decompositions that are only graded.  Fortunately, we have a substitute that gives a ``grading" instead of a bigrading.  Following \cite{weibel1997Hodge} who followed \cite{deligne1971theorie}, we can equip our de Rham cohomology groups with a filtration of this decomposition, or rather a Hodge filtration.  
\begin{definition}
For a smooth projective scheme $X$, the Hodge filtration of de Rham cohomology, denoted $F^iH_{DR}^n(X/k)$, is given by the following product of sheaf cohomology groups. \[F^iH_{DR}^n(X/k)=\prod_{\overset{p+q=n}{p\geq i}}H^q(X,\Omega_{-/k}^p)\]
\end{definition}

According to \cite{weibel1997Hodge} the general Hodge decomposition of de Rham cohomology of an arbitrary scheme can be complicated, and since we really only need this filtration for projective schemes, we will now move on to other constructions.  

\subsubsection{Cyclic Homology}
We will now present the Hodge decomposition of cyclic homology.  While we will avoid very technical details, it will be beneficial later to have a rough idea of where this decomposition comes from.  Some other authors, namely \cite{loday1997cyclic} and \cite{connes2012cyclic}, refer to this decomposition as the $\lambda$-decomposition of cyclic homology.  The reason for this is that the decomposition is obtained from the $\lambda$-operation on the cyclic bicomplex.  While we do not want to get into the technical details of this operation, in some ways the $\lambda$-operation puts a filtration or grading on the cyclic homology groups via iterated use of this $\lambda$-operation.  Avoiding all of the unnecessary technical discussion, we have the following.   

\begin{proposition}
For each of the presheaves of complexes $\textbf{HH}(-/k)$ (Hochschild homology), $\textbf{HC}(-/k)$, $\textbf{HN}(-/k)$, and $\textbf{HP}(-/k)$ there exists a Hodge decomposition.  By this we mean, taking $H$ to be one of the presheaves above, there exists a collection of presheaves of complexes $H^{(i)}$ such that for a scheme $X$ (over characteristic 0), \[H(X)=\prod H^{(i)}(X).\]  
\end{proposition}
For a more in depth look at this decomposition, see \cite{gerstenhaber1987Hodge} for $\textbf{HH}(-/k)$ and \cite{loday1997cyclic} section 4.6 for $\textbf{HC}(-/k)$ and the others.  Before we state various descent properties of these decompositions, we will first present some of their basic properties.  

Using the general computation of cyclic homology for smooth affine schemes, we know the various components already split up into a direct sum.  As luck would have it, this direct sum coincides with the Hodge decomposition.   So, at least in the smooth case, the cyclic homology groups have an obvious Hodge decomposition.  
\begin{theorem}
For $X$ smooth affine, the Hodge decomposition of cyclic homology coincides with the decomposition in terms of the de Rham cohomology of $X$. Meaning, 
\[HC_n^{(n)}(X/k)=\Omega_{X/k}^n X/d\Omega_{X/k}^{n-1} X,\]
\[HC_n^{(i)}(X/k)=H^{2i-n}_{DR}(X/k),\  \text{for}\ n/2\leq i< n,\] 
\[HC_n^{(i)}(X/k)=0,\  \text{for}\ i\leq n/2,\ n<i. \]
\end{theorem}
\begin{proof}
See \cite{loday1997cyclic} theorem 4.6.10.  
\end{proof}
Much in the same way, for smooth algebras both negative cyclic homology and periodic cyclic homology have obvious Hodge decompositions. 
\begin{theorem}
For $X$ smooth affine, the Hodge decomposition of negative cyclic homology coincides de Rham cohomology decomposition of $X$. Meaning, 
\[HN_n^{(n)}(X/k)=Z^n,\]
\[HN_n^{(i)}(X/k)=H^{n+2i}_{DR}(X/k),\  \text{for}\ n< i, \]
\[HN_n^{(i)}(X/k)=0,\  \text{for}\ i< n. \]
\end{theorem} 
\begin{proof}
See \cite{loday1997cyclic} the second part of 5.1.12.  
\end{proof}
While periodic cyclic homology is periodic, it's Hodge components are not, and they shift according to the index $n$.  We have the following.  
\begin{theorem}
For $X$ smooth affine, the Hodge decomposition of periodic cyclic homology coincides with the de Rham cohomology decomposition of $X$. 
\[HP^{(i)}_n(X/k)=H_{DR}^{2i-n}(X/k)\]
\end{theorem} 
\begin{proof}
See \cite{loday1997cyclic} the second part of 5.1.12.   
\end{proof}

Shifting our focus slightly now, we need to know what happens on smooth but not necessarily affine schemes.  The short answer is, for periodic cyclic homology things behave well, and for the others it is much more complicated.  We begin with cyclic homology, which as we stated, is not so nice.  We do however have a concise way to write this decomposition down.  Recalling the truncated de Rham complex and again the construction of hypercohomology from \ref{hypercohomology}, we are able to state the following results.    
\begin{theorem}\label{cyclichyperHodge}
For $X$ smooth over a field $k$ of characteristic zero, the Hodge decomposition of cyclic homology is given by the hypercohomology of the complexes  $\Omega_{-/k}^{\leq i}$,\[HC_n^{(i)}(X/k)=\mathbb{H}^{2i-n}(X,\Omega_{-/k}^{\leq i})\] 
\end{theorem}
\begin{proof}
This again results from using the differentials $\Omega_{X/k}^n$ in the cyclic bicomplex.  See theorem 3.3 \cite{weibel1997Hodge}. 
\end{proof}

For negative cyclic homology, we use the other half of the algebraic de Rham complex
\begin{theorem}
For $X$ smooth over a field $k$ of characteristic zero, the Hodge decomposition of the negative cyclic homology is given by the hypercohomology of the complexes  $\Omega_{-/k}^{\geq i}$,\[HN_n^{(i)}(X/k)=\mathbb{H}^{2i-n}(X,\Omega_{-/k}^{\geq i})\] 
\end{theorem}
\begin{proof}
This should make some sense intuitively.  Back when we discussed the cyclic bicomplex, we had that while cyclic homology dealt with positive half of the bicomplex, negative cyclic homology used the other side.  This is essentially, now in this case, how we get this parity between cyclic homology and negative cyclic homology.  See theorem 3.3 \cite{weibel1997Hodge}. 
\end{proof}

\begin{theorem}\label{hpderham2}
For $X$ smooth over a field of characteristic zero, the Hodge decomposition of periodic cyclic homology coincides with the de Rham cohomology decomposition of $HP_n(X/k)$. 
\[HP^{(i)}_n(X/k)=H_{DR}^{2i-n}(X/k)\]
\end{theorem}
\begin{proof}
While we will discuss the behavior of the SBI sequence under these decompositions shortly, the sequence is essentially what this result boils down to.  Very roughly, assuming some of the later discussion, we will first notice the overlap in the previous two results.  Cyclic homology is determined by the complex $\Omega_{X/k}^{\leq i}$ and negative cyclic homology is determined by $\Omega_{X/k}^{\geq i}$.  Using something like the SBI sequence, we have then that periodic cyclic homology should be determined by $\Omega_{X/k}^{i}$.  Using hypercohomology in this case is unnecessary because all of the data is concentrated in one degree.  Hence we see that the Hodge components of periodic cyclic homology should be determined by the cohomology of $\Omega_{X/k}^{i}$, which is de Rham cohomology.  See theorem 3.3 \cite{weibel1997Hodge}.  
\end{proof}

While these are concise ways to write each of the components, they are not all that easy to compute, except in some cases of periodic cyclic homology.  For some cases, such as what we have already seen in smooth affine schemes, these decompositions can be written in a way that may be easier to work with.  We will also be able to do this for the cyclic homology of smooth projective varieties.  
\begin{theorem}\label{projderham}
Let $X$ be a smooth projective variety over a field $k$ of characteristic zero.  Then the Hodge decomposition of cyclic homology is given by the $(i-n)$th level of the classical Hodge filtration on $H_{DR}^*(X/k)$, \[HC_n^{(i)}(X/k)=\prod_{\overset{p+q=2i-n}{p\leq i}}H^q(X,\Omega_{-/k}^p)=F^{i-n}H_{DR}^{2i-n}(X/k).\]  
\end{theorem}
\begin{proof}
See remark 9.8.19 \cite{weibel1995introduction} and although not concisely stated the details are in proposition 4.1 of \cite{weibel1997Hodge}. According to \cite{weibel1995introduction} result is what justifies the use of the term Hodge decomposition instead of the more constructive $\lambda$-decomposition. 
\end{proof} 
This formula still leaves something to be desired, but for right now we will leave this discussion, and look at how these decompositions relate to one another.  

In addition to these Hodge decompositions reducing these groups into more manageable pieces, we also have the decompositions behave well with long exact sequences.  
\begin{proposition}\label{Hodgesbi}For each $i$, there exists a long exact sequence of Hodge components induced by the long exact SBI sequence given by the following.   
\[
...\overset{S}{\rightarrow} HC^{(i)}_{n-1}(X/k)\overset{B}{\rightarrow} HN^{(i+1)}_{n}(X/k) \overset{I}{\rightarrow} HP^{(i+1)}_n(X/k) \overset{S}{\rightarrow} HC^{(i)}_{n-2}(X/k) \overset{B}{\rightarrow} ...
\]
\end{proposition}
\begin{proof}
The reason for this is that the $\lambda$-operation defining the Hodge decomposition of cyclic homology behaves well with each of the maps S, B, and I.  The decomposition of the other SBI sequence involving cyclic homology and Hochschild homology is usually what is stated in the literature. See \cite{loday1997cyclic} theorem 4.6.9. It is from this that we derive our result.  Combining this Hochschild SBI sequence decomposition with the commutative diagram from proposition 5.1.5 \cite{loday1997cyclic}, we arrive at the result.  This result is explicitly stated in one reference, namely \cite{connes2012cyclic} at the top of page 20.  
\end{proof}

We will now look at the cdh-fibrant versions of these Hodge components.  As luck would have it, they inherit many of the same descent properties of the complete versions of the theories.  
\begin{proposition}
Let $X$ be a scheme over a field $k$ of characteristic zero.  Each of the cdh-fibrant replacements of the sheaves $\textbf{HH}(-/k)$, $\textbf{HC}(-/k)$, $\textbf{HN}(-/k)$, and $\textbf{HP}(-/k)$ admits a natural Hodge decomposition.  This time we mean, taking $H$ to be one of the presheaves above, there exists a collection of presheaves of complexes $H^{(i)}$ such that \[\mathbb{H}_{cdh}(X,H)=\prod \mathbb{H}_{cdh}(X,H^{(i)}).\]  
\end{proposition}
\begin{proof}
This result boils down to the fact that fibrant replacements in a particular topology respect product decompositions.  See \cite{cortinas2008k} theorem 2.1.  
\end{proof}
In addition to having a nice decomposition, various descent properties of $\textbf{HH}(-/k)$, $\textbf{HC}(-/k)$, $\textbf{HN}(-/k)$, and $\textbf{HP}(-/k)$ are inherited by their Hodge decompositions.  
\begin{theorem}
Let $H$ denote any of the sheaves $\textbf{HH}(-/k)$, $\textbf{HC}(-/k)$, and $\textbf{HN}(-/k)$, and let $H^{(i)}$ be the $i$th component of it's Hodge decomposition.  Then, since $H$ satisfies scdh-descent on Sm/$l$ ($k\subset l$), so does each $H^{(i)}$.  Furthermore, for periodic cyclic homology, since $\textbf{HP}(-/k)$ it satisfies cdh-descent on Sch/$l$, then so does $\textbf{HP}^{(i)}(-/k)$.  
\end{theorem}
\begin{proof}
This is a slight modification of theorem 2.4 in \cite{cortinas2008k}.  For a smooth scheme $X$, the quasi-isomorphism $H(X)\cong\mathbb{H}_{scdh}(X,H)=\mathbb{H}_{cdh}(X,H)$ induces a quasi-isomorphism at the level of Hodge components,  $H^{(i)}(X)\cong\mathbb{H}_{scdh}(X,H^{(i)})=\mathbb{H}_{cdh}(X,H^{(i)})$.  For $\textbf{HP}(-/k)$, since $\textbf{HP}(X/k)\cong\mathbb{H}_{cdh}(X,\textbf{HP}(-/k))$ is a quasi-isomorphism for any scheme, this also induces a quasi-isomorphism for it's Hodge components, $\textbf{HP}^{(i)}(X/k)\cong\mathbb{H}_{cdh}(X,\textbf{HP}^{(i)}(-/k))$.  
\end{proof}

For more information about the Hodge decomposition of cyclic homology, see the discussion in chapter 9 of \cite{weibel1995introduction} (or chapter 4 and 5 of \cite{loday1997cyclic}).  For a more in depth account with more of the details worked out, we refer you to \cite{geller1994Hodge}, \cite{weibel1997Hodge}, with the cdh versions being developed in \cite{cortinas2008k}.  

\subsubsection{K-Theory}Now, we will look at the Hodge decomposition of K-theory, and while there is not as much to say as with cyclic homology, this is the area where we will obtain the main results.  In the same manner as before, we will avoid most of the technical discussion of the decomposition itself, and then use it's results.  Very briefly, we note that the Hodge decomposition of K-theory is given by the Adam's operation.  As we will see shortly, the Adam's operation and the $\lambda$-operation from before are related to one another with nice properties.   

\begin{proposition}
For each of the presheaves of spectra $\textbf{K}$, $\textbf{KH}$, and $\tilde{\textbf{K}}$ there exists a Hodge decomposition.  By this we mean, taking $H$ to be one of the presheaves above, there exists a collection of presheaves of complexes $H^{(i)}$ such that for a scheme $X$ (over characteristic 0), \[H(X)=\prod H^{(i)}(X).\]  
\end{proposition}
While we haven't given too many details about Hodge decompositions in any of the theories, the Hodge decompositions behave well with one another.  This is because of the following.  For K-theory, the Hodge decomposition is given by iterated use of the Adam's operation, while in cyclic homology the decomposition is given by the $\lambda$-operation.  In a particular way, the Adams operation commutes with the Chern character from K-theory to negative cyclic homology, essentially (with some modification) mapping it to the corresponding $\lambda$-operation.  This idea is summarized in the following result.  

\begin{theorem}
Let $X$ be a scheme over a field $k$ of characteristic zero.  Then for all $n$, the map $ch:K_n(X)\rightarrow HN_n(X)$ satisfies $(-1)^{k-1}k\cdot \lambda^k(ch)=k\cdot ch(\psi^k)$ (where $\psi$ is the Adam's operation), and therefore $ch$ sends $K_n^{(i)}(X)$ to $HN_n^{(i)}(X)$.  
\end{theorem}
\begin{proof}
See \cite{cortinas2009infinitesimal} theorem 7.1 and corollary 7.2.  
\end{proof}
 
Where we will use this will be our long exact sequence connecting cyclic homology cdh-fibrant cyclic homology, and K-theory.   Recall our long exact sequence from \ref{mainlong} below.  
\[...\rightarrow\tilde{K}_{n+1}(X)\rightarrow HC_n(X)\rightarrow\mathbb{H}^{-n}_{cdh}(X,\textbf{HC})\rightarrow\tilde{K}_{n}(X)\rightarrow...\]
As before, and using the fact that the Chern character behaves well with Hodge decompositions, we have that this sequence breaks up along Hodge components.  The only thing different this time is that there is a degree shift in the Hodge components. 

\begin{theorem}For a scheme $X$ over a field of characteristic zero, there is a long exact sequence of Hodge components given by the following.  
\[...\rightarrow\tilde{K}_{n+1}^{(i+1)}(X)\rightarrow HC_n^{(i)}(X)\rightarrow\mathbb{H}^{-n}_{cdh}(X,\textbf{HC}^{(i)})\rightarrow\tilde{K}_{n}^{(i+1)}(X)\rightarrow...\]
\end{theorem}
\begin{proof}
This all comes from the Chern character commuting with the Hodge decompositions.  The degree shift comes from the shift in the SBI sequence relating negative cyclic homology (the original target of the Chern character) to cyclic homology, as stated in proposition \ref{Hodgesbi}.  This idea is presented in \cite{cortinas2009k} after theorem 1.2.
\end{proof}

Using this theorem, and some of our previous results, we will be able to make some immediate computations for some Hodge components.  
\begin{corollary}\label{kHodgezero}
Let $X$ be a scheme.  For all $n$, $\tilde{K}^{(0)}_{n}(X)=0$.  
\end{corollary}
\begin{proof}
The Hodge decomposition of both cyclic homology and cdh-fibrant cyclic homology are indexed by $\mathbb{N}$.  This means that for any negative indices, their Hodge components must be zero. Since there is a shift in index for K-theory, we know that 0th Hodge component of K-theory must map to the -1st component of cyclic homology and then cdh-fibrant cyclic homology.  Therefore the long exact sequence of Hodge components is reduced to 
\[...\rightarrow 0\rightarrow \tilde{K}^{(0)}_{n}(X)\rightarrow 0\rightarrow 0\rightarrow\tilde{K}^{(0)}_{n-1}(X)\rightarrow 0\rightarrow ...\]
and we conclude that $\tilde{K}^{(0)}_{n}(X)$ must be zero for all $n$.  
\end{proof}

\begin{corollary}\label{kHodgenegative}
Let $X$ be affine, Noetherian, and finite dimensional. For all $n<0$, \[\tilde{K}^{(i+1)}_{n}(X)=\mathbb{H}^{-n}_{cdh}(X,\textbf{HC}^{(i)}).\]
\end{corollary}
\begin{proof}
Recalling from corollary \ref{cycliczero}, we know that $HC_n(X)=0$ for all $n<0$ (also meaning that each of it's components are zero).  Using the long exact sequence, we see that \[...\rightarrow  \tilde{K}^{(0)}_{n+1}(X)\rightarrow 0\rightarrow \mathbb{H}^{-n}_{cdh}(X,\textbf{HC}^{(i)})\rightarrow\tilde{K}^{(i+1)}_{n}(X)\rightarrow 0\rightarrow ...\] and therefore $\tilde{K}^{(i+1)}_{n}(X)=\mathbb{H}^{-n}_{cdh}(X,\textbf{HC}^{(i)})$ for all $n<0$.
\end{proof}

As of right now, this is the extent of our general knowledge about the Hodge components of K-theory.  This is precisely where our new research fits in.  Since much more is known about the Hodge decomposition of cyclic homology, and we have these nice theorems relating the decompositions of cyclic homology to that of K-theory, we will use this to build our new results.

\newpage
\section{New ideas I: Secondary results}
Finally we arrive at our new results.  In the first section we develop a method for computing the various versions of cyclic homology by using their descent properties.  Later, we will apply this method to particular examples such as the cyclic homology of hypersurfaces with isolated singularities. 

\subsection{Variants of cyclic homology on singular schemes}
Our goal will be to use descent properties of the various versions of cyclic homology to do computations on singular varieties.  Based on it's strong descent properties, periodic cyclic homology will be the easiest to compute on singular varieties.  
Let $X$ be a singular variety over a field $k$ of characteristic zero.  As we have mentioned many times before, we are guaranteed the existence of the following abstract blowup square corresponding to the resolution of $X$.  We have, 
\[\begin{CD}
E @>>> Y\\
@VVV @VVV\\
Z @>>> X
\end{CD} \] 
where $Y$ is a smooth variety, $Z$ is something like the singular locus (which may be singular but it is of smaller dimension) and $E$ is a union of smooth hyperplanes with normal crossings (at the worst, and it is also of smaller dimension).  We recall our the remark given at the end of section 2, where because we can iterate this process of resolving singularities for these smaller dimensional pieces, we can basically assume that each of $Z$ and $E$ are smooth.  

In this scenario, based on the descent properties of periodic cyclic homology, and assuming that each of the 3 varieties $E$, $Y$, and $Z$ are smooth varieties, we have the following result.
\begin{theorem}\label{6term}
Let $X$ be a singular variety over a field $k$ of characteristic zero.  Then there exists the following 6 term long exact sequence of de Rham cohomology groups and periodic cyclic homology groups.   

\[\begin{CD}
HP_0(X/k) @>>> H_{DR}^{ev}(Y/k)\oplus H_{DR}^{ev}(Z/k) @>>> H_{DR}^{ev}(E/k) \\
@AAA @. @VVV\\
H_{DR}^{odd}(E/k) @<<< H_{DR}^{odd}(Y/k)\oplus H_{DR}^{odd}(Z/k) @<<< HP_1(X/k) 
\end{CD} \]

\vspace{5 mm}

\noindent (Recall, $H_{DR}^{ev}(A/k)=\prod_{i\geq0}H_{DR}^{2i}(A/k)$ and $H_{DR}^{odd}(A/k)=\prod_{i\geq0}H_{DR}^{2i+1}(A/k)$.)
\end{theorem}
\begin{proof}
Based on the previous remark, we will assume that each variety besides $X$ in the blowup square, $E$, $Y$, and $Z$, are all smooth.  Since $\textbf{HP}(-/k)$ satisfies cdh-descent, by applying the sheaf $\textbf{HP}(-/k)$ to the square above we have that 
\[\begin{CD}
\textbf{HP}(E/k) @<<< \textbf{HP}(Y/k)\\
@AAA @AAA\\
\textbf{HP}(Z/k) @<<< \textbf{HP}(X/k)
\end{CD} \]
is a homotopy Cartesian square.  Therefore, when passing to the homotopy groups, we have the inherited long exact sequence for fibrations or rather the Mayer-Vietoris sequence from cdh-cohomology.  
\[...\rightarrow HP_n(X/k)\rightarrow HP_n(Y/k)\oplus HP_n(Z/k)\rightarrow HP_n(E/k)\rightarrow HP_{n-1}(X/k)\rightarrow...\]
Recalling that periodic cyclic homology is actually periodic of period 2, we are able to reduce this long exact sequence to a 6 term long exact sequence that circles around on itself.  Finally, since each of these varieties, except for $X$, is smooth over a field of characteristic zero, we can use theorem \ref{hpderham2} to write them in terms of de Rham cohomology, which proves the result.  
\end{proof}

This shows that we can completely determine the periodic cyclic homology of a singular variety if we know the de Rham cohomology of it's resolution and that of the centers and exceptional fiber of the resolution.  

This by itself is a very nice result but, although difficult to find explicitly stated in the literature, is probably very well known.  Where we will use this is to determine the negative cyclic homology of a singular variety when the cyclic homology is already known.  

Suppose we have a ring $A$, or an affine scheme $X=$ Spec$(A)$, where we already know the cyclic homology of that ring.  We might want to figure out the negative cyclic homology and the periodic cyclic homology as well.  We already have a way of relating cyclic homology to it's negative and periodic versions, this is the SBI sequence given again by the following.  
\[
...\overset{S}{\rightarrow} HC_{n-1}(A/k)\overset{B}{\rightarrow} HN_{n}(A/k) \overset{I}{\rightarrow} HP_n(A/k) \overset{S}{\rightarrow} HC_{n-2}(A/k) \overset{B}{\rightarrow} ...
\]
The problem is, this sequence by itself is not enough to determine the negative and periodic groups.  However, using our results from above for periodic cyclic homology, we know two out of three of the groups of groups in this sequence.  So, by combining this SBI sequence with our findings for periodic cyclic homology, we can try to determine the unknowns in the long exact sequence.   We have the following result, which is the first of our secondary main theorems.   

\begin{corollary}
Let $X$ be a singular variety over a field of characteristic zero, and let $Y$, $Z$, and $E$ be the resolution, center, and exceptional fiber of the resolution respectively (each assumed to be smooth varieties).  Then there is a diagram which is a combination of two long exact sequences.  A rough approximation of this diagram is as follows (where $n$ is even).  The solid arrows are the long exact SBI sequence, and the dashed arrows are the six term exact sequence from above.  
\end{corollary}

\noindent
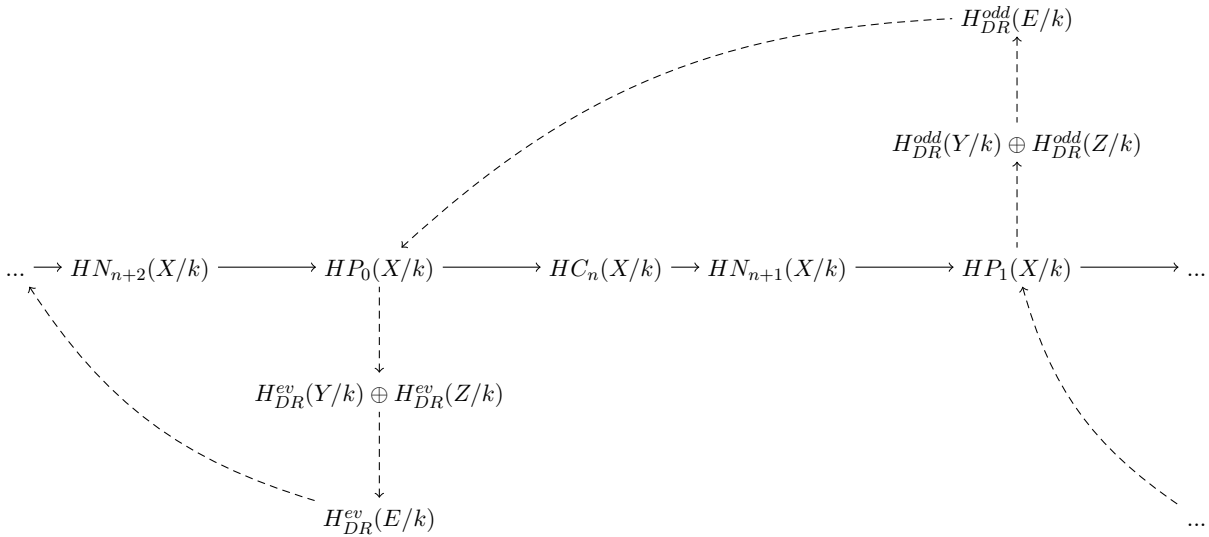
\begin{tikzpicture}[scale=.82, every node/.style={scale=.82} ]
\matrix (m) [matrix of math nodes, row sep=3em,
column sep=1em, text height=2ex, text depth=0.2ex]
{ &  &   & & & H_{DR}^{odd}(E/k) &\\
  &  &   & & & H_{DR}^{odd}(Y/k)\oplus H_{DR}^{odd}(Z/k) &\\
...& HN_{n+2}(X/k) & HP_0(X/k) & HC_{n}(X/k) & HN_{n+1}(X/k) & HP_1(X/k)& ...\\
&  & H_{DR}^{ev}(Y/k)\oplus H_{DR}^{ev}(Z/k)  & & & &  \\ 
&  & H_{DR}^{ev}(E/k)  & & & &... \\ };
\path[->]
(m-3-1) edge (m-3-2);
\path[->]
(m-3-2) edge (m-3-3);
\path[->]
(m-3-3) edge (m-3-4);
\path[->]
(m-3-4) edge (m-3-5);
\path[->]
(m-3-5) edge (m-3-6);
\path[->]
(m-3-6) edge (m-3-7);
\path[<-][densely dashed]
(m-5-3) edge (m-4-3);
\path[<-][densely dashed]
(m-4-3) edge (m-3-3);
\path[<-][densely dashed]
(m-1-6) edge (m-2-6);
\path[<-][densely dashed]
(m-2-6) edge (m-3-6);
\path[<-][densely dashed]
(m-3-1) edge [bend right=20] (m-5-3);
\path[<-][densely dashed]
(m-3-3) edge [bend left=20] (m-1-6);
\path[<-][densely dashed]
(m-3-6) edge [bend right=20] (m-5-7);
\end{tikzpicture}

Summarizing our findings up to this point, we have determined the following.  For a singular variety $X$ over a field of characteristic zero, we can always determine the periodic cyclic homology of $X$ by looking at the de Rham cohomology of the various spaces associated to it's resolution.  Furthermore, when the cyclic homology of this variety is already known, we can use this as well as our computations for periodic cyclic homology in the SBI sequence to determine the negative cyclic homology.  

\begin{note}
In general, there may be other ways to determine the two variations of cyclic homology when the cyclic homology is known.  Since the complexes defining the three theories are not that different, it would seem that whatever algebraic methods used to determine the original cyclic homology could be used to determine it's variants.  Such an example of this is given by \cite{emmanouil1995cyclic} where the author computed periodic cyclic homology in terms of infinitesimal cohomology. 

However, any method to extend these computations would have to be tailored to fit the specific scenario, and there may be unforeseen nuances that would make extending these computations difficult.  This method bypasses all of this, and leaves the nuance at the level of resolution of singularities.  Furthermore, as long as we are working with varieties over a field of characteristic zero, Hironaka's theorem guarantees that we can use this framework to compute the periodic cyclic homology (and then if the cyclic homology is known, to then compute the negative cyclic homology).  
\end{note}

\subsection{Cyclic homology of hypersurfaces with isolated singularities}

Using the framework developed in the previous section, we will attempt to apply our results to cases where the cyclic homology of a singular variety is known.  Our first application will be to extend the following theorem by Michler \cite{michler1997cyclic} to include computations about the negative cyclic homology and periodic cyclic homology.  
\begin{theorem}[Michler]
Let $k$ be an algebraically closed field of characteristic zero, and let $A$ be a reduced affine hypersurface over $k$ with only isolated singularities in affine $N$-space.  Then the Hodge components of cyclic homology are given by:
For $n>N$
\[ 
HC_n^{(i)}(A/k) \cong 
\begin{cases}
 T(\Omega_{A/k}^{N-1})\oplus H_{DR}^{N-1}(A/k) & \text{if }2i-n=N-1 \\
 H_{DR}^{2i-n}(A/k) & \text{otherwise}
\end{cases}\]
Where $T(\Omega_{A/k}^{N-1})$ is the torsion submodule of the $(N-1)$st exterior power of the K$\ddot{a}$hler differentials.  

For $n\leq N$
\[ 
HC_n^{(i)}(A/k) \cong 
\begin{cases}
\Omega_{A/k}^{n}/d\Omega_{A/k}^{n-1}& \text{if }i=n, \\
 H_{DR}^{2i-n}(A/k) & \text{for }n/2\leq i<n, \\
 0 & \text{otherwise}
\end{cases}.\]
\end{theorem}

This theorem references the Hodge decomposition of cyclic homology.  As mentioned before, this will be the setting in which we state our results.  We will now try to extend these computations of the Hodge components to the other two variants of cyclic homology

\begin{theorem}
Let $X$ be a singular variety over a field $k$ of characteristic zero.  Let $Y$, $Z$, and $E$ be the resolution, center, and exceptional fiber of the resolution respectively (all assumed to be smooth varieties).  Then for each $i$, there exists the following long exact sequence of de Rham cohomology groups and Hodge components of periodic cyclic homology groups.   
\[\rightarrow H_{DR}^{2i-n-1}(E/k) \rightarrow HP^{(i)}_{n}(X/k)\rightarrow H_{DR}^{2i-n}(Y/k)\oplus H_{DR}^{2i-n}(Z/k) \rightarrow H_{DR}^{2i-n}(E/k) \rightarrow HP^{(i)}_{n-1}(X/k) \rightarrow 
 \]
\end{theorem} 
\begin{proof}
The proof of this is basically the same as the proof of theorem \ref{6term}.  
\end{proof}

Combining this result, the SBI-sequence, and the theorem of Michler, we arrive at our first application.  
\begin{theorem}
Let $k$ be an algebraically closed field of characteristic zero, and let $X$ be a reduced affine hypersurface over $k$ with only isolated singularities in affine $N$-space.  Let $Y$, $Z$, and $E$ be the resolution, center, and exceptional fiber of the resolution respectively (all assumed to be smooth).  Then the Hodge components of negative cyclic homology can be computed using a combination of two long exact sequences, namely the Mayer-Vietoris sequence for cdh-cohomology and the SBI-sequence.  This theorem will be restated with diagrams in appendix \ref{cmich}.  
\end{theorem}

\begin{proof}
Essentially all we are doing is splitting the result of the previous section into Hodge components.  Using the previous theorem, and the fact that Hodge decompositions behave well with these long exact sequences, we have an explicit hybrid of two long exact sequences for each Hodge component.  Using this combination, there are two unknowns, the negative cyclic homology and the cyclic homology.  Applying the results of Michler, we now know the cyclic homology in this scenario.  Hence the negative cyclic homology is also computable. 
\end{proof}

Using our previous result with corollary \ref{cycliczero} will allow us to write a more concise result for negative cyclic homology in negative degrees.  

\begin{corollary}In the scenario of the previous theorem, for $n<0$, the Hodge components of the negative cyclic homology of $X$ can be computed using the following long exact sequence.  
\[\rightarrow H_{DR}^{2i-n-1}(E/k) \rightarrow HN^{(i)}_{n}(X/k)\rightarrow H_{DR}^{2i-n}(Y/k)\oplus H_{DR}^{2i-n}(Z/k) \rightarrow H_{DR}^{2i-n}(E/k) \rightarrow HN^{(i)}_{n-1}(X/k) \rightarrow 
 \]
\end{corollary}

The previous corollary does not actually use the results of Michler, we are really just using the fact that the cyclic homology is zero for negative degrees from \ref{cycliczero}.  The reason that it is stated as a corollary is that it borrows the same notation from the previous theorem and is really actually a corollary to \ref{cycliczero}.  

\newpage
\section{New ideas II: Main results}
We will now look at a different type of result.  For our previous original results, we have shown that we are able to compute the variations on the cyclic homology groups from a known cyclic homology computation.  One of our problems was that we were unable to write things concisely in terms of a known explicit resolution of singularities. While in these previous cases we have that the various components of the resolution are known to exist, they are not known.  This led to having several hybrid long exact sequence diagrams (appearing in the appendices) with various de Rham cohomology groups in terms of these nebulous pieces of resolutions.   

Now, for these new results we have the opposite information, and hence opposite problems.  This time, we know the explicit resolutions, but we have either only partial or no information about the cyclic homology.  

\subsection{The K-theory of filtered deformations of cones over smooth varieties}
Before we state and prove our main theorem, we will have to develop some more background information for our specific scenario. 

Most of our main results have grown out of the study of the following lemma.  
\begin{lemma}\label{localblowup}
Let $A$ be a local ring with maximal ideal $\mathfrak{m}$, and let $X=$ Spec$(A)$.  For the Blowup of $X$ along some ideal $I\subset A$, $Y=$ Proj$(\oplus_{i=0}^{\infty} I^i)$, with exceptional fiber \\ $E=$ Proj$(\oplus_{i=0}^{\infty} I^i/I^{i+1})$, if $E$ is smooth then $Y$ is also smooth and therefore a resolution of $X$.  
\end{lemma}
\begin{proof}This result is stated informally in \cite{o1997smoothness} at the beginning of section 2. Their argument is completely incomprehensible, therefore we have decided to include a cleaner version of their proof.  

Let $I=\langle f_1,...,f_r\rangle$, and assume that $E=$ Proj$(\oplus_{i=0}^{\infty} I^i/I^{i+1})$ is smooth. In order to show that $Y$ is smooth, we need to show that at every (closed) point $P$, the local ring at $P$ of the scheme $Y$ is a regular local ring. Recall that the Proj construction takes a graded ring $(\oplus_{i=0}^{\infty} R_i)$, and defines a space using the collection of all relevant homogeneous prime ideals.  A relevant ideal is defined by the property that it does not contain the irrelevant ideal, meaning that $J\not\supset\oplus_{i>0} R_i$ (the irrelevant ideal of $\oplus_{i=0}^{\infty} R_i)$. 

Let $P$ be a closed point on $Y$, hence $P$ is a maximal relevant ideal.  First we note by the maximality of $P$ that $\mathfrak{m}$ (the maximal ideal of $A$ that is now included in the blowup algebra) is contained in (the zero-ith graded piece of) $P$.  Recall that we use that notation $A[It]$ for $(\oplus_{i=0}^{\infty} I^i)$ so that the indeterminate $t$ can keep track of the grading.  Since $P$ is a relevant ideal we know that for some $i\in\{1,...r\}$, $f_it\notin P$, otherwise, $P$ would contain the irrelevant ideal, and hence would not be a point on $Y$.  We do know however that $P$ contains the remaining $f_jt$, otherwise it would contradict the maximality of $P$.  

Now, if we look at the localization of $A[It]$ at $P$, we have that $IA[It]_P=f_iA[It]_P$.  The reason for this is a two way containment argument.  First, since $f_i\in I$, we know that $IA[It]_P\supset f_iA[It]_P$.  Next we want to show that for any element in $IA[It]_P$, it can also be written as an element in $f_i A[It]_P$, hence the other containment.  We have that a typical element in $IA[It]_P$ looks like a linear combination of the generating set of elements.  By looking at one summand from such a linear combination, it will make the calculation easier, and we can recombine later to recreate a typical element.  Without loss of generality, take $f_j$ to be some element in the generating set of $I$.  We want to rewrite $f_j\frac{a}{b}$, where $a\in A[It]$ and $b\in A[It]\setminus P$, as an element from $f_iA[It]_P$ which would look like $f_i\frac{c}{d}$ where $c\in A[It]$ and $d\in A[It]\setminus P$.  Setting $c=f_jta$ and $d=f_itb$ (recall we can do this because $f_it\notin P$) we have that,
\[f_j\frac{a}{b}=f_j\frac{(f_it)a}{(f_it)b}=f_i\frac{f_jta}{f_itb}=f_i\frac{c}{d}.\]
So, we see that (after a proper argument involving a linear combination of these terms) any element in $IA[It]_P$ is also an element of $f_iA[It]_P$.  Finally we conclude that $IA[It]_P\subset f_iA[It]_P$, and by two way containment, $IA[It]_P=f_iA[It]_P$.  

An important fact that we will need is that $f_i$ is a nonzero divisor on $A[It]_P$.  This is simply because the localization map $A[It]\rightarrow A[It]_P$ that sends $x\rightarrow x/1$ preserves the property of being a zero divisor.  So, since $f_i$ is not a zero divisor in $A[It]$, $f_i/1=f_i$ is not a zero divisor in $A[It]_P$.  

Recall that the associated graded ring, $\oplus_{i=0}^{\infty} I^i/I^{i+1}$ can be rewritten as a quotient of $A[It]$, $\oplus_{i=0}^{\infty} I^i/I^{i+1}=A[It]/IA[It]$.  Using this fact and the correspondence theorem for ideals and ideals in quotient rings, we know that since $IA[It]\subset P$, that $P/IA[IT]$ is a relevant homogeneous ideal in $\oplus_{i=0}^{\infty} I^i/I^{i+1}=A[It]/IA[It]$. 

Localizing at $P/IA[IT]$, we know that $(A[It]/IA[It])_{P/IA[It]}$.  By properties of localization (the fact that it commutes with quotients), and by the fact that $IA[It]_P=f_iA[It]_P$, we know that \[(A[It]/IA[It])_{P/IA[It]}=A[It]_P/IA[It]_P=A[It]_P/f_iA[It]_P.\]  Furthermore, since $E=$ Proj$(A[It]/IA[It])$ is assumed to be smooth, we know that $A[It]_P/f_iA[It]_P$ is a regular local ring.  By corollary 4.4.13 of \cite{weibel1995introduction}, we know that the global dimension of $A[It]_P$ is equal to the global dimension of $A[It]_P/f_iA[It]_P$ plus one.  Since $A[It]_P/f_iA[It]_P$ is regular, we know that it has finite global dimension, so therefore $A[It]_P$ also has finite global dimension.  We conclude that $A[It]_P$ is a regular local ring.  

Finally, since $P$ was taken as an arbitrary closed point on $Y$ and the local ring at $P$ of the scheme $Y$ is a regular local ring, we know that $Y$ is smooth, which completes the proof.  
\end{proof}

An immediate corollary to this result that we can adapt from our cdh-cohomology framework is the following.  
\begin{theorem}\label{thcole}
Let $E=$ Proj$(A)$ be a smooth projective scheme over some field $k$ of characteristic zero, and let $R$ be a local ring such that $R$ is a filtered deformation of $A$ along some ideal filtration for some ideal $I\subset R$ such that $R/I$ is smooth and Spec$(R)$ is a scheme essentially of finite type over $k$.  Then for each $n<0$ and $i\geq0$, there is a long exact sequence of Hodge components of K-theory and cyclic homology given by 
\[...\rightarrow HC_{n+1}^{(i)}(E)\rightarrow\tilde{K}^{(i+1)}_{n}(R)\rightarrow HC_{n}^{(i)}(\text{Proj}(R[It])) \rightarrow HC_{n}^{(i)}(E)\rightarrow\tilde{K}^{(i+1)}_{n-1}(R)\rightarrow... \] 
\end{theorem}

\begin{proof}
First, we recall that a filtered deformation is defined by the use of the $gr$ construction, meaning that under the assumptions,  $A=gr_I(R)$.  Substituting this into a blowup square, we have that the blowup of Spec$(R)$ along the subscheme defined by $I$ looks like the following.  
\[\begin{CD}
\text{Proj}(gr_I(R))  @>>> \text{Proj}(R[It])\\
@VVV @VVV\\
\text{Spec}(R/I)@>>> \text{Spec}(R)
\end{CD}
\hspace{.5 in}=\hspace{.5 in}
\begin{CD}
E=\text{Proj}(A)  @>>> \text{Proj}(R[It])\\
@VVV @VVV\\
\text{Spec}(R/I)@>>> \text{Spec}(R)
\end{CD}\]
Since $R$ is a local ring, and the exceptional fiber $E$ is assumed to be smooth, we have by lemma \ref{localblowup} that Proj$(R[It])$ is smooth, and therefore this is a resolution of Spec$(R)$.  Furthermore, by another assumption, the center of the blowup or the subscheme defined by Spec$(R/I)$ is also smooth.  

So, we have an abstract blowup square where each of the schemes in the square, except for Spec$(R)$ are smooth.  Using the smoothness of each of these schemes we have that for each $i\geq 0$ the Mayer-Vietoris sequence for cdh-fibrant cyclic homology, 
\[...\rightarrow\mathbb{H}_{cdh}^{-n}(\text{Spec}(R),\textbf{HC}^{(i)})\rightarrow\mathbb{H}_{cdh}^{-n}(\text{Proj}(R[It]),\textbf{HC}^{(i)})\oplus\mathbb{H}_{cdh}^{-n}(\text{Spec}(R/I),\textbf{HC}^{(i)})\hspace{1 in}\]\[\hspace{3in} \rightarrow\mathbb{H}_{cdh}^{-n}(E,\textbf{HC}^{(i)})\rightarrow\mathbb{H}_{cdh}^{-n+1}(\text{Spec}(R),\textbf{HC}^{(i)})\rightarrow... .\]
reduces to 
\[...\rightarrow HC_{n+1}^{(i)}(E)\rightarrow\mathbb{H}_{cdh}^{-n}(\text{Spec}(R), \textbf{HC}^{(i)})\rightarrow HC_{n}^{(i)}(\text{Proj}(R[It]))\oplus HC_{n}^{(i)}(\text{Spec}(R/I))\hspace{.8 in}\]\[\hspace{3.25 in} \rightarrow HC_{n}^{(i)}(E)\rightarrow\mathbb{H}_{cdh}^{-n+1}(\text{Spec}(R), \textbf{HC}^{(i)})\rightarrow... .\]

Under the assumption that $n<0$ we know two things.  The first comes from \ref{cycliczero}, which tells us that in the sequence above, since Spec$(R/I)$ is affine, Noetherian, and finite dimensional, that $HC_{n}^{(i)}(\text{Spec}(R/I))=0$.  The second comes from \ref{kHodgenegative}.  This tells that, again since Spec$(R)$ is now affine, Noetherian, and finite dimensional, that 
\[\tilde{K}^{(i+1)}_{n}(R)=\tilde{K}^{(i+1)}_{n}(\text{Spec}(R))=\mathbb{H}^{-n}_{cdh}(X,\textbf{HC}^{(i)})\]

So, putting this together, for $n<0$, we have that following exact sequence of K-groups and cyclic homology groups.  
\[...\rightarrow HC_{n+1}^{(i)}(E)\rightarrow\tilde{K}^{(i+1)}_{n}(R)\rightarrow HC_{n}^{(i)}(\text{Proj}(R[It])) \rightarrow HC_{n}^{(i)}(E)\rightarrow\tilde{K}^{(i+1)}_{n-1}(R)\rightarrow... .\]
thus completing the proof of the theorem.
\end{proof}

This is already a nice result, but again, we want to be able to easily compute the pieces of the cyclic homology in the resulting long exact sequence.  Later we will try to rephrase a similar long exact sequence from the main theorem so that we do just that: make the cyclic homology easier to compute. 

As a preview, we will rewrite the cyclic homology groups in terms of the hypercohomology groups of the truncated de Rham complex.  From \ref{cyclichyperHodge}, we know that since $E$ and Proj$(R[It])$ are smooth schemes we have that the Hodge components of their cyclic homology coincide with hypercohomology.   We have then that \[HC_n^{(i)}(E)=\mathbb{H}^{2i-n}(E,\Omega_{-/\mathbb{Q}}^{\leq i}),\hspace{.25 in}\text{and}\hspace{.25 in}HC_n^{(i)}(\text{Proj}(R[It]))=\mathbb{H}^{2i-n}(\text{Proj}(R[It]),\Omega_{-/\mathbb{Q}}^{\leq i}).\]
Replacing those cyclic homology groups in the previous long exact sequence with these hypercohomology groups yields a slightly more computable long exact sequence. 
\begin{corollary}
In the scenario of the previous theorem, the resulting long exact sequence can be rewritten in the following way.  
\[...\rightarrow \mathbb{H}^{2i-n-1}(E,\Omega_{-/\mathbb{Q}}^{\leq i})\rightarrow\tilde{K}^{(i+1)}_{n}(R)\rightarrow \mathbb{H}^{2i-n}(\text{Proj}(R[It]) \rightarrow \mathbb{H}^{2i-n}(E,\Omega_{-/\mathbb{Q}}^{\leq i})\rightarrow\tilde{K}^{(i+1)}_{n-1}(R)\rightarrow... \] 
\end{corollary}
\begin{proof}
This proof was stated before the corollary.  
\end{proof}
With this result, we are already able to see some concrete applications.  We will state one possible application, but before we do so, we will require the following lemma.    
\begin{lemma}\label{grlocalization}
Let $R$ be a Noetherian ring equipped with the filtration by an ideal $I\subset R$.  Let $P$ be a prime ideal in $R$ that contains $I$, and let $S=R/I\setminus P/I$.  Then $gr_I(R)_{St^0}\cong gr_{I_{P}}(R_P)$.  
\end{lemma}
\begin{proof}
This is stated similar to how we have written it at the beginning of section 1.1 of \cite{giorgi2006irreducible} (although they state it incorrectly, or at least it is unclear exactly what they mean).  There is more discussion and there are some more of the details worked out in an alternative (but equivalent) form on pages 53, 54 of \cite{herrmann1988equimultiplicity}.  The proof of this is not really stated in any of the references, so we will give a rough proof here.  

For this proof, we will use the fact that the associated graded ring is a quotient of the Rees algebra $R[It]$.  By doing this we will inherit several localization facts from the formation of the Rees algebra.  

First, localization commutes with the formation of the Rees algebra in the following way.  For a multiplicatively closed subset $T\subset R$, we have that $Tt^0\subset R[It]$ is a multiplicatively closed subset of the (zero-ith graded piece of the) Rees algebra.  So, the localizing by $T$ and then forming the Rees algebra is the same as forming the Rees algebra and then localizing by $Tt^0$, $R_T[I_Tt]=R[It]_{Tt^0}$. Applying this to our scenario, we are localizing by a prime $P$, hence our multiplicatively closed subset is $T=R\setminus P$.  We have then that $R_P[I_Pt]=R_T[I_Tt]=R[It]_{Tt^0}$.  

Recalling the quotient construction of the associated graded ring now, we have that \[R_P[I_Pt]/I_PR_P[I_pt]=R[It]_{Tt^0}/(IR[It])_{Tt^0}.\]  In addition to commuting with the formation of the Rees algebra, localization commutes with quotients in the following way.  Let $U$ be the image of $T$ in the quotient map $R[It]\rightarrow R[It]/IR[It]$, then $R[It]_T/(IR[It])_T=(R[It]/IR[It])_U$.  In our case, since the map from the Rees algebra to the associated graded ring is a graded morphism, the image of $Tt^0$ in $R[It]/IR[It]$ is $U=(T/I)t^0$.  We see that $T/I=R/I\setminus P/I=S$, so, from the assumptions of the lemma, we have that 
\[gr_{I_{P}}(R_P)=R_P[I_Pt]/I_PR_P[I_pt]=(R[It]/IR[It])_{St^0}=gr_I(R)_{St^0}.\]
\end{proof}
\noindent This lemma will also be useful later on.

Immediately from this lemma, we have the following application of \ref{thcole}.  
\begin{theorem}[Application]\label{app1}
Let $A=k[x_1,...,x_n]/J$ be a graded ring with $J$ a proper homogeneous ideal of \\ $k[x_1,...,x_n]$, and $E=$ Proj$(A)$ is smooth projective scheme.  Let \\ $R=(k[x_1,...,x_n]/I)_{\langle x_1,...,x_n \rangle}$ such that $I_{min}\cong J$.  Then $R$ is a filtered deformation of $A$ along the ideal filtration given by it's maximal ideal $\mathfrak{m}_R$, and therefore (because it also satisfies the other assumptions) it fits into the long exact sequence from theorem \ref{thcole}. 
\end{theorem}
\begin{proof}
We want to show that $R$ is a filtered deformation of $A$ along the filtration given by it's maximal ideal.  By the definition of filtered deformation, we simply need to show that $gr_{\mathfrak{m}_R}(R)=A$.  Let $R'$ be the unlocalized version of $R$, i.e. $R'=k[x_1,...,x_n]/I$.  We know by lemma \ref{associatedgraded}, that for the ideal filtration $\mathfrak{m}=\langle x_1,...,x_n\rangle$ on $R'$, that $gr_{\mathfrak{m}}(R')=k[x_1,...,x_n]/I_{min}=k[x_1,...,x_n]/J$.  We also know that for $S=R/\mathfrak{m}_R\setminus\mathfrak{m}_R/\mathfrak{m}_R=k^*$ (the units in $k$), by the previous lemma (\ref{grlocalization}), we have that $gr_{\mathfrak{m}_{\mathfrak{m}}}(R'_{\mathfrak{m}})=gr_{\mathfrak{m}}(R')_{S}=gr_{\mathfrak{m}}(R')$.  Therefore, putting everything together, we have that, 
\[gr_{\mathfrak{m}_R}(R)=gr_{\mathfrak{m}_{\mathfrak{m}}}(R'_{\mathfrak{m}})=gr_{\mathfrak{m}}(R').\]
So, we have shown that $gr_{\mathfrak{m}_R}(R)=A$ and therefore $R$ is a filtered deformation of $A$.  

As far as the other assumptions of the previous theorem, they follow almost by definition.  We have that Spec$(R/\mathfrak{m}_R)=$ Spec$(k)=\{0\}$ which is just a point so it is obviously smooth.  And we know that by definition, since $R$ is a localization of a finitely generated $k$-algebra, that Spec$(R)$ is a scheme essentially of finite type over $k$, and this completes the proof for each of the assumptions of the previous theorem.  

So, we have shown that each of the objects above satisfy the assumptions of the theorem, and therefore they are just one example of applications of that theorem, and they fit into the long exact sequence from the conclusion of theorem \ref{thcole}.  
\end{proof}
While this is not necessarily the most interesting geometric example, it is still a computation where we can gain concrete information.  
\begin{warning}
Completions, and quotients of formal power series rings are rings that cannot be used in the application of the previous theorem. This is because, while they are nice local rings, they do not form schemes of essentially finite type over a field.  
\end{warning}

All of our results so far have dealt with localizations, and local data.  From this local data, and in particular their resolutions, we will try to gain some information about global data and resolution of singularities.  We begin with the following lemma.  
\begin{lemma}\label{globalblowup}
Let $R=k[x_1,..,x_n]/I$ such that $X=$ Spec$(R)$ has only one isolated singularity (say at the origin $\mathfrak{m}=\langle x_1,...,x_n\rangle$), and let $Y=$ Proj$(R[\mathfrak{m}t])$ be the blowup at the origin.  Now, let $R_{\mathfrak{m}}$ be the localization of $R$ at $\mathfrak{m}$, and let $Y'=$ Proj$(R_{\mathfrak{m}}[\mathfrak{m}_\mathfrak{m}t])$ be the blowup of Spec$(R_{\mathfrak{m}})$ at the point $\mathfrak{m}_{\mathfrak{m}}$.  If the blowup of the local scheme, $Y'$, is smooth, then so is $Y$, and hence $Y$ is a resolution of $X$.  
\end{lemma}
While this result is probably known, it has so far been impossible to find it explicitly stated in the literature. While geometrically this makes sense, we will give a more algebraic proof.  
\begin{proof}
One useful property of a blowup is that $\pi:Y\rightarrow X$ is an isomorphism in the complement of the center of the blowup.  Since $X$ is assumed to have only one isolated singularity at the point $\mathfrak{m}$, in order to determine whether or not the blowup $Y$ is smooth, we simply need to check points on the fiber over $\mathfrak{m}$.  

Another useful property of blowups is that they commute with localization.  By this we mean the following.  For a multiplicatively closed set $S\in R$, $St^0$ is a multiplicatively closed set in $R[\mathfrak{m}t]$.  Hence, for $S=R\setminus\mathfrak{m}$, \[R_{\mathfrak{m}}[\mathfrak{m}_\mathfrak{m}t]=R_{S}[\mathfrak{m}_St]=R[\mathfrak{m}t]_{St^0}.\]  
Therefore, we have that Proj$(R_{\mathfrak{m}}[\mathfrak{m}_\mathfrak{m}t])=$ Proj$(R[\mathfrak{m}t]_{St^0})$, and the blowup commutes with localization.  

With these properties in mind, we can now complete the proof.  Let $\pi:Y\rightarrow X$ be the blowup of $X$ with the map $\pi$ induced by the inclusion of $\varphi:R\rightarrow R[\mathfrak{m}t]$ where $r\mapsto rt^0$.  We can characterize the points in the fiber over $\mathfrak{m}$ in the following way.  For the inclusion map $\varphi$ from above and for a prime ideal $P\subset R[\mathfrak{m}t]$, $\pi(P)=\mathfrak{m}$ (as points) if $\varphi^{-1}(P)=\mathfrak{m}$.  Furthermore, the only way that $\varphi^{-1}(P)=\mathfrak{m}$ can happen is if $P\cap R[\mathfrak{m}t]^0=\mathfrak{m}$ (where $R[\mathfrak{m}t]^0$ denotes the zero-ith graded piece of $R[\mathfrak{m}t]$).  

We want to show that $Y$ is nonsingular at $P$.  In order to do this, we simply need to show that the stalk of the structure sheaf at $P$ (which is the localization of $R[\mathfrak{m}t]$ at $P$, $R[\mathfrak{m}t]_P$) is a regular local ring.  This localization is related to the localization from the assumptions in the following way.  For $S=R\setminus\mathfrak{m}$ as above, since $St^0\subset R[\mathfrak{m}t]\setminus P$, we have that 
\begin{align*}
R[\mathfrak{m}t]_P&=(R[\mathfrak{m}t]_{St^0})_{P_{S}}\\
&=R_{S}[\mathfrak{m}_St]_{P_S}\\
&=R_{\mathfrak{m}}[\mathfrak{m}_\mathfrak{m}t]_{P_{\mathfrak{m}}}.
\end{align*}
In particular, since $Y'=$ Proj$(R_{\mathfrak{m}}[\mathfrak{m}_\mathfrak{m}t])$ is assumed to smooth, we know that $R_{\mathfrak{m}}[\mathfrak{m}_\mathfrak{m}t]_{P_{\mathfrak{m}}}$ is a regular local ring.  Finally by the previous chain of equalities, we know that $R[\mathfrak{m}t]_P$ is a regular local ring.  Since $P$ was taken as arbitrary such that $\pi(P)=\mathfrak{m}$ we know that every point in the fiber over $\mathfrak{m}$ is a smooth point, and we conclude that $Y$ is smooth.  
\end{proof}

\subsection{Main theorem and it's proof}
We can now state and prove the main theorem.  After which, we will state some corollaries that rework the result of the main theorem into slightly more usable pieces.   

\newpage 

\begin{theorem}[Main theorem]\label{maintheorem}
Let $A=k[x_1,...,x_n]/J$ be a graded ring where $J$ is a proper homogeneous ideal of $k[x_1,...,x_n]$ such that $E=$ Proj$(A)$ is a smooth projective scheme over some field $k$ of characteristic zero. Let $R=k[x_1,...,x_n]/I$ such that $I_{min}\cong J$ and Spec$(R)$ has only one isolated singularity at the origin.  Then $R$ is a filtered deformation of $A$ along the ideal filtration given by the ideal $\mathfrak{m}=\langle x_1,...,x_n\rangle$, and for each $n<0$ and $i\geq0$, there is a long exact sequence of Hodge components of K-theory and cyclic homology given by
\[...\rightarrow HC_{n+1}^{(i)}(E)\rightarrow\tilde{K}^{(i+1)}_{n}(R)\rightarrow HC_{n}^{(i)}(\text{Proj}(R[\mathfrak{m}t])) \rightarrow HC_{n}^{(i)}(E)\rightarrow\tilde{K}^{(i+1)}_{n-1}(R)\rightarrow... .\]
\end{theorem}
\begin{proof}[Proof of the main theorem]This proof is essentially a combination of a couple of proofs that have come before.  Namely the proof of \ref{thcole}, and \ref{app1}.  

As has been referenced several times in the last section, we immediately know that $R$ is a filtered deformation of $A$.  This because by lemma \ref{associatedgraded}, $gr_{\mathfrak{m}}(R)=A$.    

For the remaining conclusions of the theorem, we will use our cdh-cohomology framework.  As we have done before, we need to show that each scheme, except for the bottom right corner, of the following abstract blowup square is a smooth scheme.  We have that, 
\[\begin{CD}
E=\text{Proj}(A)  @>>> \text{Proj}(R[\mathfrak{m}t])\\
@VVV @VVV\\
\{0\}@>>> \text{Spec}(R)
\end{CD}\]
is an abstract blowup square incorporating the data of the blowup of Spec$(R)$ at the origin. Immediately we see that $\{0\}$ is just a point, so it is smooth scheme.  Next, by the fact that $gr_{\mathfrak{m}}(R)=A$, and the assumption that $E$= Proj$(A)$ smooth, we know that, the top left, $E=$ Proj$(A)$= Proj$(gr_{\mathfrak{m}}(R))$ is smooth.  All that remains is Proj$(R[\mathfrak{m}t])$.  

By lemma \ref{globalblowup}, since Spec$(R)$ has only one isolated singularity at the origin, in order to show that this is in fact a resolution, we can look at the local data. Consider the blowup of $R_{\mathfrak{m}}$ at it's maximal ideal $\mathfrak{m}_{\mathfrak{m}}$.  By lemma \ref{localblowup}, since $R_{\mathfrak{m}}$ is a local ring, if the exceptional fiber of the blowup of Spec$(R_{\mathfrak{m}})$ at $\mathfrak{m}_{\mathfrak{m}}$ is smooth, we know that the blowup must also be smooth.  By theorem \ref{app1}, we know that the exceptional fiber of this local blowup is $E=$ Proj$(A)$, which is assumed to be smooth.  Therefore, (by \ref{localblowup}) we know that the local blowup, Proj$(R_{\mathfrak{m}}[\mathfrak{m}_{\mathfrak{m}}t])$ is also smooth.  Looking back now at the global situation, by lemma \ref{globalblowup}, since the local blowup is smooth, we know that the global blowup of Spec$(R)$ at the origin, Proj$(R[\mathfrak{m}t])$, must also be smooth. 

This shows that every scheme in the previous blowup square, except for Spec$(R)$, is smooth.  Using the smoothness of each of these schemes we have that for each $i\geq 0$ the Mayer-Vietoris sequence for cdh-fibrant cyclic homology, 
\[...\rightarrow\mathbb{H}_{cdh}^{-n}(\text{Spec}(R),\textbf{HC}^{(i)})\rightarrow\mathbb{H}_{cdh}^{-n}(\text{Proj}(R[\mathfrak{m}t]),\textbf{HC}^{(i)})\oplus\mathbb{H}_{cdh}^{-n}(\{0\},\textbf{HC}^{(i)})\hspace{1 in}\]\[\hspace{3in} \rightarrow\mathbb{H}_{cdh}^{-n}(E,\textbf{HC}^{(i)})\rightarrow\mathbb{H}_{cdh}^{-n+1}(\text{Spec}(R),\textbf{HC}^{(i)})\rightarrow... .\]
reduces to 
\[...\rightarrow HC_{n+1}^{(i)}(E)\rightarrow\mathbb{H}_{cdh}^{-n}(\text{Spec}(R), \textbf{HC}^{(i)})\rightarrow HC_{n}^{(i)}(\text{Proj}(R[\mathfrak{m}t]))\oplus HC_{n}^{(i)}(\{0\})\hspace{.8 in}\]\[\hspace{3.25 in} \rightarrow HC_{n}^{(i)}(E)\rightarrow\mathbb{H}_{cdh}^{-n+1}(\text{Spec}(R), \textbf{HC}^{(i)})\rightarrow... .\]

Under the assumption that $n<0$ we know two things. From \ref{cycliczero}, we know that $HC_{n}^{(i)}(\{0\})=0$, and from \ref{kHodgenegative} we know that \[\tilde{K}^{(i+1)}_{n}(R)=\tilde{K}^{(i+1)}_{n}(\text{Spec}(R))=\mathbb{H}^{-n}_{cdh}(X,\textbf{HC}^{(i)}).\]

So, putting this together, for $n<0$, we have the following exact sequence of K-groups and cyclic homology groups.  
\[...\rightarrow HC_{n+1}^{(i)}(E)\rightarrow\tilde{K}^{(i+1)}_{n}(R)\rightarrow HC_{n}^{(i)}(\text{Proj}(R[\mathfrak{m}t])) \rightarrow HC_{n}^{(i)}(E)\rightarrow\tilde{K}^{(i+1)}_{n-1}(R)\rightarrow... .\] which completes the proof.  
\end{proof}

While this is a nice result, as we stated, it is still not the easiest to work with.  Using our computation of cyclic homology for smooth schemes, we might be able to make things easier to compute.  First, we recall that cyclic homology for smooth schemes can be computed in terms of the hypercohomology of the truncated de Rham complex.  We have 
\[HC_n^{(i)}(S)=\mathbb{H}^{2i-n}(S,\Omega_{-/\mathbb{Q}}^{\leq i}),\] 
where $S$ is any smooth scheme.  So, according to the result of the main theorem, we can rewrite the long exact sequence in the following way.  

\begin{corollary}Under the assumptions of the main theorem, there is a long exact sequence of K-theory and hypercohomology groups given by the following.  
\[...\rightarrow \mathbb{H}^{2i-n-1}(E,\Omega_{-/\mathbb{Q}}^{\leq i})\rightarrow \tilde{K}_{n}^{(i+1)}(R)\rightarrow \mathbb{H}^{2i-n}(\text{Proj}(R[\mathfrak{m}t])),\Omega_{-/\mathbb{Q}}^{\leq i})\rightarrow\hspace{1.5 in}\]\[\hspace{1.5 in} \mathbb{H}^{2i-n-1}(E,\Omega_{-/\mathbb{Q}}^{\leq i})\rightarrow \tilde{K}_{n-1}^{(i+1)}(R)\rightarrow \mathbb{H}^{2i-n+1}(\text{Proj}(R[\mathfrak{m}t])),\Omega_{-/\mathbb{Q}}^{\leq i})\rightarrow ...\] 
Under the additional assumption that $k=Q$, we can rewrite the resulting long exact sequence with de Rham cohomology groups and hypercohomology groups in the following way. 
\[...\rightarrow F^{i-n-1}H_{DR}^{2i-n-1}(E/\mathbb{Q})\rightarrow \tilde{K}_{n}^{(i+1)}(R)\rightarrow \mathbb{H}^{2i-n}(\text{Proj}(R[\mathfrak{m}t])),\Omega_{-/\mathbb{Q}}^{\leq i})\rightarrow\hspace{1.5 in}\]\[\hspace{1.5 in}  F^{i-n}H_{DR}^{2i-n}(E/\mathbb{Q})\rightarrow \tilde{K}_{n-1}^{(i+1)}(R)\rightarrow \mathbb{H}^{2i-n+1}(\text{Proj}(R[\mathfrak{m}t])),\Omega_{-/\mathbb{Q}}^{\leq i})\rightarrow ...\] 
where $F^p$ denotes the $p$th level of the classical Hodge filtration on de Rham cohomology.  
\end{corollary}
\begin{proof}
The proof of the first part of this corollary was essentially before the statement of the corollary.  The result that we refer to is more precisely stated in \ref{cyclichyperHodge}.  

Assuming that $k=\mathbb{Q}$, in order to refine this further, we rely on the result of \ref{projderham}.  Since $E=$ Proj$(A)$ is a smooth projective variety over $\mathbb{Q}$, the Hodge components of it's cyclic homology coincides with the the classical Hodge filtration on de Rham cohomology.  Unfortunately we cannot refine the other piece any further.  This is because, while $\text{Proj}(R[\mathfrak{m}t])$ is a nice smooth projective scheme, it is not a variety over $k=\mathbb{Q}$.  

So we conclude, that under the additional assumption that $k=\mathbb{Q}$, that for for each $n<0$ and $i\geq0$, there is a long exact sequence of Hodge components of K-theory and de Rham cohomology given by the following,  
\[...\rightarrow F^{i-n-1}H_{DR}^{2i-n-1}(E/\mathbb{Q})\rightarrow \tilde{K}_{n}^{(i+1)}(R)\rightarrow \mathbb{H}^{2i-n}(\text{Proj}(R[\mathfrak{m}t])),\Omega_{-/\mathbb{Q}}^{\leq i})\rightarrow\hspace{1.5 in}\]\[\hspace{1.5 in}  F^{i-n}H_{DR}^{2i-n}(E/\mathbb{Q})\rightarrow \tilde{K}_{n-1}^{(i+1)}(R)\rightarrow \mathbb{H}^{2i-n+1}(\text{Proj}(R[\mathfrak{m}t])),\Omega_{-/\mathbb{Q}}^{\leq i})\rightarrow ...\]  
where $F^p$ denotes the $p$th level of the classical Hodge filtration on de Rham cohomology. 
\end{proof}

While the previous corollary makes things slightly more easy to compute, all of our results are still just long exact sequences.  The following corollary to the main theorem makes it possible for us show that these K-groups vanish for low enough degrees.   For the following we will denote the blowup $\text{Proj}(R[\mathfrak{m}t])$ by $Y$. 
  
\begin{corollary}
In the scenario of the main theorem, $\tilde{K}_n(R)$ is computable for $n<-d$, where $d$ is the dimension of Spec$(R)$.  For $n<-d$, we have that \[\tilde{K}_n(R)\cong 0,\] In addition to this we have that  \[
\tilde{K}_{-d}(R)\rightarrow H^{d}(Y,\mathcal{O}_Y)\] is surjective.  
\end{corollary}
\begin{proof}
The proof of this relies on the computation of cyclic homology in low degrees.   First, we know that if $d$ is the dimension of Spec$(R)$, then the blowup $Y$ is also dimension $d$.  Furthermore, the dimension of the exceptional fiber is smaller than $d$.  Since it is composed of hypersurfaces, it should be dimension $d-1$.  By proposition \ref{cyclicnegd}, we know that the cyclic homology vanishes for low degrees.  This tells us is that in the resulting long exact sequence from the main theorem (taken all together without the hodge decomposition), 
\[...\rightarrow HC_{n+1}(E)\rightarrow\tilde{K}_{n}(R)\rightarrow HC_{n}(Y) \rightarrow HC_{n}(E)\rightarrow\tilde{K}_{n-1}(R)\rightarrow...\]
for $n<-d$ each of the cyclic homology terms vanishes, meaning that the $\tilde{K}$ groups must also vanish.  Furthermore when $n=-d$, again from \ref{cyclicnegd}, we have that 

\[...\rightarrow H^{d-1}(E,\mathcal{O}_E)\rightarrow\tilde{K}_{-d}(R)\rightarrow H^{d}(Y,\mathcal{O}_Y) \rightarrow 0\rightarrow0 \rightarrow...\]

From exactness, we conclude that $\tilde{K}_{-d}(R)\rightarrow H^{d}(Y,\mathcal{O}_Y)$ is surjective.
\end{proof}
We should note that this corollary, while obtained using slightly different means, is similar to the main theorem of the paper \cite{cortinas2005cyclic}.  Obviously, their result is a much stronger statement since it is a statement about the complete K-theory groups, not just the fiber.  Since the two results agree, we could say that the previous corollary is a partial clarification of what happens on the fiber $\tilde{K}_n$.   

So, in summary, this last result gives explicit computations in low degrees, namely that the $\tilde{K}_{n}(R)$ are all 0.  Furthermore, the fact that $\tilde{K}_{-d}(R)\rightarrow H^{d}(Y,\mathcal{O}_Y)$ is surjective gives us a condition for the nonvanishing of at least one $\tilde{K}_n$.  

This is important for a couple of reasons.  First, if our ultimate goal is to be able to detect and classify singularities, having a condition that forces $\tilde{K}_{-d}(R)$ to be nontrivial will show that the affine scheme Spec$(R)$ is definitely not smooth.   Of course there are simpler ways of determining if a scheme has singularities, but it is possible that this criteria will become more useful in other ways.  

Second, and finally, if our ultimate goal is to compute the $K_{n}(R)$ (notice the missing $\sim$), this is one more piece of the puzzle.  Furthermore, under additional assumptions (like assuming our deformations are graded rings under non-standard gradings and quasi-homogeneous coordinates systems) it may be possible to use something like the ``standard trick" from the paper \cite{cortinas2009k} to compute $K_{n}(R)$ directly from $\tilde{K}_{n}(R)$ without knowledge of $KH_n(R)$.  

Along with attempting to refine our cyclic homology computations from the main theorem, this would of course be the main goal of any further research in computing the K-theory of filtered deformations of graded polynomial algebras.

%
%
%
%
%

\newpage

\appendix
\section{Diagrams for the the extensions of the theorem of Michler}
\subsection{Variations of cyclic homology}\label{cmich}
\begin{theorem}
Let $k$ be an algebraically closed field of characteristic zero, and let $X$ be a reduced affine hypersurface over $k$ with only isolated singularities in affine $N$-space.  Let $Y$, $Z$, and $E$ be the resolution, center, and exceptional fiber of the resolution respectively (all assumed to be smooth).  Then the Hodge components of negative cyclic homology can be computed using the following combination of two long exact sequences.
\end{theorem}

For $n>N$ and $2i-n=N-1$,  
 
\noindent\begin{tikzpicture}[scale=.78, every node/.style={scale=.78} ]
\matrix (m) [matrix of math nodes, row sep=3em,
column sep=1em, text height=2ex, text depth=0.2ex]
{ &     & & & H_{DR}^{2i-n+1}(E/k) &\\
  &     & & & H_{DR}^{2i-n+1}(Y/k)\oplus H_{DR}^{2i-n+1}(Z/k) &\\
...&   HP_{n+2}^{(i+1)}(X/k) & T(\Omega_{X/k}^{N-1})\oplus H_{DR}^{N-1}(X/k) & HN_{n+1}^{(i+1)}(X/k) & HP_{n+1}^{(i+1)}(X/k)& ...\\
&   H_{DR}^{2i-n}(Y/k)\oplus H_{DR}^{2i-n}(Z/k)  & & & &  \\ 
&   H_{DR}^{2i-n}(E/k)  & & & &... \\ };
\path[->]
(m-3-1) edge (m-3-2);
\path[->]
(m-3-2) edge (m-3-3);
\path[->]
(m-3-3) edge (m-3-4);
\path[->]
(m-3-4) edge (m-3-5);
\path[->]
(m-3-5) edge (m-3-6);
\path[<-][densely dashed]
(m-5-2) edge (m-4-2);
\path[<-][densely dashed]
(m-4-2) edge (m-3-2);
\path[<-][densely dashed]
(m-1-5) edge (m-2-5);
\path[<-][densely dashed]
(m-2-5) edge (m-3-5);
\path[<-][densely dashed]
(m-3-1) edge [bend right=30] (m-5-2);
\path[<-][densely dashed]
(m-3-2) edge [bend left=20] (m-1-5);
\path[<-][densely dashed]
(m-3-5) edge [bend right=20] (m-5-6);
\end{tikzpicture}

\vspace{5 mm}

For $n>N$ and $2i-n\neq N-1$,

\noindent\begin{tikzpicture}[scale=.87, every node/.style={scale=.87} ]
\matrix (m) [matrix of math nodes, row sep=3em,
column sep=1em, text height=2ex, text depth=0.2ex]
{ &     & & & H_{DR}^{2i-n+1}(E/k) &\\
  &     & & & H_{DR}^{2i-n+1}(Y/k)\oplus H_{DR}^{2i-n+1}(Z/k) &\\
...&   HP_{n+2}^{(i+1)}(X/k) & H_{DR}^{2i-n}(X/k) & HN_{n+1}^{(i+1)}(X/k) & HP_{n+1}^{(i+1)}(X/k)& ...\\
&   H_{DR}^{2i-n}(Y/k)\oplus H_{DR}^{2i-n}(Z/k)  & & & &  \\ 
&   H_{DR}^{2i-n}(E/k)  & & & &... \\ };
\path[->]
(m-3-1) edge (m-3-2);
\path[->]
(m-3-2) edge (m-3-3);
\path[->]
(m-3-3) edge (m-3-4);
\path[->]
(m-3-4) edge (m-3-5);
\path[->]
(m-3-5) edge (m-3-6);
\path[<-][densely dashed]
(m-5-2) edge (m-4-2);
\path[<-][densely dashed]
(m-4-2) edge (m-3-2);
\path[<-][densely dashed]
(m-1-5) edge (m-2-5);
\path[<-][densely dashed]
(m-2-5) edge (m-3-5);
\path[<-][densely dashed]
(m-3-1) edge [bend right=30] (m-5-2);
\path[<-][densely dashed]
(m-3-2) edge [bend left=20] (m-1-5);
\path[<-][densely dashed]
(m-3-5) edge [bend right=20] (m-5-6);
\end{tikzpicture}

\newpage

For $n\leq N$ and $i=n$,

\noindent\begin{tikzpicture}[scale=.87, every node/.style={scale=.87} ]
\matrix (m) [matrix of math nodes, row sep=3em,
column sep=1em, text height=2ex, text depth=0.2ex]
{ &     & & & H_{DR}^{2i-n+1}(E/k) &\\
  &     & & & H_{DR}^{2i-n+1}(Y/k)\oplus H_{DR}^{2i-n+1}(Z/k) &\\
... & HP_{n+2}^{(i+1)}(X/k) &\Omega_{X/k}^{n}/d\Omega_{X/k}^{n-1} & HN_{n+1}^{(i+1)}(X/k) & HP_{n+1}^{(i+1)}(X/k)& ...\\
&   H_{DR}^{2i-n}(Y/k)\oplus H_{DR}^{2i-n}(Z/k)  & & & &  \\ 
&   H_{DR}^{2i-n}(E/k)  & & & &... \\ };
\path[->]
(m-3-1) edge (m-3-2);
\path[->]
(m-3-2) edge (m-3-3);
\path[->]
(m-3-3) edge (m-3-4);
\path[->]
(m-3-4) edge (m-3-5);
\path[->]
(m-3-5) edge (m-3-6);
\path[<-][densely dashed]
(m-5-2) edge (m-4-2);
\path[<-][densely dashed]
(m-4-2) edge (m-3-2);
\path[<-][densely dashed]
(m-1-5) edge (m-2-5);
\path[<-][densely dashed]
(m-2-5) edge (m-3-5);
\path[<-][densely dashed]
(m-3-1) edge [bend right=30] (m-5-2);
\path[<-][densely dashed]
(m-3-2) edge [bend left=20] (m-1-5);
\path[<-][densely dashed]
(m-3-5) edge [bend right=20] (m-5-6);
\end{tikzpicture}

\vspace{5 mm}

For $n\leq N$ and $n/2\leq i<n$,

\noindent\begin{tikzpicture}[scale=.87, every node/.style={scale=.87} ]
\matrix (m) [matrix of math nodes, row sep=3em,
column sep=1em, text height=2ex, text depth=0.2ex]
{ &     & & & H_{DR}^{2i-n+1}(E/k) &\\
  &     & & & H_{DR}^{2i-n+1}(Y/k)\oplus H_{DR}^{2i-n+1}(Z/k) &\\
... & HP_{n+2}^{(i+1)}(X/k) & H_{DR}^{2i-n}(X/k) & HN_{n+1}^{(i+1)}(X/k) & HP_{n+1}^{(i+1)}(X/k)& ...\\
&   H_{DR}^{2i-n}(Y/k)\oplus H_{DR}^{2i-n}(Z/k)  & & & &  \\ 
&   H_{DR}^{2i-n}(E/k)  & & & &... \\ };
\path[->]
(m-3-1) edge (m-3-2);
\path[->]
(m-3-2) edge (m-3-3);
\path[->]
(m-3-3) edge (m-3-4);
\path[->]
(m-3-4) edge (m-3-5);
\path[->]
(m-3-5) edge (m-3-6);
\path[<-][densely dashed]
(m-5-2) edge (m-4-2);
\path[<-][densely dashed]
(m-4-2) edge (m-3-2);
\path[<-][densely dashed]
(m-1-5) edge (m-2-5);
\path[<-][densely dashed]
(m-2-5) edge (m-3-5);
\path[<-][densely dashed]
(m-3-1) edge [bend right=30] (m-5-2);
\path[<-][densely dashed]
(m-3-2) edge [bend left=20] (m-1-5);
\path[<-][densely dashed]
(m-3-5) edge [bend right=20] (m-5-6);
\end{tikzpicture}

\vspace{5 mm}

For $n\leq N$ and $i<n/2$, since the Hodge component of cyclic homology is 0, we reduce to the following.  
\[
 \rightarrow H_{DR}^{2i-n-1}(E/k)\rightarrow HN^{(i)}_{n}(X/k) \rightarrow H_{DR}^{2i-n}(Y/k)\oplus H_{DR}^{2i-n}(Z/k) \rightarrow H_{DR}^{2i-n}(E/k) \rightarrow HN^{(i)}_{n-1}(X/k)\rightarrow  
\]

\newpage
\nocite{*}
\bibliographystyle{plain}
\bibliography{thesis}

\begin{thebibliography}{10}

\bibitem{aguilar2002algebraic}
Marcelo Aguilar, Samuel Gitler, and Carlos Prieto.
\newblock {\em Algebraic topology from a homotopical viewpoint}.
\newblock Springer Verlag, 2002.

\bibitem{atiyah1969introduction}
Michael~Francis Atiyah and Ian~Grant Macdonald.
\newblock {\em Introduction to commutative algebra}, volume~2.
\newblock Addison-Wesley Reading, 1969.

\bibitem{beltrametti2009lectures}
Mauro~C Beltrametti.
\newblock {\em Lectures on curves, surfaces and projective varieties: a
  classical view of algebraic geometry}.
\newblock European Mathematical Society, 2009.

\bibitem{connes2012cyclic}
Alain Connes and Caterina Consani.
\newblock Cyclic homology, {S}erre's local factors and the
  $\lambda$-operations.
\newblock {\em arXiv preprint arXiv:1211.4239}, 2012.

\bibitem{cortinas2006obstruction}
Guillermo Corti{\~n}as.
\newblock The obstruction to excision in {K}-theory and in cyclic homology.
\newblock {\em Inventiones mathematicae}, 164(1):143--173, 2006.

\bibitem{cortinas2005cyclic}
Guillermo Corti{\~n}as, Christian Haesemeyer, Marco Schlichting, and Charles~A
  Weibel.
\newblock Cyclic homology, cdh-cohomology and negative {K}-theory.
\newblock {\em arXiv preprint math/0502255}, 2005.

\bibitem{cortinas2009k}
Guillermo Corti{\~n}as, Christian Haesemeyer, Mark~E Walker, and Charles~A
  Weibel.
\newblock {K}-theory of cones of smooth varieties.
\newblock {\em arXiv preprint arXiv:0905.4642}, 2009.

\bibitem{cortinas2008k}
Guillermo Corti{\~n}as, Christian Haesemeyer, and C~Weibel.
\newblock K-regularity, cdh-fibrant {H}ochschild homology, and a conjecture of
  {V}orst.
\newblock {\em Journal of the American Mathematical Society}, 21(2):547--561,
  2008.

\bibitem{cortinas2009infinitesimal}
Guillermo Corti{\~n}as, Christian Haesemeyer, and Charles~A Weibel.
\newblock Infinitesimal cohomology and the {C}hern character to negative cyclic
  homology.
\newblock {\em Mathematische Annalen}, 344(4):891--922, 2009.

\bibitem{cortinas2009relative}
Guillermo Corti{\~n}as and Charles Weibel.
\newblock Relative chern characters for nilpotent ideals.
\newblock In {\em Algebraic Topology}, pages 61--82. Springer, 2009.

\bibitem{cutkosky2004resolution}
Steven~Dale Cutkosky.
\newblock {\em Resolution of singularities}, volume~63.
\newblock Cambridge Univ Press.

\bibitem{deligne1971theorie}
Pierre Deligne.
\newblock Th{\'e}orie de {H}odge, {II} et {III}.
\newblock {\em Publications Math{\'e}matiques de l'IHES}, 40(1):5--57, 1971.

\bibitem{eisenbud1995commutative}
David Eisenbud.
\newblock {\em Commutative {A}lgebra: with a view toward algebraic geometry},
  volume 150.
\newblock Springer, 1995.

\bibitem{eisenbud2000geometry}
David Eisenbud and Joe Harris.
\newblock {\em The geometry of schemes}, volume 197.
\newblock Springer, 2000.

\bibitem{emmanouil1995cyclic}
Ioannis Emmanouil.
\newblock The cyclic homology of affine algebras.
\newblock {\em Inventiones mathematicae}, 121(1):1--19, 1995.

\bibitem{feigin1985additive}
Boris~L Feigin and Boris~L Tsygan.
\newblock Additive {K}-theory and crystalline cohomology.
\newblock {\em Functional Analysis and its Applications}, 19(2):124--132, 1985.

\bibitem{fox1993introduction}
Thomas~F Fox.
\newblock An introduction to algebraic deformation theory.
\newblock {\em Journal of pure and applied algebra}, 84(1):17--41, 1993.

\bibitem{geller1989cyclic}
S~Geller, L~Reid, and C~Weibel.
\newblock The cyclic homology and {K}-theory of curves.
\newblock {\em J. reine angew. Math}, 393:39--90, 1989.

\bibitem{geller1994Hodge}
Susan~C Geller and Charles~A Weibel.
\newblock Hodge decompositions of {L}oday symbols in {K}-theory and cyclic
  homology.
\newblock {\em K-theory}, 8(6):587--632, 1994.

\bibitem{gerstenhaber1964deformation}
Murray Gerstenhaber.
\newblock On the deformation of rings and algebras.
\newblock {\em The Annals of Mathematics}, 79(1):59--103, 1964.

\bibitem{gerstenhaber1966deformation2}
Murray Gerstenhaber.
\newblock On the deformation of rings and algebras: {II}.
\newblock {\em The Annals of Mathematics}, 84(1):1--19, 1966.

\bibitem{gerstenhaber1987Hodge}
Murray Gerstenhaber and Samuel~D Schack.
\newblock A {H}odge-type decomposition for commutative algebra cohomology.
\newblock {\em Journal of Pure and Applied Algebra}, 48(1):229--247, 1987.

\bibitem{giorgi2006irreducible}
Erika Giorgi.
\newblock On the irreducible components of the form ring and an application to
  intersection cycles.
\newblock {\em Communications in Algebra{\textregistered}}, 34(8):2755--2767,
  2006.

\bibitem{goodwillie1986relative}
Thomas~G Goodwillie.
\newblock Relative algebraic {K}-theory and cyclic homology.
\newblock {\em The Annals of Mathematics}, 124(2):347--402, 1986.

\bibitem{grothendieck1966rham}
Alexander Grothendieck.
\newblock On the de {R}ham cohomology of algebraic varieties.
\newblock {\em Publications Math{\'e}matiques de l'Institut des Hautes
  {\'E}tudes Scientifiques}, 29(1):95--103, 1966.

\bibitem{haesemeyer2004descent}
Christian Haesemeyer.
\newblock Descent properties of homotopy {K}-theory.
\newblock {\em Duke Mathematical Journal}, 125(3):589--619, 2004.

\bibitem{hartshorne1975rham}
Robin Hartshorne.
\newblock On the de {R}ham cohomology of algebraic varieties.
\newblock {\em Publications Math{\'e}matiques de l'IH{\'E}S}, 45(1):6--99,
  1975.

\bibitem{hartshorne1977algebraic}
Robin Hartshorne.
\newblock {\em Algebraic geometry}, volume~52.
\newblock springer Verlag, 1977.

\bibitem{hauser2003hironaka}
Herwig Hauser.
\newblock The {H}ironaka theorem on resolution of singularities.
\newblock {\em Bull. Amer. Math. Soc}, 40(3):323--403, 2003.

\bibitem{hauserblowupsresolution}
Herwig Hauser.
\newblock Blowups and resolution.
\newblock
  http://homepage.univie.ac.at/herwig.hauser/Publications/blowups-and-resolution-apr-2013.pdf,
  2013.

\bibitem{herrmann1988equimultiplicity}
Manfred Herrmann, Shin Ikeda, Ulrich Orbanz, and Boudewijin Moonen.
\newblock {\em Equimultiplicity and blowing up: an algebraic study}.
\newblock Springer-Verlag Berlin, 1988.

\bibitem{hironaka1964resolution}
Heisuke Hironaka.
\newblock Resolution of singularities of an algebraic variety over a field of
  characteristic zero: {II}.
\newblock {\em The Annals of Mathematics}, 79(2):205--326, 1964.

\bibitem{kollar2009lectures}
J{\'a}nos Koll{\'a}r.
\newblock {\em Lectures on {R}esolution of {S}ingularities (AM-166)}.
\newblock Number 166. Princeton University Press, 2009.

\bibitem{loday1997cyclic}
Jean-Louis Loday.
\newblock {\em Cyclic homology}, volume 301.
\newblock Springer, 1997.

\bibitem{mclane1994sheaves}
S~McLane and I~Moerdijk.
\newblock {\em Sheaves in geometry and logic: {A} first introduction to topos
  theory}.
\newblock Springer, 1994.

\bibitem{michler1997cyclic}
Ruth~I Michler.
\newblock Cyclic homology of affine hypersurfaces with isolated singularities.
\newblock {\em Journal of Pure and Applied Algebra}, 120(3):291--299, 1997.

\bibitem{o1997smoothness}
L~O'Carroll and G~Valla.
\newblock On the smoothness of blow ups.
\newblock {\em Communications in Algebra}, 25(6):1861--1872, 1997.

\bibitem{thomason1985algebraic}
Robert~W Thomason.
\newblock Algebraic {K}-theory and {\'e}tale cohomology.
\newblock In {\em Annales scientifiques de l'{\'E}cole Normale Sup{\'e}rieure},
  volume~18, pages 437--552. Soci{\'e}t{\'e} math{\'e}matique de France, 1985.

\bibitem{thomason2007higher}
Robert~W Thomason and Thomas Trobaugh.
\newblock Higher algebraic {K}-theory of schemes and of derived categories.
\newblock In {\em The Grothendieck Festschrift Volume III}, pages 247--435.
  Springer, 2007.

\bibitem{voevodsky2000homotopy}
Vladimir Voevodsky.
\newblock Homotopy theory of simplicial sheaves in completely decomposable
  topologies, 2000.

\bibitem{voevodsky2008unstable}
Vladimir Voevodsky.
\newblock Unstable motivic homotopy categories in {N}isnevich and
  cdh-topologies.
\newblock {\em arXiv preprint arXiv:0805.4576}, 2008.

\bibitem{weibel1996cyclic}
Charles Weibel.
\newblock Cyclic homology for schemes.
\newblock {\em Proceedings of the American Mathematical Society},
  124(6):1655--1662, 1996.

\bibitem{weibel1997Hodge}
Charles Weibel.
\newblock The {H}odge filtration and cyclic homology.
\newblock {\em K-theory}, 12(2):145--164, 1997.

\bibitem{weibel2013k}
Charles Weibel.
\newblock The {K}-book: {A}n introduction to algebraic {K}-theory.
\newblock {\em Book-in-progress, available at its author's webpage: http://www.
  math. rutgers. edu/weibel/Kbook. html}, 2013.

\bibitem{weibel1995introduction}
Charles~A Weibel.
\newblock {\em An introduction to homological algebra}, volume~38.
\newblock Cambridge university press, 1995.

\bibitem{weibel1991etale}
Charles~A Weibel and Susan~C Geller.
\newblock Etale descent for {H}ochschild and cyclic homology.
\newblock {\em Commentarii Mathematici Helvetici}, 66(1):368--388, 1991.

\end{thebibliography}

\end{document}